\documentclass[12pt, a4paper, parskip=half, abstracton]{scrartcl}

\usepackage{array}
\usepackage{marginnote}
\usepackage{xcolor}
\usepackage{amscd,amssymb,amsfonts,amsmath,latexsym,amsthm}
\usepackage{hyperref}
\usepackage[all,cmtip]{xy}
\usepackage{mathrsfs}
\usepackage{graphicx}
\usepackage{bm} 
\usepackage[nameinlink, capitalise]{cleveref}

\usepackage{etoolbox}

\def\hB{\hspace*{\fill}$\qed$}

\usepackage[nottoc]{tocbibind} 

\usepackage{defs_pp1}
\usepackage{slashed}
\usepackage[utf8]{inputenc}
\usepackage{microtype}
\usepackage[english]{babel}
\usepackage{mathtools}
\usepackage{bm}
\usepackage{esvect}

\title{$E$-theory  of $X$-$C^{*}$-algebras and functor formalisms}
\author{
Ulrich Bunke\thanks{Fakult{\"a}t f{\"u}r Mathematik,
Universit{\"a}t Regensburg,
93040 Regensburg,
ulrich.bunke@mathematik.uni-regensburg.de} 
}

\numberwithin{equation}{section}
\newtheorem{theorem}{Theorem}[section] 
\newtheorem{prop}[theorem]{Proposition}
\newtheorem{lem}[theorem]{Lemma}
\newtheorem{constr}[theorem]{Construction}

\newtheorem{ddd}[theorem]{Definition}
\newtheorem{kor}[theorem]{Corollary}

\newtheorem{fact}[theorem]{Fact}
\theoremstyle{remark}
\theoremstyle{definition}

\newtheorem{ex}[theorem]{Example}
\newtheorem{rem}[theorem]{Remark}
 
\newcommand{\papr}{\mathrm{pprfr}}

\newcommand{\rmB}{\mathrm{B}}
\newcommand{\pt}{\mathrm{pt}}
\newcommand{\Ani}{\mathbf{Ani}}

\newcommand{\CH}{\mathbf{CH}}

\newcommand{\CK}{c\mathcal{K}}

\newcommand{\CoSh}{\mathrm{CoShv}}

\newcommand{\EC}{\mathcal{EC}}
\newcommand{\RelCat}{\mathbf{RelCat}}
\newcommand{\nucl}{\mathrm{nucl}}

\newcommand{\alg}{\mathrm{alg}}
\newcommand{\CoShv}{\mathrm{CoShv}}
\newcommand{\Shv}{\mathrm{Shv}}

\newcommand{\ee}{\mathrm{e}}

\newcommand{\All}{ \mathcal{A}\mathrm{ll}}

\newcommand{\EE}{\mathrm{E}}
\newcommand{\sepa}{\mathrm{sep}}
\newcommand{\LCH}{\mathbf{LCH}}

\newcommand{\nCalg}{C^{*}\mathbf{Alg}^{\mathrm{nu}}}

\newcommand{\ho}{\mathrm{ho}}

\newcommand{\cN}{\mathcal{N}}

\newcommand{\gd}{\mathrm{gd}}

\newcommand{\Cofib}{\mathrm{Cofib}}

\newcommand{\Fib}{{\mathrm{Fib}}}

\newcommand{\incl}{\mathrm{incl}}

\newcommand{\coker}{\mathrm{coker}}

\newcommand{\cP}{\mathcal{P}}

\newcommand{\cK}{\mathcal{K}}

\newcommand{\CAlg}{{\mathbf{CAlg}}}

\newcommand{\cW}{{\mathcal{W}}}

\newcommand{\cD}{{\mathcal{D}}}

 \newcommand{\Cat}{{\mathbf{Cat}}}

\newcommand{\Open}{{\mathbf{Open}}}

\newcommand{\reg}{\mathrm{reg}}
\newcommand{\Prim}{\mathrm{Prim}}

\newcommand{\Calg}{{\mathbf{C}^{\ast}\mathbf{Alg}}}

\renewcommand{\Pr}{\mathbf{Pr}}

\newcommand{\op}{\mathrm{op}}

\newcommand{\exa}{\mathrm{ex}}

\renewcommand{\LCH}{\mathrm{LCH}}

 \newcommand{\uc}{\mathrm{le}}
  \newcommand{\rmc}{\mathrm{c}}
\newcommand{\Corr}{\mathrm{Corr}}

\newcommand{\Locale}{\mathbf{Locale}}
\newcommand{\Frame}{\mathbf{Frame}}
 \newcommand{\Poset}{\mathbf{Poset}}

\newcommand{\st}{\mathrm{st}}
\renewcommand{\Open}{\mathrm{Open}}

 \newcommand{\homol}{\mathrm{homol}}

\newcommand{\pac}{\mathrm{pac}}
\newcommand{\ac}{\mathrm{ac}}

\newcommand{\pfr}{\mathrm{pfr}}
\newcommand{\prfr}{\mathrm{prfr}}
\newcommand{\afr}{\mathrm{afr}}
\newcommand{\pafr}{\mathrm{pafr}}

\newcommand{\puc}{\mathrm{ple}}

\begin{document}  \maketitle 
  
  \begin{abstract}  
  We show that  $E$-theory for locally compact Hausdorff spaces 	
  constitutes a six-functor formalism   which is equivalent  
to the six-functor formalism of $\EE$-valued sheaves.  
We furthermore show that the $E$-theory category for locales that can be written as  unions of finite open sublocales is equivalent to the category of $\EE$-valued cosheaves. \end{abstract} 
  \tableofcontents
  
  \section{Introduction}
  
To every locale $X$ one can  functorially associate   the $E$-theory category $E(X)$. It is a presentable stable $\infty$-category defined as the target of the universal
 homological functor on the category of $X$-$C^{*}$-algebras -- $C^{*}$-algebras equipped with a morphism of locales
 $I(A)\to X$, where $I(A)$ is the locale of closed $*$-ideals of $A$. Our first main result describes the category $E(X)$ 
 for "sufficiently finite" locales in homotopy-theoretic terms:
    \begin{theorem}[\cref{hgitgjetrgetogerpigt}] 
For a locale $X$  that can be written as  a  union of finite open sublocales, there is   a natural equivalence
$$E(X)\simeq \CoShv(X,\EE)\ .$$  
\end{theorem}
Turning to the topological setting, our second result extends this homotopy-theoretic perspective to locally compact Hausdorff spaces:\begin{theorem}[\cref{zzopjrtzjrztjrth}]
The assignment $X\mapsto E(X)$ is part of a six-functor formalism
on $\LCH$ that is equivalent to the $\EE$-valued sheaf-theoretic six-functor formalism
$X\mapsto \Shv(X,\EE) $, where    $\EE:=E(*)$.
 \end{theorem}
We will give more precise statements below after 
 providing a historical introduction to the field.

Non-commutative homotopy theory in the sense of the present paper  is the study of $C^{*}$-algebras up to homotopy invariance and $K$-stability.
A fundamental  invariant in  this field is the topological  $K$-theory of $C^{*}$-algebras 
which  plays a central role in the classification of $C^{*}$-algebras, index theory, algebraic topology, mathematical physics and other fields.  By   axiomatizing core properties of $K$-theory through  the notion of a homological functor, one is led to study universal homological functors, specifically    $KK$-theory or $E$-theory. 
These concepts extend beyond plain $C^{*}$-algebras to   categories of $C^{*}$-algebras with additional structures like actions by groups,   groupoids, tensor categories, as well as   parametrization over a space.
The primary objective  of the present paper is to describe the $E$-theory category $E(X)$ of the category $X\nCalg$ of $X$-$C^{*}$ algebras -- that is, $C^{*}$-algebras parametrized by a topological space or, more generally, a locale $X$.

  $KK$-theory  was introduced   in  \cite{kasparovinvent} as a bivariant functor $(A,B)\mapsto KK(A,B)$ on $C^{*}$-algebras sending
$A,B$ to the group of equivalence classes of Kasparov modules (see \cite{blackadar} for a comprehensive account). Through the  Kasparov product, 
  this  bivariant  functor can be interpreted  as a functor from $C^{*}$-algebras to an additive category.
The $E$-theory functor was first  constructed in \cite{higson} as the  universal homotopy invariant, $K$-stable 
and exact functor from the category of $C^{*}$-algebras to an additive category.
 The fundamental distinction  between $KK$-and $E$-theory lies in their respective exactness properties: while $KK$-theory is exact for
 exact sequences of $C^{*}$-algebras admitting a completely positive contractive split, the 
 $E$-theory functor  is exact for all exact sequences. 
For  separable $C^{*}$-algebras, $E$-theory was subsequently reformulated 
   in \cite{zbMATH04182148} using  homotopy classes of asymptotic morphisms.
This framework was later extended to  non-separable and equivariant settings  in \cite{Guentner_2000}, although in those cases exactness  remains tied to separability assumptions.

For a  locally compact Hausdorff space  $X$, the additive $E$-theory category   $E(X)$ for   $C_{0}(X)$-algebras was introduced in \cite{Park_2000} for the separable case, and subsequently extended   in \cite{Popescu_2004}, also
 to include the groupoid-equivariant setting. 
These  works further developed the maximal balanced tensor product $\otimes_{X}$ of $C_{0}(X)$-algebras, the induced tensor  bifunctor on the $E$-theory category, and the contravariant functoriality for maps  between   locally compact Hausdorff spaces. Note that $KK$-theory for $C_{0}(X)$-algebras was already  introduced in the 
  in the seminal paper \cite{kasparovinvent}.

 An alternative characterization of $C_{0}(X)$-algebras, which does  not explicitly  involve  the algebra $C_{0}(X)$, was  provided  in  
  \cite{Dadarlat_2012}. This new definition can be interpreted for arbitrary  topological spaces  --    
or even locales --  giving rise to the   notion of  $X$-$C^{*}$-algebras.  The paper   \cite{Dadarlat_2012}  extends the construction of the additive $E$-theory category
$E(X)$  for separable $X$-$C^{*}$-algebras from 
second countable locally compact Hausdorff  spaces to  general second countable
topological spaces, using the framework of asymptotic morphisms. 
 This generalization facilitates the  study of the $E$-theory category $E(X)$ for finite topological spaces, 
 which is more amenable    to explicit homotopy-theoretic calculations.

%
%
%
%

It was observed in \cite{MR2193334}   that  universal homological
functors on $C^{*}$-algebras naturally take values in triangulated categories, a fact that  also applies to the  $E$-theory category $E(X)$.
In \cite{Dadarlat_2012}   the dependence of $E(X)$ on the space $X$ was further investigated, revealing  
  adjunctions for open and closed embeddings. This work also initiated the explicit  calculation of the mapping groups
 of  $E(X)$. A key result is
  \cite[Thm. 3.2]{Dadarlat_2012} which expresses   the  mapping  groups of $E(X)$ in terms of the   mapping groups   in $E(Y_{i})$ for a family 
  of finite quotients $(X\to Y_{i})_{i}$.  Consequently 
   the collection of evaluation functors  $$(\ev_{U}:  E(X)\to \Ab)_{U\in \Open(X)}$$ is shown to be jointly conservative
  in  \cite[Thm. 3.10]{Dadarlat_2012}.  This result  further   motivates  the  problem of describing the   mapping  groups 
  of $E(Y)$ for finite topological spaces $Y$.  While initial results were obtained in \cite{zbMATH06347396}, \cite{zbMATH06404312} (see also \cite{filtrated} where the same principles were applied in the case of $KK$-theory), 
  complete calculations remain restricted to   spaces with   very  few points -- typically fewer than  three -- 
 or to posets   $\Open(Y)$ with a very simple structure, such as   $[n]$. From the  perspective of the present paper the primary  reason for these  limitations is that  working solely with   triangulated categories neglects crucial
 higher structural information.

 Most   known triangulated categories   arise as the homotopy categories of stable $\infty$-categories, a 
  principle  that also applies to $KK$- and $E$-theory.
Building on developments for $KK$-theory in     \cite{LN}, \cite{KKG} in \cite{Bunke:2023aa}
the $E$-theory  for $C^{*}$-algebras was redefined as the universal
functor to a cocomplete stable $\infty$-category which is  homotopy invariant,
$K$-stable, exact and s-finitary. This universal functor can be constructed by iteratively forcing these universal properties.   

An alternative characterization and construction of
$E$-theory as a universal homotopy invariant, $K$-stable, exact and filtered colimit-preserving functor 
 \begin{equation}\label{vasdcasdcadscasca}  \ee:\nCalg\to \EE\end{equation} from $C^{*}$-algebras  to a cocomplete stable $\infty$-category 
was given in     \cite{budu}  (including  the equivariant case for discrete groups).
 In \cite{Bunke:2023aa},  it was  further verified -- by comparing the universal property of its
homotopy category as an additive category with the universal property of the $E$-theory from     \cite[Thm. 3.6]{higson} -- that the composition $$\ho\circ \ee:\nCalg\to \EE\to \ho(\EE) $$
is equivalent to the classical $E$-theory functor.

The  reason that   the aforementioned references  at certain points restrict to separable $C^{*}$-algebras,  and in consequence  to second countable topological spaces, 
lies in the details of the construction of the composition of asymptotic morphisms following \cite{zbMATH04182148},
and the  Connes-Higson construction  of asymptotic morphisms  from   exact sequences.

  A functor which is   homotopy invariant, $K$-stable, and preserves both exact sequences  and filtered colimits will be called homological (see  \cref{kogwegergfw}).   
In order to axiomatize this,    we introduce   in \cref{lphertgerthe} the concept of an $E$-theory context as a category  for which we can formulate the conditions of a homological functor. We further define morphisms
between $E$-theory contexts such  that precompositon with such a morphism preserves homological functors.
This gives rise to the category $\EC$ of $E$-theory contexts. 

In  \cref{koprtheggrtge}, following the general path  for the construction of universal homological functors
as in \cite{zbMATH01596160},   \cite{zbMATH05771937}, \cite{MR3070515}, we     show that there  exists a functor
\begin{equation}\label{gerwgewurgowergwergwerfe}  \EE:\EC\to \Pr_{\st}^{L}\end{equation}
  from the category of $E$-theory contexts  $\EC$ to the category $\Pr^{L}_{\st}$ of presentable stable $\infty$-categories
  and left-adjoint functors. This functor is equipped 
 with a natural transformation $\ee:\id\to \EE$ (where $\EE$ is regarded as $\Cat_{\infty}$-valued),
such that, for every  $E$-theory context $\cA$ in $\EC$, 
the functor $$\ee_{\cA}:\cA\to \EE(\cA)$$ is the universal homological functor for $\cA$.
In \cref{kohpehtregetrg} we further extend this construction to incorporate lax symmetric 
monoidal structures.

Note that our construction of the $E$-theory functor is designed to yield the universal homological functor.  It does not involve any explicit description of the mapping spaces  in terms of asymptotic morphisms or similar constructions. 
However, as explained in \cite[Sec. 14]{Bunke:2023aa} one can use asymptotic morphisms to represent points in the mapping spaces.  In cases where a classical construction exists  -- for example, if $\cA=X\nCalg$ for a topological space $X$ --
a comparison of the mapping  groups of $\ho(\EE(\cA))$ with the mapping groups of the  classical triangulated category 
could in principle be achieved by comparing the universal properties of the categories. While we expect that    the $E$-theory groups constructed in the present paper coincide  with the classical   $E$-theory groups in the cases where both are defined,   we will not  provide a formal proof. One reason is the lack of formulated sufficiently general universal properties of the classical constructions  -- analogous   to  \cite[Thm. 3.6]{higson} --
  which could be compared with the ones of our construction.

%
%
%
%

  In  this paper we consider $E$-theory contexts $\cA$ consisting  of $C^{*}$-algebras  with  additional structure,  and where morphisms
   are structure-preserving homomorphisms between $C^{*}$-algebras. We focus  particularly  on 
   systems of ideals parametrized by a poset $P$ admiting a maximal element, leading to the $E$-theory context   of $P$-$C^{*}$-algebras $P\nCalg$. 
  Within this framework, we may   impose   further conditions on the
  parametrization, most notably :
  \begin{enumerate}
  \item left exact:  non-empty finite meets are sent to intersections of ideals
  \item regular:   non-empty finite meets are sent to intersections and filtered joins are sent to closures of unions of ideals    \item continuous:    finite meets are sent  to intersections and  joins are sent to closures of sums of ideals    \end{enumerate}
  These conditions lead to  a sequence of sub-$E$-theory contexts \begin{equation}\label{hrtgrtgerge}P^{\rmc}\nCalg\subseteq P^{\reg}\nCalg\subseteq P^{\uc}\nCalg\subseteq P\nCalg\ .
\end{equation} For a topological space $X$ the $E$-theory context
  $\Open(X)^{\rmc}\nCalg$ recovers the definition of the category of $X$-$C^{*}$-algebras given in    \cite{Dadarlat_2012}.
    \cref{gergwergwerioug9erwgwregw} is devoted to the study of  categorical properties of these subcategories. 
  Their understanding is  crucial   because a homological functor  on an $E$-theory context $\cA$  must    preserve those filtered colimits and must  send those exact sequences to fibre sequences which exist in $\cA$. Filtered colimits and exact sequences (i.e., bifibre sequences) are characterized in terms of the  categorical language of $\cA$. To verify that a given functor is homological we must therefore identify these structures explicitly.
 As indicated in \eqref{hrtgrtgerge}, the $E$-theory context $\cA$  usually arises  as a subcategory of an ambient  category of $C^{*}$-algebras with structures in which  
 limits and colimits are well-understood.
In this situation we   must  know  that there are no "unexpected" filtered colimits or exact sequences
  in $\cA$ -- that is, structures that do not coincides with those  in the ambient category.
  
  In 
   homotopy theory, the presentability of a category is a very convenient property, as it implies the existence of all limits and colimits and allows for straightforward  applications of adjoint functor theorems. 
Using standard arguments, we observe that the
  presentablility of $\nCalg$  implies that of the  
   categories $P^{\uc}\nCalg$ and $P\nCalg$, when viewed as subcategories of the functor category $\Fun(P,\nCalg)$. The regular and the continuous cases are more involved. The following result, which establishes their presentability under additional assumptions,  may   be of independent interest. 
      \begin{theorem}
   \mbox{} \begin{enumerate}
   \item (\cref{jioretherthertgetrgertg}) If $P$ is a stably locally compact frame, then $P^{\reg}\nCalg$ is presentable.
   \item (\cref{okprherhrtgertge}) If $P=\Open(X)$ for a locally compact Hausdorff space $X$, then $P^{c}\nCalg$ is presentable.
   \end{enumerate}
   \end{theorem}
 
There is an adjunction $F:\Poset\leftrightarrows \Frame:\incl$,
where $F(Q)$ is the free frame on the poset $Q$. The functor $F$ preserves the property of having a maximal element.
We then have $Q\nCalg\simeq F(Q)^{\rmc}\nCalg$ (see \cite[Sec. 2.9]{zbMATH05657129}). Therefore $P^{\rmc}\nCalg$ is presentable if $P\cong F(Q)$ for some
poset $Q$. 
It remains an interesting question whether,  in these statements, the assumptions on $P$ can be further relaxed, beyond these obvious cases. Our current proof of  \cref{okprherhrtgertge}  heavily relies on the specific structure of a locally compact Hausdorff space and the  explicit use of the algebra $C_{0}(X)$.
        
  Following    \cite{Dadarlat_2012},  we associate 
to any topological space $X$ the $E$-theory context
 $$X\nCalg:=\Open(X)^{\rmc}\nCalg\ .$$
 This definition naturally extends to the setting of locales, yielding  a functor
 \begin{equation}\label{gergwerferwfwerfw}  (-)\nCalg:\Locale\to \EC\ , \quad X\mapsto X\nCalg\ , \quad f\mapsto f_{!}\end{equation} 
 from locales to $E$-theory contexts (see \cref{lkohperghtrgertgertg9}).
 By composing this assignment with the $E$-theory functor   defined in \eqref{gerwgewurgowergwergwerfe} we obtain
 the functor
 $$E:\Locale\to \Pr^{L}_{\st}\ , \quad X\mapsto E(X)\ , \quad  f\mapsto f_{!}\ ,$$
 see  \cref{ohkperrtgertgertg}. 
   
    We  return   to the problem of describing the category $E(X)$ explicitly.
    Our first main theorem  provides a complete homotopy-theoretic characterization of  $E(X)$ for "sufficiently small" locales $X$. The definition of cosheaves $\CoShv(X,\cD)$ with values in a presentable stable $\infty$-category $\cD$ on a locale $X$ is recalled in \cref{ohipeorthretgrtgrteg}.
 For every $A$ in $E(X)$ and open $U$ in $X$ we have an evaluation
 $A(U)$ in $\EE$. The assignment $U\mapsto A(U)$ turns out to be an $\EE$-valued cosheaf    
    on $X$ which we will denoted by $s_{X}(A)$.
    \begin{theorem}[\cref{gweoirgjoergeferggr}.\ref{jiogwrtgergwerfwerf1} and \cref{khoperthergtrge}]\label{hgitgjetrgetogerpigt}
If $X$ is a  locale that can be written as a    filtered union of finite open sublocales, then  there is an equivalence 
$$s_{X}:E(X)\stackrel{\simeq}{\to} \CoShv(X,\EE)\ .$$   
 
\end{theorem}

For interesting infinite topological spaces, even those as simple as the unit interval $[0,1]$, 
 finding an explicit description of the category 
$E(X)$   has  remained an open problem, although some calculations of  the 
morphism groups $\pi_{0}\map_{E([0,1])}(A,B)$  for elementary $C([0,1])$-algebras
$A$ and $B$ are known \cite{zbMATH06347396}.

It turns out that considering the entire functor $X\mapsto E(X)$ at once simplifies the problem considerably, at least for locally compact Hausdorff spaces.  By  specialization  to locally compact Hausdorff spaces the functor
\eqref{gergwerferwfwerfw} gives rise to a functor
\begin{equation}\label{gewrferfewrgv} (-)\nCalg:\LCH\to \EC\ , \quad X\mapsto X\nCalg\ , \quad f \mapsto f_{!}\ . \end{equation}
In \cref{ijiopgwergrgwferfw} we equip the assignment
$$\LCH\ni X\mapsto X\nCalg\in \EC$$ with a contravariant functoriality and an external tensor product  derived from the maximal tensor product of $C^{*}$-algebras. Parts of these structures have  already  been considered in  \cite{zbMATH01213731},
 \cite{Park_2000} and \cite{Popescu_2004}.
\begin{prop}[\cref{ijiopgwergrgwferfw}]\label{hokpjpezjtzjrtzjrztj}
There is a  lax symmetric monoidal functor to $E$-theory contexts (in the sense of \cref{tohpertgertgrteg})
$$(-)\nCalg:\LCH^{\op}\to \Cat\ , \quad X\mapsto X\nCalg\ , \quad f\mapsto f^{*}\ .$$
\end{prop}
To  formally describe the interaction of   the covariant functoriality  with the contravariant functoriality and the tensor products
 we employ the language of functor formalisms.  Following \cite{zbMATH06780490}, \cite{zbMATH07666924}, \cite{arXiv:2410.13038}, \cite{arXiv:2412.15780}, \cite{arXiv:2510.26269}, \cite{arXiv:2507.13537} we
  adopt the concept of a functor formalism on the Nagata context $(\LCH,I,P)$ of locally compact Hausdorff spaces, open inclusions, and proper maps.  We refer to \cref{kojpjrtzjzhrtzh} for the notion of a three-functor formalism on $(\LCH,I,P)$. 
 The three functors are $f^{*}$ and $f_{!}$ for all morphisms in $\LCH$ and 
 the tensor product $\otimes_{X}$ for all $X$ in $\LCH$. To constitute   a three-functor formalism,  these must satisfy various compatibility   conditions detailed  in \cref{kojpjrtzjzhrtzh}.
 \begin{theorem}[\cref{okpherthrtgetgetg}]
 The lax symmetric monoidal functor  
 from \cref{hokpjpezjtzjrtzjrztj}  is a three-functor formalism on the Nagata context $(\LCH,I,P)$. \end{theorem}
 
By composing the lax symmetric monoidal extension of the $E$-theory functor from   \eqref{gerwgewurgowergwergwerfe} with the functor from \cref{hokpjpezjtzjrtzjrztj} we get a lax symmetric monoidal functor
 \begin{equation}\label{gwergwerfwerfw}
 E: \LCH^{\op}\to \Pr^{L}_{\st}\ , \quad X\mapsto E(X)\ , \quad f\mapsto f^{*}\ .\end{equation}
  Exploiting   the fact that $E$ takes values in presentable $\infty$-categories
and left adjoint functors we arrive at:
 \begin{theorem}[\cref{tkphrethergertgertg}]
 The lax symmetric monoidal functor  $E$ in
  \eqref{gwergwerfwerfw}  is a presentable six-functor formalism on the Nagata context $(\LCH,I,P)$.  \end{theorem}
 The six functors are the functors $f^{*}$,  $f_{!}$ and $\otimes_{X}$ from above, their right-adjoints $f_{*}$, $f^{!}$, and the interal hom-functor $\Hom_{X}$. We refer to \cref{hertkprtgertgtregertg} for the complete definition of a six-functor formalism.
  Furthermore, the result of   \cite[Prop. 3.3.3]{arXiv:2410.13038} or  \cite[Thm 3.3]{arXiv:2412.15780} ensures that our definition of a six-functor  formalism is equivalent to the definition of a six-functor  formalism as a lax symmetric monoidal  functor defined on a span category associated to the geometric context $(\LCH,\All)$   \cite{zbMATH06780490},   \cite[Def. 2.4]{arXiv:2412.15780}, \cite[Def. 2.5]{arXiv:2510.26269}, \cite[Def. 1.2.4]{arXiv:2410.13038}, where the class of $!$-able maps, denoted by  $\All$,  consists of all maps.

%
%
%
%

The classical six-functor formalism on the Nagata context $(\LCH,I,P)$, that agrees with $\EE$   on the one-point space $\pt$,  is the sheaf-theoretic six-functor formalism
$$\Shv(-,\EE):\LCH^{\op}\to \Pr^{L}_{\st},  \quad X\mapsto  \Shv(X,\EE)\ ,  \quad f\mapsto \hat f^{*}$$
(see \cite{Volpe:2021aa} for a complete construction).  In \cite{buvo} we characterize $\Shv(-,\cD)$ for  $\cD$ in $\Calg(\Pr^{L}_{\st})$ with dualizable underlying  object    among  six-functor  formalisms $D$ with $D(\pt)=\cD$
  by universal properties. This  characterization is recalled here as \cref{thkoperthertegrtger} and rests on three conditions.

The first condition is that the six-functor formalism $D$ is a coefficient system \cite{arXiv:2507.13537} (see  also \cite{zbMATH07666924} for the origin  of this concept in algebraic geometry). Being a coefficient system is equivalent to  the canonical descent condition  described in \cref{kohperthertgertg}.
\begin{theorem}[\cref{thkoptrhergtrgege}]
The six-functor formalism  $E$ in \eqref{gwergwerfwerfw} is a coefficient system.
\end{theorem}
The second condition for  \cref{thkoperthertegrtger} concerns the associated 
 cohomological functor $\Gamma^{D}$ (see   \cref{okphertgertgerg}). In the case of the six-functor formalism of sheaves $\Shv(-,\cD)$ it  sends  $X$ in $\LCH$  to the usual sheaf cohomology 
 $$\Gamma^{\Shv(-,\cD)}(X,1)\in \cD$$
 of  the constant sheaf on $X$
with value $1$, the tensor unit of $\cD$. In the case of the six-functor formalism $E$ from \eqref{gwergwerfwerfw} the associated cohomological functor sends 
$X$ in $\LCH$ to $$\Gamma^{E}(X)=\lim_{K\subseteq X} \ee(C(K))\in \EE\ ,$$
where the limit runs over all compact subsets of $X$. The second condition   requires   the associated cohomological functor
  to be  finitary in the sense of \cref{iugwoierwerfwefwef}, i.e., it must  send cofiltered systems of compact Hausdorff spaces to colimits.  While the  detailed  argument for the continuity of $\Gamma^{\Shv(-,\cD)}$  is quite involved
 (see \cite[Thm. 3.10]{NKP}),   the argument for the continuity of $\Gamma^{E}$ given in the proof of \cref{tkohperthretgertg} is  considerably simpler and merely uses Gelfand duality and  the fact (true by definition)  that the $E$-theory functor 
\eqref{vasdcasdcadscasca} preserves filtered colimits.
 
 The  third condition  for  \cref{thkoperthertegrtger} is that the six-functor formalism $D$ is section-determined (see \cref{lkpohejztjrtz}). For $E$ the classical analogue of this  property was verified  in   \cite[Thm. 3.10]{Dadarlat_2012}.   \begin{theorem}[\cref{trkopherthrgertg}]\label{trkopherthrrrrgertg}
The functor $E$  in \eqref{vsdfvsfsdr} is section-determined.
\end{theorem} Our proof of \cref{trkopherthrrrrgertg} relies on the same idea as  that of
 \cite[Thm. 3.10]{Dadarlat_2012}, but   technically it is completely independent.

We can now apply   \cref{thkoperthertegrtger}  in order to obtain the second main theorem of the present paper:
\begin{theorem}[\cref{trkohjperthretgertgertg}]\label{zzopjrtzjrztjrth}
There is a natural equivalence of six-functor formalisms
\begin{equation}\label{oijoagadsgasgg} \cB:\Shv(-,\EE)\stackrel{\simeq}{\to} E(-):\LCH^{\op}\to \Pr^{L}_{\st}\ . \end{equation}
\end{theorem}

This equivalence provides   complete homotopy-theoretic  characterization of $E(X)$ and its associated functorial operations. Furthermore it enables 
 the  transfer of established  results and constructions from sheaf theory to $E$-theory of $X$-$C^{*}$-algebras -- a direction which will be explored 
 in subsequent papers.

As an immediate consequence, we highlight the following result regarding dualizability. 
 Since $\EE$ is known to be a dualizable presentable stable $\infty$-category by \cite{budu}, and it has been established that this property extends to  $\Shv(X,\EE)$  for any locally compact Hausdorff space $X$
  \cite{Efimov:2024aa} (see also  \cite{NKP} for a detailed argument) we obtain:

\begin{kor} For every locally compact Hausdorff space the presentable stable $\infty$-category  
$E(X)$ is dualizable. 
\end{kor}
We note that the method  employed in  \cite{budu}  to prove the dualizability of $\EE$ does not admit a  direct generalization  to the case of $E(X)$.

If $F$ is in $\Shv(X,\EE)$, then we have the equivalence $\cB(F)(x)\simeq F_{x}$,
where for $A$ in $E(X)$ we let  $A(x):=  s_{X}(A)(X) /s_{X}(A)(X\setminus \{x\})$ denote the fibre  of $A$ at $x$, and
$F_{x}$ denotes the stalk of the sheaf $F$ at $x$. Recall that a locally compact Hausdorff space
$X$ is called hypercomplete if equivalences in $\Shv(X,\Ani)$ can be detected on stalks. For example, finite-dimensional spaces are hypercomplete. If $X$ is hypercomplete, then equivalences in any sheaf category taking  values in a dualizable target can be detected on stalks. In particular, this  applies to $\Shv(X,\EE)$.
The following consequence of \cref{zzopjrtzjrztjrth} is an analog of \cite[Thm. 1.1]{zbMATH05555716}:
\begin{kor}
If the locally compact Hausdorff space $X$ is hypercomplete, then
the collection of fibres $(E(X)\ni A\mapsto A(x)\in \EE)_{x\in X}$ is jointly conservative.
\end{kor}

 Note that \cref{hgitgjetrgetogerpigt} and \cref{zzopjrtzjrztjrth} describe $E(X)$ modulo the knowledge of 
 the symmetric monoidal category  $\EE$ in $\CAlg(\Pr_{\st}^{L})$ itself. The latter contains a full subcategory, called the bootstrap class,  which is equivalent to the category $\Mod_{KU}(\Sp)$ known from homotopy theory. The 
 bootstrap class $\rmB\subseteq \EE$   is the image of the symmetric monoidal left-adjoint
 in the adjunction  \begin{equation}\label{goiewrugiowergerwferfw}  B:\Mod_{KU}(\Sp)\leftrightarrows \EE:K:=\map_{\EE}(1,-)\ ,\end{equation} 
 whose right adjoint is the lax symmetric monoidal $K$-theory functor.
Crucially,   the calculation $$\map_{\EE}(1,1)\simeq KU$$
  implies that $\EE$ is enriched in $\Mod_{KU}(\Sp)$. This, in turn, implies that
 the six-functor formalism $E$   \eqref{gwergwerfwerfw}  actually takes values in  $KU$-module enriched stable $\infty$-categories. If $A$ is in the bootstrap class of $\EE$, then
 the UCT and the K\"unneth formula
 $$\map_{\EE}(A,B)\simeq \map_{KU}(K(A),K(B)) \ ,\quad K(A\otimes B)\simeq K(A)\otimes_{KU} K(B)$$ 
 for all $B$ in $\EE$
  are formal consequences of the definitions.
 More generally we say that $A$ in $\EE$ belongs to the K\"unneth class if the canonical map
  $$K(A)\otimes_{KU} K(B)\to K(A\otimes B)$$ is an equivalence for every $B$ in $\EE$. 
  The problem of defining a bootstrap class in $E(X)$ and to obtain a UCT for finite spaces $X$   has already been considered    in
\cite{zbMATH05657129}. 
In view of \cref{kohpreherthrtge}, for locales $X$ that are unions of finite open sublocales we could define the bootstrap class $B(X)\subseteq E(X)$
as the preimage of the full subcategory
$$\CoShv(X,\Mod_{KU}(\Sp))\subseteq \CoShv(X,\EE)$$
under the equivalence $s_{X}$, where the inclusion above is realized by the pointwise application of the colimit-preserving functor $B$ from \eqref{goiewrugiowergerwferfw}. 
\begin{ddd}
We define the $K$-theory cosheaf  functor
$$\CK:(K\circ -)\circ s_{X}:E(X)\to \CoShv(X,\Mod_{KU}(\Sp))\ .$$
\end{ddd}
The following is the general version of the universal coefficient theorem for $E(X)$.
\begin{kor} \label{hrthrhtrgrtger} Assume that $X$ is a locale  that is the union of finite open sublocales.
If $A$ is in the bootstrap class $B(X)$, then we have
$$\map_{E(X)}(A,B)\simeq \map_{\CoShv(X,\Mod_{KU}(\Sp))}(\CK(A),\CK(B))$$
for all $B$ in $E(X)$.
\end{kor}

\begin{rem} For a locally compact Hausdorff space $X$
one can show that the composition $$\cV\circ \cK_{X}:E(X)\to \Shv(X,\EE)\to \CoShv(X,\EE)$$
of the Verdier duality functor and the right-adjoint (or inverse) of $\cB_{X}:\Shv(X,\EE)\to E(X)$ from 
\cref{zzopjrtzjrztjrth} is equivalent to the functor $s_{X}$ from \cref{gweoirgjoergeferggr}. 
It   follows from Verdier duality that
$$s_{X}:E(X)\to   \CoShv(X,\EE)$$ is an equivalence also for locally compact Hausdorff spaces.
Consequently, the above considerations and \cref{hrthrhtrgrtger} remains  valid for locally compact Hausdorff spaces $X$. We will discuss this and applications to the Künneth class in future work. 
\end{rem}

{\em Acknowledgements: The author would like to thank Benjamin Dünzinger,   Marco Volpe, and Christoph Winges   for many fruitful discussions regarding various aspects of this work.
}

  \section{$E$-theory contexts and $E$-theory}

  \subsection{$E$-theory contexts}

In this section, we introduce the concept of an $E$-theory context.
Consider the symmetric monoidal category $$\cN:=\nCalg_{\sepa, \nucl} $$ of separable nuclear  $C^{*}$-algebras    as an object in $\CAlg(\Cat)$ and form
the   category $\Mod_{\cN}(\Cat)_{*}$ of pointed $\cN$-modules.
Let $\cA$ be in $\Mod_{\cN}(\Cat)_{*}$  with  zero object $0$.
\begin{ddd}
A commutative square $$\xymatrix{A\ar[r]\ar[d] &B \ar[d] \\0 \ar[r] & C} $$ in $\cA$ is called an exact sequence if it is  both cartesian and cocartesian.
\end{ddd}

 \begin{ddd}\label{gwergerfwerfw}
 A morphism $f:\cA\to \cB$ in $\Mod_{\cN}(\Cat)_{*}$ is called $E$-admissible if it preserves finite products, filtered colimits   and exact sequences.
 \end{ddd}
 

  
We denote the $\cN$-module structure of $\cA$ by the tensor product $$-\otimes -:\cA\otimes \cN\to \cA\ .$$ 
\begin{ddd}\label{lphertgerthe}\mbox{}
\begin{enumerate} \item 
An object $\cA$ in $\Mod_{\cN}(\Cat)_{*}$ is called an $E$-theory context if, for every $N$ in $\cN$ and  every $A$ in $\cA$,
the functors $$A\otimes -:\cN\to \cA\ , \qquad  -\otimes N:\cA\to \cA$$ are $E$-admissible.
\item We let $\EC$ denote the subcategory of $\Mod_{\cN}(\Cat)_{*}$ consisting of $E$-theory contexts and $E$-admissible morphisms.
\end{enumerate}
\end{ddd}

Note that $\EC$ is a 2-category.
 \begin{ex}
 The category $\cN$ itself  is an $E$-theory context.    \hB
 \end{ex}
 
 In \cref{gergwergwerioug9erwgwregw} we will describe a variety of further examples of $E$-theory contexts.

 \begin{rem}
 Filtered colimits, products and exact sequences in $\cA$ are defined purely in terms of the category $\cA$.
 We do not require  an $E$-theory context to admit  all filtered colimits, finite products, kernels or quotients.
A morphism between  $E$-theory contexts is only required to preserve those 
filtered colimits, finite products or exact sequences which exist in the domain.

Most of our examples of $E$-theory contexts in  \cref{gergwergwerioug9erwgwregw} arise from diagrams of $C^{*}$-algebras. These examples have natural explicit candidates for filtered colimits, products or exact sequences. When working with  these $E$-theory contexts, it is crucial to verify that  these "natural" candidates  coincide 
with the 
  categorical constructions.
\hB
 \end{rem}

\subsection{The $E$-theory functor}
In this section we introduce the notion of a homological functor on an $E$-theory context. We then   construct the $E$-theory functor as the universal homological functor.

An $E$-theory context  $\cA$ carries a natural  topological enrichment such that 
for every    second countable  compact Hausdorff space
$W$,   there is a natural isomorphism
$$\Hom_{\Top}(W,\Hom_{\cA}(A,B))\cong \Hom_{\cA}(A,B\otimes C(W))\ .$$
The topological enrichment   induces a   notion of homotopy equivalence in $\cA$.

 Furthermore, we define a left upper corner inclusion for an object  $A$ in $\cA$ to be a morphism of the form     $$A\cong A\otimes \C\stackrel{\id\otimes i}{\to} A\otimes K\ , $$
where $i:\C\to K$ is the morphism in $ \cN $  from $\C$ into the algebra $K$ of compact operators on a separable Hilbert space  that sends $1$ in $\C$ to a minimal projection in $K$.

\begin{rem}
These definitions  rely on   the $\cN$-module structure of $\cA$  and the fact that both  $C(W)$ (for a second countable compact Hausdorff space) and $K$   belong to $\cN$.  
 \hB
\end{rem}

Let $\cA$ be an $E$-theory context and
  $F:\cA\to \cC$ be a functor  into a cocomplete stable $\infty$-category.

  \begin{ddd}\label{kogwegergfw}
 The functor $F:\cA\to \cC$ is called homological satisfies  the following properties:
 \begin{enumerate}
 \item $F$ preserves zero objects.
 \item  $F$ sends homotopy equivalences to equivalences.
 \item  $F$ sends left upper corner inclusions to equivalences.
 \item  $F$ sends exact sequences to fibre sequences.
 \item  $F$ preserves  finite products.
 \item  $F$  preserves  filtered colimits.
  \end{enumerate}
  \end{ddd}

 We denote by $\Fun^{\homol}(\cA,\cC)$   the full subcategory of $\Fun(\cA,\cC)$ consisting of homological functors.
 
 \begin{ex}
 The category $\nCalg$ of $C^{*}$-algebras carries the structure of an $E$-theory context, see also \cref{kophehrtrgertg} below. The $K$-theory functor
 $$K:\nCalg\to \Sp$$
 serves as the primary    motivating example of a homological functor. \hB
 \end{ex}

\begin{rem}
In order address   size issues,
 we fix two consecutive universes, referred to as the universes of small and large sets. We assume that all $C^{*}$-algebras
 and the index sets for the  filtered colimits in \cref{kogwegergfw} in the definitions of $E$-theory contexts and their morphisms belong to the small universe. In contrast, 
the functor categories,  the categories of spectra $\Sp$ and the test categories $\cC$  below  belong to the large universe.
\hB\end{rem}

 \begin{theorem}\label{koprtheggrtge} For every  $E$-theory context $\cA$,
 there exists a unique homological functor $$\ee_{\cA}:\cA\to \EE(\cA)$$  inducing an equivalence
\begin{equation}\label{hrtgrtgetghertherhtr}\ee_{\cA}^{*}:\Fun^{\colim}(\EE(\cA),\cC)\stackrel{\simeq}{\to}\Fun^{\homol}(\cA,\cC)
 \end{equation} for every cocomplete stable $\infty$-category $\cC$.
\end{theorem}\begin{proof}
This proof follows the general procedure for constructing  universal homological functors \cite{zbMATH01596160},   \cite{zbMATH05771937}, \cite{MR3070515}.
Consider  the functor
$$\Sigma^{\infty}_{+}y:\cA\to \Fun(\cA^{\op},\Sp)$$
defined as  the composition of the  Yoneda embedding   with the suspension spectrum functor.
We identify a  small set of morphisms  in $\Fun(\cA^{\op},\Sp) $ consisting of:
\begin{enumerate}
 
\item $0\to \Sigma^{\infty}_{+}y(0)$
\item $\Sigma^{\infty}_{+}y(f)$ for all homotopy equivalences $f$ in $\cA$
\item $\Sigma^{\infty}_{+}y(f)$ for all left-upper corner embeddings $f$ in $\cA$
\item the comparison maps $\bigoplus_{i\in I}  \Sigma^{\infty}_{+}y( A_{i}) \to \Sigma^{\infty}_{+}y(\prod_{i\in I} A_{i})$ for all  finite  families $(A_{i})_{i\in I}$ in $\cA$ for which  the  product exists.
\item  the comparison maps  $\colim_{i\in I}  \Sigma^{\infty}_{+}y( A_{i})\to  \Sigma^{\infty}_{+}y(\colim_{i\in I} A_{i})$ for all   filtered diagrams $I\ni i\mapsto  A_{i}\in \cA$ in $\cA$  for which the colimit exists.
 \end{enumerate}
 
Let $$L_{0}: \Fun(\cA^{\op},\Sp) \leftrightarrows   \mathrm{LOC}:\incl$$
be the  left Bousfield localization  with respect to these morphisms,
where $ \mathrm{LOC}$ denotes the full subcategory of $\Fun(\cA^{\op},\Sp)$ of local objects.

In a second step
we form the left Bousfield localization 
$$L_{1}: \mathrm{LOC}\leftrightarrows \EE(\cA):\incl$$
which inverts the comparison morphisms 
\begin{equation}\label{greferfwerfwerfwerfrw}L_{0}\Sigma^{\infty}_{+}y(A)\to \Fib(L_{0}\Sigma^{\infty}_{+}y(B)\to L_{0}\Sigma^{\infty}_{+}y(C))
\end{equation}
for all exact sequences $0\to A\to B\to C\to 0$ in $\cA$. The first step in this two-step process 
ensures that $L_{0}(0)\simeq 0$, which is required for the comparison maps in \eqref{greferfwerfwerfwerfrw} to be well-defined.
 We then define \begin{equation}\label{refwerfwerfwrfefer}L:=L_{1}\circ L_{0}:\Fun(\cA,\Sp)\to \EE(\cA)
\end{equation}   and set 
  $$\ee_{\cA}:=  L\circ \Sigma^{\infty}_{+}y:\cA\to \EE(\cA)\ .$$
By construction, this functor has the desired universal property  \eqref{hrtgrtgetghertherhtr}.

Note that the $E$-theory functor $\ee_{\cA}:\cA\to \EE(\cA)$ is uniquely determined (up to equivalence) by the 
universal property \eqref{hrtgrtgetghertherhtr}.  
\end{proof}

\begin{rem}
It follows from the construction in the proof of \cref{koprtheggrtge}  that  $\EE(\cA)$  is an object of $\Pr^{L}_{\st}$.
Furthermore, using the existence of the Toeplitz extension $$0\to K\to \cT\to C(S^{1})\to 0$$
in $\cN$, one can show, following the arguments in  \cite{Bunke:2023aa},  that $\EE(\cA)$ satisfies  Bott periodicity,
i.e, that there is an  equivalence $\Omega^{2}\simeq \id$. \hB
\end{rem}

    \begin{lem}\label{kophertgretger9}
 A  morphism
   $f:\cA\to \cB$  of $E$-theory contexts   admits a unique colimit-preserving factoriation
$$\xymatrix{\cA\ar[r]^{f}\ar[d]^{\ee_{\cA}} &\cB \ar[d]^{\ee_{\cB}} \\ \EE(\cA)\ar@{..>}[r]^{\EE(f)} &\EE(\cB) }\ . $$ 
 \end{lem}
\begin{proof}
By the  universal property of $\ee_{\cA}$, it suffices  to show that
$\ee_{\cB}\circ f$ is a homological functor. Since $f$ is a morphism of $\cN$-modules, it necessarily preserves 
  homotopy equivalences  and left upper corner inclusions. 
Furthermore, since $f$ is $E$-admissible, it preserves   exact sequences,  finite products and  filtered colimits.  Consequently,    $\ee_{\cB}\circ f$  satisfies all conditions of \cref{kogwegergfw} and is thus a homological functor.     \end{proof}

\begin{theorem}\label{kohpertgertgertge}
There exists a  2-functor
\begin{equation}\label{bdsvsdfvsdvsdfs}\EE:\EC\to \Pr^{L}_{\st}\end{equation}   equipped with a natural transformation $\ee:\id\to \EE$ of $\Cat_{\infty}$-valued functors  making the square
$$\xymatrix{\EC\ar[r]^{\EE}\ar[d]  &\Pr^{L}_{\st} \ar[d]  \\\Cat \ar[r] & \Cat_{\infty}} $$ commute.
\end{theorem}
\begin{proof}
Recall that $\EE(\cA)$ is obtained as a localization of $\Fun(\cA^{\op},\Sp)$ at the set of morphisms
$W_{\cA}$ inverted by   $ L$ in \eqref{refwerfwerfwrfefer}.  In light of \cref{kophertgretger9}
the $2$-functor
$$\Fun((-)^{\op},\Sp):\EC\to \Pr_{\st}^{L}$$ refines to a $2$-functor into the category of relative presentable stable  categories
$$\Fun((-)^{\op},\Sp):\EC\to \mathbf{Rel}\Pr_{\st}^{L}\ ,\quad \cA\mapsto (\Fun(\cA^{\op},\Sp),W_{\cA})\ .$$ Applying  the localization functor
$ \mathbf{Rel}\Pr_{\st}^{L}\to  \Pr_{\st}^{L}$ yields the desired $2$-functor $\EE$.
%
%
\end{proof}

\subsection{Symmetric monoidal $E$-theory}

We do not possess  a symmetric monoidal category of $E$-theory contexts. If we model the symmetric monoidal category $\Cat$ by an operad $\Cat^{\times}$  as in \cite{HA}, then    $\EC$ can be viewed as a suboperad 
$\EC^{\otimes}$. However, the suboperad   lacks sufficient  cocartesian lifts to  constitute  a symmetric monoidal category.

Using the operad language,  one can understand lax symmetric monoidal functors into $\EC$ as operad maps into $\Cat^{\times}$ that happen to take values in $\EC^{\otimes}$. The following
definitions provide a translation  of this perspective into the language of    ordinary category theory. 
 
 Consider a lax symmetric  monoidal  functor
$$\cA:\cP \to \Cat$$ from a symmetric monoidal category $(\cP,\otimes_{\cP},1_{\cP})$.  \begin{ddd} \label{tohpertgertgrteg}\mbox{} \begin{enumerate} \item  
 We call $\cA$ a lax symmetric monoidal functor from $\cP$ to $E$-theory contexts if it satifies:
   \begin{enumerate}
\item \label{herhrtge}$\cA(1_{\cP}) $ receives a symmetric monoidal functor from $\cN$  and $\cA(p)$ is pointed for every $p$ in $P$.  
\item  \label{herhrt2ge}  For every morphism  $p\to q$ in $\cP$ the morphism 
$\cA(p)\to \cA(q)$ is $E$-admissible.  \item \label{kohpetrhertgtreg} For every $A$ in $\cA(p)$ and every  $q$ in $\cP$, the functor
$A\boxtimes  -:\cA(q)\to \cA(p\otimes_{\cP} q)$  is $E$-admissible.\end{enumerate}
 \item A natural transformation $\cA\to \cB $ between two such functors   is called  a natural transformation between lax symmetric monoidal functors of $E$-theory contexts if the morphism $\cA(p)\to \cB(p)$ is $E$-admissible for every $p$ in $\cP$.\end{enumerate}
 \end{ddd}
 

\begin{rem}
In Condition \cref{tohpertgertgrteg}.\ref{herhrtge} we use that $\cA(1_{\cP})$ is naturally a symmetric monoidal category.
\end{rem}

 \begin{rem}\label{oekpzjetzherhtr}
 If we take $\cP=*$ in  \cref{tohpertgertgrteg}, then we refer to $\cA$ as a symmetric monoidal $E$-theory context. \hB
 \end{rem}

 \begin{rem}
 Let $\cA$ be a lax symmetric monoidal functor from $\cP$ to $E$-theory contexts as defined in  \cref{tohpertgertgrteg}.
   Condition \ref{herhrtge}  of \cref{tohpertgertgrteg} 
 ensures that the values $\cA(p)$ belong to $\Mod_{\cN}(\Cat)_{*}$, while condition 
  \ref{kohpetrhertgtreg}  implies that $ \cA(p)$ is  indeed an $E$-theory context.
Consequently, by Condition \ref{herhrt2ge}, we obtain  a functor \begin{equation}\label{gwergwrferfw}\cA:\cP\to \EC\ .
\end{equation} 
   \hB \end{rem}

By composing the functor from \eqref{gwergwrferfw} with $\EE$ from \eqref{bdsvsdfvsdvsdfs}, we obtain the functor
\begin{equation}\label{ferwfwegeregwegergw}\EE(\cA):\cP\to \Pr^{L}_{\st}\ . \end{equation} 
 
\begin{theorem}\label{kohpehtregetrg} If $\cA$ is a lax symmetric monoidal functor to $E$-theory contexts, then
the functor  in \eqref{ferwfwegeregwegergw}  admits a canonical refinement to a lax symmetric monoidal functor.
\end{theorem}
\begin{proof}
The composition of  lax symmetric monoidal functors  
$$\Fun(\cA(-),\Sp):\cP\stackrel{\cA}{\to} \Cat\stackrel{\Fun(-,\Sp)}{\to} \Pr^{L}_{\st}\stackrel{\incl}{\to} \Cat_{\infty}$$  lifts to a lax symmetric monoidal functor into the $\infty$-category of relative $\infty$-categories
$$(\Fun(\cA(-),\Sp),\cW(-)):\cP\to \RelCat_{\infty}\ ,$$ 
where $\cW(p)$ denotes the set of morphisms inverted by 
$L:\Fun(\cA(p),\Sp)\to \EE(\cA(p))$.
To this end we must check
  for every $p,q$ in $\cP$, object $X$ in $\Fun(\cA(p),\Sp)$ and $f$ in $\cW(p)$, that
the morphism $X\boxtimes f$ is in $\cW(p\times q)$.
Since the external product $\boxtimes$ preserves colimits in each variable it suffices to show this for
objects of the form $X=\Sigma^{\infty}_{+}y(A)$ with $A$ in $\cA(p)$, and the generators $f$ of the localization 
listed in the proof of \cref{koprtheggrtge}. This follows immediately from the fact that the functor  
$A\boxtimes -$ is $E$-admissible by \cref{tohpertgertgrteg}.\ref{kohpetrhertgtreg}

We now compose with the symmetric monoidal localization functor
$$\RelCat_{\infty}\to \Cat_{\infty}$$  to obtain a lax symmetric monoidal functor
$$\EE(\cA):\cP\to \Cat_{\infty}\ .$$
Finally, we observe that this functor factorizes  through a lax symmetric monoidal  functor into $\Pr^{L}_{\st}$, as indicated 
 by the dotted arrow in the following diagram:
$$\xymatrix{\cP\ar[rr]^{\EE(\cA)}\ar@{..>}[dr]&&\Cat_{\infty}\\&\Pr^{L}_{\st}\ar[ur]&}\ .$$   \end{proof}

  \begin{ex}\label{kophehrtrgertg}
 The symmetric monoidal category $\nCalg$ of all $C^{*}$-algebras with   the maximal tensor product constitutes a symmetric monoidal $E$-theory context  in the sense of \cref{tohpertgertgrteg} and \cref{oekpzjetzherhtr}.  This follows from the fact, that  
 the  maximal tensor product in $\nCalg$ preserves finite products, filtered colimits  and exact sequences.

 By comparing universal properties,  the lax symmetric monoidal functor obtained in \cref{kohpehtregetrg} coincides with the $E$-theory functor  introduced in \cite{Bunke:2023aa}, \cite{budu}.  As verified in  \cite[Sec. 13]{Bunke:2023aa},   on the level of homotopy categories it recovers the classical $E$-theory of \cite{zbMATH04182148}, \cite{MR1068250}, \cite{Guentner_2000}. 
\hB
 \end{ex}

 \section{$E$-theory contexts associated to posets}\label{gergwergwerioug9erwgwregw}
 
 In this section, we introduce  the $E$-theory contexts in the chain \eqref{hrtgrtgerge}, along with several additional  ones, and describe their categorical properties. The general principle is to describe each of them as a subcategory of the preceding category. We then consider closure properties under categorical constructions, such as exact sequences or filtered colimits. In most cases the smaller category is a localization of the ambient one; and in this case it inherits relevant categorical properties -- for example   presentability.

 \subsection{Functor categories}

%
%
%
 
  Since $\nCalg$ is a symmetric monoidal category we have a 
  lax symmetric monoidal functor
  \begin{equation}\label{ferfewrvefweeeee}\Fun(-,\nCalg):\Cat^{\op}\to \Cat\ , \end{equation} 
  where $\Cat $ is equipped with the Cartesian monoidal structure.

  \begin{prop}\label{kohprthertgretget}
  The functor in \eqref{ferfewrvefweeeee} is  a lax symmetric monoidal functor to $E$-theory contexts.
 \end{prop}
 \begin{proof} 
 We verify the conditions from \cref{tohpertgertgrteg}.
 
Clearly, we  have $\Fun(\{*\},\nCalg) \simeq \nCalg \in \CAlg(\Mod_{\cN}(\Cat))$. 
 Furthermore, for every $P$ in $\Cat$ the category $\Fun(P,\nCalg)$ is pointed by the constant functor 
to  zero algebra.
 
 If $f:P\to Q$ is a morphism in $   \Cat$, then $$f^{*}: \Fun(Q,\nCalg)\to \Fun(P,\nCalg)$$
  is $E$-admissible since limits and colimits in the functor category are formed pointwise and are therefore preserved by $f^{*}$.
  
 For $A$ in $\Fun(P,\nCalg)$ and $Q$ in $\Cat$ the tensor product $$A\boxtimes -: \Fun(Q,\nCalg)\to 
  \Fun(P\times Q,\nCalg)$$
  sends $B$ in $\Fun(Q,\nCalg)$ to the functor
  $(p,q)\mapsto A(p)\otimes B(q)$. 
   Since limits and colimits in the functor category are formed pointwise and the maximal tensor product in $\nCalg$ preserves finite products,  filtered colimits  and exact sequences
  we conclude  that $A\boxtimes -$ is $E$-admissible.  
   \end{proof}
   
   Note that $\nCalg$ is a presentable category. It is $\aleph_{1}$-presentable,  and its subcategory of $\aleph_{1}$-compact objects is precisely the subcategory of separable $C^{*}$-algebras. 
  In particular,  $\nCalg$ admits all limits and colimits. \begin{kor} For every $P$ in $\Cat$
  the functor category $\Fun(P,\nCalg)$ is   presentable.    
\end{kor}
Consequently, the category $\Fun(P,\nCalg)$  has all limits and colimits. 

 \subsection{$P$-$C^{*}$-algebras}
 
  A poset $P$ is called pointed if it admits a maximal element $\infty_{P}$. 
 We consider the full symmetric monoidal subcategory of pointed posets,  $\Poset_{(*)}\subseteq \Cat$. 
 The notation $(*)$ instead of $*$    indicates that morphisms in  $\Poset_{(*)}$ are not required to preserve the maximal elements.

Assume that $P$ is a pointed poset.
\begin{ddd}\mbox{} \begin{enumerate}\item
A $P$-$C^{*}$-algebra is an object $A$ of $\Fun(P,\nCalg)$
such that for every $p,q$ in $P$ with $p\le q$  the map $A(p)\to A(q)$ is an ideal inclusion.\item 
We let $P\nCalg$ denote the full subcategory of $\Fun(P,\nCalg)$  of $P$-$C^{*}$-algebras.  \end{enumerate}
\end{ddd}
 \begin{rem} We let   $I(A)$ denote the frame of  closed two-sided $*$-ideals  in a $C^{*}$-algebra $A$  with the inclusion relation.
 An object of $P\nCalg$ is a pair $(A,A(-):P\to I(A))$ consisting of a $C^{*}$-algebra and a poset morphism such that $A(\infty_{P})=A$.
 We call $A$ the underlying algebra.
 A morphism of $P$-$C^{*}$-algebras $(A,A(-))\to (B,B(-))$   is then morphism $f:A\to B$  of $C^{*}$-algebras such that  
  $f(A(p))\subseteq B(p)$ for all $p$ in $P$.
  \end{rem}

 \begin{prop}\label{khterthertgetrg9} For every pointed poset $P$ we have a left Bousfield localization
\begin{equation}\label{goierjgiowerfwrefrefw} L:\Fun(P,\nCalg)\leftrightarrows P\nCalg:\incl\ .
\end{equation} 
The category $P\nCalg$ is presentable, and the inclusion $\incl$ is $E$-admissible.      \end{prop}
     \begin{proof}
     We  first construct the functor $L$ and provide the unit and counit of the adjunction.
If $A$ is in $ \Fun(P,\nCalg) $, then 
we set $L(A)(\infty_{P}):=A(\infty_{P})$. For every $p$ in $P$
 we furthermore define $L(A)(p)$ as the ideal generated by the image of $A(p)\to A(\infty_{P})$.
We extend $L$ to a functor by defining it on morphisms in the canonical way. 

 The unit and counit  of the adjunction are given by the canonical morphism
 $A\to \incl\circ  L(A)$ and the identity $L\circ \incl(A)= A$.
 The triangle identities are straightforward to check.
 
We  now check that $P\nCalg$ is closed in $\Fun(P,\nCalg)$ under filtered colimits. In  
  $\Fun(P,\nCalg)$  they are calculated pointwise. We  employ the fact that  filtered colimits in   $\nCalg$ preserve   ideal inclusions (see \cref{hpzhjetophherth}). 

 Since $\incl$ preserves kernels, it remains to  show that $P\nCalg$ is closed in $\Fun(P,\nCalg)$ under quotients. We consider a  pushout   $$\xymatrix{A\ar[r]\ar[d] &B \ar[d] \\ 0\ar[r] & C} $$ in $\Fun(P,\nCalg)$
  and assume that $A$  and $B$ belong to $ P\nCalg$.  We can assume that $A\cong \ker(B\to C)$ without changing the quotient.  We must show that  
  $C$  also belongs to $  P\nCalg$.  Since quotients in the functor category are formed pointwise,
 we have $C(p)\cong B(p)/A(p)$ for every $p$ in $P$.   Then, 
 $B(p)/A(p)\cong (B(p)+A(\infty_{P}))/A(\infty_{P})  $
 is an ideal in $C(\infty_{P})$.
%
%

 Since $P\nCalg$ is an accessible left Bousfield localization of a presentable category, it is itself presentable.
\end{proof}

 \begin{rem}
 One can check that the Bousfield localization \eqref{goierjgiowerfwrefrefw} is symmetric monoidal, and that $\incl$ and $L$ are symmetric monoidal. \hB
 \end{rem}


\begin{prop}\label{koprtzjrtzhrzthr} We have a  lax symmetric monoidal functor
  \begin{equation}\label{ferfewrvef4r4r4r4r4weeeee}(-)\nCalg:\Poset_{(*)}\to \Cat\end{equation}  to $E$-theory contexts, with the structure of  a subfunctor of $\Fun(-,\nCalg)_{|\Poset_{(*)}}$.
\end{prop}
\begin{proof}
 If $f:P\to Q$ is a morphism in $\Poset_{(*)}$, then
 $f^{*}$ sends $A$ in $Q\nCalg$ to the functor 
 $ p\mapsto A(f(p))$, which again  belongs to $P\nCalg$.   
 
 For $A$ in $P\nCalg$ and $B$ in $Q\nCalg$ one checks  using \cref{ihrfiwuoghewergwreg} that
 the tensor product
 $A\boxtimes B$  belongs to $(P\times Q)\nCalg$.
  
 By \cref{khterthertgetrg9} the category $P\nCalg$ is closed in $\Fun(P,\nCalg)$ under  finite products and   filtered colimits, and 
 the inclusion $P\nCalg  \to\Fun(P,\nCalg) $ is $E$-admissible.     

  It therefore follows from \cref{kohprthertgretget} that $(-)\nCalg$ is a  lax symmetric monoidal functor to $E$-theory contexts, and that  the inclusion  is a natural transformation between lax symmetric monoidal functors to $E$-theory contexts $(-)\nCalg\to \Fun(-,\nCalg)_{|\Poset_{(*)}} $.
\end{proof}

\subsection{Left exact $P$-$C^{*}$-algebras}
A morphism between posets is called partially  left-exact if it preserves all  non-empty  finite meets.
So "partially" means that we do not require the preservation of the maximal element. 
We let $\Poset_{(*)}^{\puc}$ denote the wide subcategory of $\Poset_{(*)}$ consisting of partially  left-exact morphisms. A morphism in $\Poset_{(*)}^{\puc}$ is
left-exact if it also preserves  empty meets, i.e., maximal elements.

Assume that $P$ is a pointed poset. \begin{ddd}
\mbox{}
\begin{enumerate}
\item
A   $P$-$C^{*}$-algebra is called left-exact  if the functor $A:P\to I(A(\infty_{P}))$ is left-exact.   \item
We let $P^{\uc}\nCalg$ denote the full subcategory of $P\nCalg$ consisting of left-exact $P$-$C^{*}$-algebras.
\end{enumerate}
\end{ddd}

 
 \begin{prop}\label{phertgertgetr}
 For every pointed poset $P$ we have a left Bousfield localization
 \begin{equation}\label{fqwfqwedqewdqw}
 L:P\nCalg\leftrightarrows P^{\uc}\nCalg:\incl\ .
\end{equation}
The category $P^{\uc}\nCalg$ is presentable, and the inclusion $\incl$ is $E$-admissible.  \end{prop}
 \begin{proof}
 We first show that $P^{\uc}\nCalg$ is closed in $P\nCalg$ under filtered colimits.
 We use that filtered colimits of exact sequences are exact (see \cref{hpzhjetophherth}). Let $(A_{i})_{i\in I}$ be a filtered family in $P^{\uc}\nCalg$ with colimit $A$ in $P\nCalg$. 
We must show that $A$ is again left exact.  Let $(p_{f})_{f\in F}$ be a finite family in $P$  such that $\bigwedge_{f\in F} p_{f} $ exists. Then  we have the system of exact sequences
$$0\to \bigcap_{f\in F} A_{i}(p_{f})  \to A_{i}(p)\to A_{i}(p)/A_{i}(\bigwedge_{f\in F} p_{f}) \to 0$$
since the $A_{i}$ are left-exact. Hence
$$0\to \colim_{i\in I} \bigcap_{f\in F} A_{i}(p_{f}) \to A(p)\to A(p)/ A(\bigwedge_{f\in F}p_{f}) \to 0\ .$$
We now use  that  $$ \colim_{i\in I}\bigcap_{f\in F} A_{i}(p_{f}) \cong \bigcap_{f\in F} A(p_{f})   \ .$$
To see this, we note that we  can realize the filtered colimit inside $A$.  Then
 the relation $\subseteq$ is clear.  Every element of $\bigcap_{f\in F} A(p_{f})  $ can be approximated by   products
 $\prod_{f\in F} a_{f}$ with $a_{f}\in A(p_{f})$.
 For fixed $f$ the element $a_{f}$ can be approximated by the image of an element $a_{i,f}$ in $A_{i}(p_{f})$.
Since $F$ is finite, we can take a common $i$ and then 
 $\prod_{f\in F} a_{i,f}  \in \bigcap_{f\in F} A_{i}(p_{f})$.
 This shows that $a$ belongs to $\colim_{i\in I} \bigcap_{f\in F} A_{i}(p_{f})  $.
  We conclude that
 $$A( \bigwedge_{f\in F}p_{f}) \cong  \bigcap_{f\in F} A(p_{f})   \ .$$

 Note that limits in $ P\nCalg$ are created in $\Fun(P,\nCalg)$ and are therefore taken pointwise.
Since left-exactness is  a condition   which can be phrased in terms of limits  involving the values of the functors,
this condition is preserved under limits. Consequently, $P^{\uc}\nCalg$ is closed in $P\nCalg$ under limits.

We can now apply the adjoint functor theorem and obtain the existence of the left-adjoint $L$.
Since $P^{\uc}\nCalg$ is an accessible left Bousfield localization
of a presentable category it is itself presentable.

In order to show that $\incl$ is $E$-admissibe it remains to show that it preserves finite products and exact sequences. 
Being a right adjoint,   it preserves finite products and  kernels;  we must show that it also preserves quotients.  To this end, 
it suffices to show that $P^{\uc}\nCalg$ is closed in $P\nCalg$ under quotients.   We consider a push-out $$\xymatrix{A\ar[r]\ar[d] &B \ar[d] \\ 0\ar[r] & C}$$ in $P\nCalg$ and assume that $A$ and $B$ are in $P^{\uc}\nCalg$. We must show that then, $C$ is also  in $P^{\uc}\nCalg$. We can replace $A$ by the kernel of $B\to C$.
%
Let $(p_{f})_{f\in F}$ be a finite family in $P$  such that $\bigwedge_{f\in F} p_{f} $ exists.
Then
\begin{align*} \bigcap_{f\in F}C(p_{f})  = \bigcap_{f\in F} \frac{B(p_{f})}{A(p_{f})} \stackrel{!}{=} \frac{ \bigcap_{f\in F} B(p_{f})}{  \bigcap_{f\in F} A(p_{f})}
\stackrel{!!}{=} \frac{B(\bigwedge_{f\in F}p_{f})}{  A(\bigwedge_{f\in F}p_{f})}=C( \bigwedge_{f\in F}p_{f})\ , \end{align*}where the  equality marked by $!$ is due to \cref{khoprtehegrtege9} and $!!$ follows from the left-exactness of $A$ and $B$.
 \end{proof}
 
 \begin{rem}
 One can check that the left Bousfield localization \eqref{fqwfqwedqewdqw} and $\incl$ are symmetric monoidal.  \hB \end{rem}

\begin{prop} \label{kophrthertgertgtrge} We have a  lax symmetric monoidal functor
 \begin{equation}\label{bsdfvdfvfst5t5t5tdvsdf}  (-)^{\uc}\nCalg :\Poset_{(*)}^{\puc}\to \Cat\end{equation} 
  to $E$-theory contexts, with the structure  of a subfunctor of $(-)\nCalg_{|\Poset_{(*)}^{\puc}}$.
\end{prop}
\begin{proof}
 If $f:P\to Q$ is a morphism in $\Poset^{\puc}_{(*)}$, then
 $f^{*}$ sends $A$ in $Q^{\uc}\nCalg$ to an object of   
  $P^{\uc}\nCalg$ since $f$ preserves non-empty meets.
   
 For $A$ in $P^{\uc}\nCalg$ and $B$ in $Q^{\uc}\nCalg$ one checks using \cref{kohpertherggetrg} that
 the tensor product
 $A\boxtimes B$   belongs to $(P\times Q)^{\uc}\nCalg$.
  
 By \cref{phertgertgetr} the category $P^{\uc}\nCalg$ is closed in $ P\nCalg$ under   finite products and   filtered colimits,
and  the inclusion $P^{\uc}\nCalg\to P\nCalg$ is $E$-admissible.
    
It therefore follows from \cref{koprtzjrtzhrzthr} that $(-)^{\uc}\nCalg$ is a  lax symmetric monoidal functor to $E$-theory contexts, and that the inclusion $(-)^{\uc}\nCalg\to (-)\nCalg_{|\Poset_{(*)}^{\puc}}$ is a natural transformation between lax symmetric monoidal functors to $E$-theory contexts.
\end{proof}

\begin{rem}
If $f:P\to Q$ in $\Poset_{(*)}$ is partially left-exact, then $f(\infty_{P})$ is not necessarily a maximal element of $Q$.
For $A$ in $Q^{\uc}\nCalg$ the underlying $C^{*}$-algebra of $f^{*}A$ is therefore $A(f(\infty_{P}))$, which is an ideal in $A$. However,  if $f$ is left-exact, then $f^{*}$
preserves the underlying $C^{*}$-algebras. \hB
\end{rem}

\subsection{Regular $P$-$C^{*}$-algebras}

 A map between posets is a called preframe morphism if it preserves finite meets and filtered joins.
If it only preserves non-empty finite meets, then we add the prefix  "partial". 
 We let $\Poset_{(*)}^{\papr}$ denote the symmetric monoidal category of  pointed posets  and partial 
 preframe morphisms.

Let $P$ be  a pointed poset.
\begin{ddd}
\mbox{}
\begin{enumerate}
\item A left exact $P$-$C^{*}$-algebra   $A$   is called regular if the functor $A:P\to I(A)$  is a  preframe morphism.\item We let $P^{\reg}\nCalg$ denote the full subcategory of $P^{\uc}\nCalg$ consisting of regular  $P$-$C^{*}$-algebras.
\end{enumerate}
\end{ddd}

\begin{prop}\label{gwerwkeropfweferf}
The subcategory $P^{\reg}\nCalg$ of $P^{\uc}\nCalg$ is closed under finite products, filtered colimits, kernels, and quotients, and 
the inclusion $P^{\reg}\nCalg \to P^{\uc}\nCalg$  is $E$-admissible.
\end{prop}
\begin{proof}
Since regularity is formulated as a condition involving filtered colimits of the evaluations, and since
finite products, filtered colimits, and quotients  in $P^{\uc}\nCalg$ are formed pointwise, 
we can conclude that the subcategory $ P^{\reg}\nCalg$ of $P^{\uc}\nCalg$ is closed under finite products, filtered colimits,   and quotients.
  For the kernels, we consider a pullback $$\xymatrix{A\ar[r]\ar[d] &B \ar[d] \\ 0\ar[r] & C} $$ in $P^{\uc}\nCalg$ and assume that
  $B$ and $C$ belong to $P^{\reg}\nCalg$.
  We must show that then,  $A$ also belongs to $P^{\reg}\nCalg$. We will use   the distributivity law in the frame $I(B)$.
  Let $(p_{i})_{\in I}$ be a filtered family in $P$  such that $p:=\bigvee_{i\in I}p_{i}$ exists.
  Then $A(p_{i})$ for every $i$ in $I$ is  also an ideal in $B$, and
  $$A(p)=A\cap B(p)=A\cap \bigvee_{i\in I}B(p_{i})=  \bigvee_{i\in I}(A\cap B(p_{i}))= \bigvee_{i\in I} A(p_{i})\ .$$
  \end{proof}

\begin{prop} \label{okpherhtrgertgetrg}We have a  lax symmetric monoidal functor
 \begin{equation}\label{bsdfvdfvfst5t5t5tdvsdf}  (-)^{\reg}\nCalg :\Poset_{(*)}^{\papr}\to \Cat\end{equation} 
  to $E$-theory contexts, with the structure of a  subfunctor of $(-)^{\uc}\nCalg_{|\Poset_{(*)}^{\papr}}$.
\end{prop}
\begin{proof}
 If $f:P\to Q$ is a morphism in $\Poset^{\papr}_{(*)}$, then
 $f^{*}$ sends $A$ in $Q^{\reg}\nCalg$ to  an object of  $P^{\reg}\nCalg$ since $f$ preserves non-empty meets and filtered joins.
   
 For $A$ in $P^{\reg}\nCalg$ and $B$ in $Q^{\reg}\nCalg$ one checks using    \cref{jgwoiergwergwreg} that
 the tensor product
 $A\boxtimes B$   belongs to $(P\times Q)^{\reg}\nCalg$.
  
 By \cref{gwerwkeropfweferf} the category $P^{\reg}\nCalg$ is closed in $ P^{\uc}\nCalg$ under   finite products, and   filtered colimits, and the 
  inclusion $P^{\reg}\nCalg \to  P^{\uc}\nCalg$ is $E$-admissible.
  
It therefore follows from \cref{kophrthertgertgtrge} that $(-)^{\reg}\nCalg$ is a  lax symmetric monoidal functor to $E$-theory contexts, and that  the inclusion $(-)^{\reg}\nCalg\to (-)^{\uc}\nCalg_{|\Poset_{(*)}^{\papr}}$ is a natural transformation between lax symmetric monoidal functors to $E$-theory contexts.
\end{proof}

Under additional conditions on the poset $P$  we can present the category of  regular $P$-$C^{*}$-algebras as a right Bousfield localization of the category of left-exact ones and show that $P^{\reg}\nCalg$ is presentable. In the following we explain the details.

Let $P$ be a poset. We recall the following definitions from the theory of posets (see  \cite[Sec. 3]{Lehner:2026aa} for more background).
\begin{enumerate}\item For $p,q$ in $P$ we say that $p\ll q$ (way below) if  for every filtered subset $R$ in $P$ with $q=\bigvee R$ there exists  $r$ in $R$ such that $p\le r$.
\item $P$ is called  a preframe  if it admits finite meets and all filtered joins, and if meets distribute over  filtered joins.  
 \item  A    preframe $P$ is locally compact if for every $q$ in $P$ we have $$q=\bigvee \{p\in P\mid p\ll q\}\ . $$
\item A     preframe  $P$  is stably locally compact if $p\ll q$ and $p\ll r$ imply $p\ll q\wedge r$.
\end{enumerate}

Let $P$ be a poset 
\begin{ddd}\label{khortetrggtgeg}
We define a functor
$R:P\nCalg\to P\nCalg$  given on objects by 
$$A\mapsto \left(q\mapsto R(A)(q):=\bigvee_{p\ll q} A(p)\right)$$ for all $q$ in $P$, and on morphisms in the canonical manner.
\end{ddd}

We have a natural transformation
$$r:R\to \id:P\nCalg\to P\nCalg$$
induced by the obvious inclusions $R(A)(q)\subseteq A(q)$ of ideals for all $q$ in $P$.
 The functor $R$ in \eqref{gwerfwerfwerfwerferf2} below is the restriction of the functor from \cref{khortetrggtgeg}
 to $P^{\uc}\nCalg$, and implicitly we claim that it then takes values in $P^{\reg}\nCalg$.
  
    \begin{theorem}\label{jioretherthertgetrgertg}
  We assume that     $P$ is a stably locally compact preframe.
   \begin{enumerate} \item  We   have a right Bousfield localization
\begin{equation}\label{gwerfwerfwerfwerferf2}
\incl:P^{\reg}\nCalg\leftrightarrows P^{\uc}\nCalg: R\ .
\end{equation}
\item 
The functor $R$ in \eqref{gwerfwerfwerfwerferf2} is $E$-admissible.
\item The category 
 $P^{\reg}\nCalg$ is presentable.
 \end{enumerate}
   \end{theorem}
   \begin{proof}
 We start by showing that $R$ sends $P^{\uc}\nCalg$ to $P^{\reg}\nCalg$. 
 Let $A$ be in $P^{\uc}\nCalg$. We first show that $R(A)$ is again left-exact.
 Let $p,q$ be in $P$. 
 Then we calculate
    \begin{eqnarray*}
    R(A)(p\wedge q)&= &\bigvee_{r\ll p\wedge q} A(r)\\ &\stackrel{!}{=}&
      \bigvee_{r\ll p \land s\ll q} A(r)\wedge A(s)\\&=&
        \bigvee_{r\ll p} A(r)\wedge  \bigvee_{ s\ll q} A(s)\\&=&
        R(A)(p)\wedge R(A)(q)\end{eqnarray*}
   where we used the facts that $P$ is stably locally compact and $A$ is left-exact  at the equality marked by $!$.   In order to see that $R(A)$ is regular, we calculate  \begin{eqnarray*}
 \bigvee_{p\ll q}R(A)(p)
&=& \bigvee_{p\ll q}  \bigvee_{r\ll p}A(r)\\
 &=& \bigvee_{r\ll q}A(r)
\end{eqnarray*}
since $(r)_{\exists p\mid r\ll p\ll q}$ is cofinal in $(r)_{r\ll q}$.
 
  We now provide the unit and the counit of the adjunction \eqref{gwerfwerfwerfwerferf2}.
We have a  natural transformation
$r:\incl \circ R\to \id$   induced by the  inclusions $R(A)(p)\to A(p)$.
We furthermore have an isomorphism
$R\circ \incl\cong \id$ since in this case the inclusions are equalities.

 Since $R$ is a right-adjoint, it preserves finite products and kernels.
We next show that $R$ preserves  filtered colimits and quotients.
 Filtered colimits in $P^{\uc}\nCalg$ and $P^{\reg}\nCalg$ are calculated pointwise.
 Since $R(A)(p)$ is given by a colimit of values of $A$, and since colimits commute with colimits,
we conclude that $R$ commutes with filtered colimits.

We  consider a   push-out diagram $$\xymatrix{A\ar[r]\ar[d] &B \ar[d] \\ 0\ar[r] & C} $$ in $P^{\uc}\nCalg$.
We can assume that $A$ is the kernel of $B\to C$.
We must show that $$\xymatrix{R(A)\ar[r]\ar[d] &R(B) \ar[d] \\ 0\ar[r] & R(C)} $$  is again a   push-out  diagram.
 Since quotients  in  $P^{\uc}\nCalg$ and $P^{\reg}\nCalg$ are calculated pointwise,
 it suffices to show that for every $q$ in $P$
 
  $$\xymatrix{R(A)(q)\ar[r]\ar[d] &R(B)(q) \ar[d] \\ 0\ar[r] & R(C)(q)} $$ is a push-out in $\nCalg$.   Let $c$ be in $R(C)(q)= \bigvee_{p\ll q}C(p)$.
Then then we can choose $c'$ in $  \bigcup_{p\ll q}C(p)$ approximating $c$.
We can then find a lift $b'$ in $ \bigcup_{p\ll q}B(p)$ of $c'$.  
It follows that $R(B)(q)\to R(C)(q)$ has a dense range, and is therefore surjective. 

Assume now that $b$   in $ R(B)(q)= \bigvee_{p\ll q}B(p)$
is sent to zero in $R(C)(q)$. Then $b$ in $A$.   
We can approximate $b$ by elements $b'$ in $\bigcup_{p\ll q}B(p)$.
Then also $ub'$ approximates $b$ for   suitable members $u$ of an approximate unit of $A$.
Then $ub'\in A(p)$ for some $p\ll q$. 
It follows that $b\in R(A)(q)$.

 We finally show that $P^{\reg}\nCalg$ is presentable. By \cref{phertgertgetr} there exists 
 $\kappa$ such that $P^{\uc}\nCalg$ is $\kappa$-presentable. An object of $P^{\reg}\nCalg$
which is $\kappa$-compact in $P^{\uc}\nCalg$ is also $\kappa$-compact in $P^{\reg}\nCalg$ since $\incl$ preserves filtered colimits, so in particular $\kappa$-filtered ones.
Since $R$ preserves filtered colimits, $\incl$ preserves $\kappa$-compact objects. Hence the
$\kappa$-compact objects  in $P^{\reg}\nCalg$ are precisely the $\kappa$-compact objects in $P^{\uc}\nCalg$ which are in addition regular. If $A$ is in $ P^{\reg}\nCalg$, then $\incl A\cong \colim_{i} A_{i}$ in $P^{\uc}\nCalg$ for a system of $\kappa$-compact objects  $A_{i}$ in $P^{\uc}\nCalg$.
Then also $A\cong \colim_{i} R(A_{i})$ in $ P^{\reg}\nCalg$. It remains to see that 
  $R$ preserves $\kappa$-compact objects. Using the explicit description of $R$, if we take 
  $\kappa$ 
  bigger than the cardinality of $P$, then this is the case. 
     \end{proof}

\subsection{Almost continuous $P$-$C^{*}$-algebras}

A morphism between posets is called an almost partial frame morphism if it preserves all non-empty finite meets and all non-empty joins.  The word "almost" thus indicates that we do not require the preservation of the minimal element. If it preserves all finite meets, then it is called an almost  frame morphism. We let $\Poset_{(*)}^{\pafr }$ denote the category of posets and partial almost  frame  morphisms.

Let $P$ be  a pointed poset.
\begin{ddd}
\mbox{}
\begin{enumerate}
\item A regular  $P$-$C^{*}$-algebra   $A$   is called  almost continuous if the functor $A:P\to I(A)$  is an almost  frame morphism.
\item We let $P^{\ac}\nCalg$ denote the full subcategory of $P^{\reg}\nCalg$ consisting  of almost continuous   $P$-$C^{*}$-algebras.
\end{enumerate}
\end{ddd}

  Let $P$ be a pointed poset.  
  \begin{lem}\label{kohperthertgertgetrg} 
  The subcategory $P^{\ac}\nCalg$ is closed in $P^{\reg}\nCalg$ under  finite products, filtered colimits, kernels, and  quotients, and the inclusion
$ P^{\ac}\nCalg\to  P^{\reg}\nCalg$ is $E$-admissible.     \end{lem}
\begin{proof}  
Let $(A_{i})_{i\in I}$ be a filtered system in $P^{\ac}\nCalg$ and $A:=\colim_{i\in I}A_{i}$ in $P^{\reg}\nCalg$. We must show that $A$ is again almost continuous.
 Let $(p_{k })_{k\in K}$ be a family in $P$ whose join exists. It is clear, that
  $$\bigvee_{k\in K}A(p_{k})\subseteq A(\bigvee_{k\in K}p_{k})$$ and we must show the converse inclusion.
Consider $a$   in $A(\bigvee_{k\in K}p_{k})$. 
Then 
  $a$ can be approximated by images of elements in $A_{i}( \bigvee_{k}p_{k})$.
The latter are in $\bigvee_{k\in K} A_{i}(p_{k})$, and hence the approximants belong to $ \bigvee_{k\in K}A(p_{k})$.  Consequently,
$a$ is in  $ \bigvee_{k\in K}A(p_{k})$.

 Since we already know that quotients in $P^{\reg}\nCalg$ are formed pointwise and the additional condition for  being in  $P^{\ac}\nCalg$  is the preservation of  finite non-empty joins, i.e., some finite colimits, 
 it is clear that    $P^{\ac}\nCalg$ is closed in $P^{\reg}\nCalg$ under
 quotients. 
 
 The closedness of $P^{\ac}\nCalg$  in $P^{\reg}\nCalg$  under finite products is clear. 
 
 We now consider kernels.
      Assume  that $$\xymatrix{A\ar[r]\ar[d] &B \ar[d] \\ 0\ar[r] & C} $$ is a  pullback in $P^{\reg}\nCalg$
 such that
$B$    belongs  to $P^{\ac}\nCalg$. For the present purpose, we  could assume  that $C$ belongs to $P^{\ac}\nCalg$, but we do not need to do this.  
Let $(p_{i})_{i\in I}$ be a family in $P$ such that $p:=\bigvee_{i\in I}p_{i}$ exists. 
Then $A(p_{i})$ for every $i$ in $I$ is  also an ideal in $B$, and
  $$A(p)=A\cap B(p)=A\cap \bigvee_{i\in I}B(p_{i})=  \bigvee_{i\in I}(A\cap B(p_{i}))= \bigvee_{i\in I} A(p_{i})\ ,$$
  where we use that $I(A)$ is a frame and hence distributive.
%
   \end{proof}

\begin{prop} \label{koperhtetrhrgetr}We have a   functor
 \begin{equation}\label{bsdfvdfvfst5t5t5tdvsdf}  (-)^{\ac}\nCalg :\Poset_{(*)}^{\pafr}\to \EC\end{equation} 
  to $E$-theory contexts,   with the structure of  a subfunctor of  $(-)^{\reg}\nCalg_{|\Poset_{(*)}^{\pafr}}$.
\end{prop}
\begin{proof}
Note that we do not claim that this functor is lax symmetric monoidal. 

 If $f:P\to Q$ is a morphism in $\Poset^{\pac}_{(*)}$, then
 $f^{*}$ sends $A$ in $Q^{\ac}\nCalg$ to an object of $P^{\ac}\nCalg$ since $f$ preserves all non-empty finite meets and   non-empty joins.
   
 For $A$ in $P^{\ac}\nCalg$ and $B$ in $\nCalg$ one checks using \cref{jgwoiergwergwreg} that
 the tensor product
 $A\boxtimes B$    belongs to $P^{\ac}\nCalg$.
The functor $\cA$ takes values in $\Mod_{\nCalg}(\Cat)$, so, in particular, also in $\Mod_{\cN}(\Cat)$.
  
  By \cref{kohperthertgertgetrg}
 the category $P^{\ac}\nCalg$ is closed in $ P^{\reg}\nCalg$ under    finite products, and   filtered colimits, and
 the inclusion $P^{\ac}\nCalg \to  P^{\reg}\nCalg$ is $E$-admissible.  It therefore follows from \cref{okpherhtrgertgetrg} that $(-)^{\ac}\nCalg$ is a   functor to $E$-theory contexts, and that the inclusion
$(-)^{\ac}\nCalg\to (-)^{\reg}\nCalg_{|\Poset_{(*)}^{\pafr}}$ is a natural transformation between functors to $E$-theory contexts.
\end{proof}

\subsection{Continuous $P$-$C^{*}$-algebras}
A morphism between posets is called a partial frame morphism if it preserves non-empty   finite meets and all joins.
We let $\Poset^{\pfr}_{(*)}$ denote the category of pointed posets and partial frame morphisms.

\begin{ex}
Let $X$ be a topological space and $U$ be an open subset. Then, the inclusion $\Open(U)\to \Open(X)$ is a partial frame morphism. It corresponds to the partially defined map $X\supseteq U\to U$. This is the motivation for the word partial. Partially defined maps 
between locally compact Hausdorff spaces occur, for example, in the description of the Gelfand dual of a map between commutative $C^{*}$-algebras. \hB
\end{ex}

Let $P$ be  a pointed poset.
\begin{ddd}\label{zehjherthertherth9}
\mbox{}
\begin{enumerate}
\item An  almost continuous  $P$-$C^{*}$-algebra   $A$   is called  continuous if the functor $A:P\to I(A)$  is  a frame morphism. \item We let $P^{\rmc}\nCalg$ denote the full subcategory of $P^{\ac}\nCalg$ consisting  of continuous  $P$-$C^{*}$-algebras.
\end{enumerate}
\end{ddd}

Compared with almost continuity, the only new condition is that $A(0_{P})=0$. 
   \begin{lem}\label{kohperthertgertgetrg1} 
  The subcategory $P^{\rmc}\nCalg$ is closed in $P^{\ac}\nCalg$ under  finite products, filtered colimits, kernels, and  quotients, and the inclusion
$ P^{\rmc}\nCalg\to  P^{\ac}\nCalg$ is $E$-admissible.   \end{lem}
\begin{proof}
The assertions are straightforward to check.
\end{proof}

 \begin{prop} \label{gwererwfwerfwrf}We have a   functor
 \begin{equation}\label{bsdfvdfvfst5t5t5tdvsdf}  (-)^{\rmc}\nCalg :\Poset_{(*)}^{\pfr}\to \EC\end{equation} 
  to $E$-theory contexts with the structure of  a subfunctor of $(-)^{\ac}\nCalg_{|\Poset_{(*)}^{\pfr}}$.

\end{prop}
\begin{proof}
  If $f:P\to Q$ is a morphism in $\Poset^{\pfr}_{(*)}$, then
 $f^{*}$ sends $A$ in $Q^{\rmc}\nCalg$  to $P^{\rmc}\nCalg$ since $f$ preserves non-empty meets and   all  joins.
   
 For $A$ in $P^{\rmc}\nCalg$ and $B$ in $\nCalg$,  
 the tensor product
 $A\boxtimes B$     belongs to $P^{\rmc}\nCalg$.
  
By \cref{kohperthertgertgetrg1} the category $P^{\rmc}\nCalg$ is closed in $ P^{\ac}\nCalg$ under   finite products and   filtered colimits, and
 the inclusion $P^{\rmc}\nCalg \to  P^{\ac}\nCalg$  is $E$-admissible.    It therefore follows from \cref{koperhtetrhrgetr} that $(-)^{\rmc}\nCalg$ is a   functor to $E$-theory contexts, and that the inclusion
  $(-)^{\rmc}\nCalg\to (-)^{\ac}\nCalg_{|\Poset_{(*)}^{\pfr}}$
  is a   natural transformation between functors to $E$-theory contexts.
\end{proof}

\subsection{The $E$-theory functor for locales and topological spaces}

 Recall that a frame is a poset, which admits finite meets and all joins, and in which meets distribute over joins. A morphism between frames is a morphism of posets which preserves finite meets and all joins.
 The category of locales is defined by $$\Locale:=\Frame^{\op}\ .$$ 
 \begin{ex} The typical example of a frame is the poset of open subsets of a topological space $\Open(X)$. A continuous map $f:X\to Y$ between topological spaces  
 induces a frame morphism $f^{-1}:\Open(Y)\to \Open(X)$. For us, the motivation to introduce locales is the desire to 
 use  a uniform symbolic language for topological spaces and frames. \hB
 \end{ex}
 
 For a locale $X$ we let $\cP(X)$ denote the corresponding poset.
 \begin{ddd} \label{lkohperghtrgertgertg9} We define the functor  
 $$(-)\nCalg:=(\cP(-))^{\rmc}\nCalg:\Locale\to \EC\ , $$ where for   a map of locales 
 $f:X\to Y$, we use the notation $$f_{!}:=\cP(f)^{*}:X\nCalg\to Y\nCalg\ .$$ \end{ddd}

\begin{ddd}\label{ohkperrtgertgertg}
We define the $E$-theory  functor for locales  by  $$E:\Locale\stackrel{(-)\nCalg}{\to} \EC\stackrel{\EE}{\to} \Pr^{L}_{\st}\ , \quad X\mapsto E(X):=E(X\nCalg)\ , \quad f\mapsto f_{!}:=E(f_{!})\ .$$
\end{ddd}

We have a functor $\Top\to \Locale$ which sends a topological space $X$ to the locale $\Open(X)$ and a map $f:X\to Y$
of topological spaces to the map of frames $f^{-1}:\Open(Y)\to \Open(X)$.
\begin{ddd}For a topological space $X$ we write $$E(X):=\Open(X)^{c}\nCalg\ .$$
and for a morphism  
$f:X\to Y$ between topological spaces, we  set
$f_{!}:E(X)\to E(Y)$.\end{ddd}

  \begin{ex} \label{kohperrgrgrgretge} A primitive ideal in a  $C^{*}$-algebra $A$ is the kernel of an irreducible representation
$A \to B(H)$, where $B(H)$ is the $C^{*}$-algebra of bounded operators on some Hilbert space $H$. We let $\Prim(A)$ denote the topological space of primitive ideals, which is equipped with the
Jacobson topology \cite[3.1.1]{zbMATH04073740}. The closed subsets of $\Prim(A)$ in this topology are the sets $\{p\in \Prim(A)\mid  I\subseteq p\}$
for the closed ideals $I$ of $A$. It follows that
there is an isomorphism of frames,  $\Open(\Prim(A))\cong  I(A)$ which sends an open
 set $U $ to the closed ideal $A(U):=\bigcap_{p\in \Prim(A)\setminus U} p$.   We conclude that $$\Prim(A)\nCalg\simeq  I(A)^{\rmc}\nCalg\ ,$$
 where $\Prim(A)$ on the left-hand side is considered as a topological space or locale, and $I(A)$ on the right-hand side is a pointed poset.
 \hB
 \end{ex}

 \section{Adjunctions, localization and cosheaves}\label{kophtzhetrgtrge}

 \subsection{Continuous and almost continuous $P$-$C^{*}$-algebras}
 
Recall that a morphism $f:P\to Q$ between posets is a partial almost frame morphism if it preserves all non-empty 
finite meets and all non-empty finite joins. It is a partial frame morphism, if it, in addition, preserves the empty join.
By \cref{gwererwfwerfwrf} we have the  functor $(-)^{\rmc}\nCalg$ to $E$-theory contexts defined on
the category $\Poset_{(*)}^{\pfr}$  of pointed posets and partial frame morphisms.
 In the following,  we extend  its functoriality   to the category  $\Poset_{(*)}^{\pafr}$ of pointed posets and 
 partial almost frame morphisms.
\begin{prop}
 \label{kohpertghergrtgreg}
\mbox{}
\begin{enumerate}
\item\label{khopzrjtjrh} For every $P$ in $\Poset_{(*)}$ having a minimal element, we have a  left Bousfield localization \begin{equation}\label{goijweoifwerfwerfw}L:P^{\ac}\nCalg\leftrightarrows P^{\rmc}\nCalg:\incl\ .
\end{equation}
 \item  \label{hkpertgretggetg}For all  $P,Q$ in $ \Poset_{(*)}$ and every minimal element preserving morphism $f:P\to Q$ in $\Poset_{(*)}^{\pfr}$   the square
\begin{equation}\label{vifjiowvfevsf}\xymatrix{ Q^{\rmc}\nCalg\ar[r]^{\incl}\ar[d]^{f^{*}} & Q^{\ac}\nCalg \ar[d]^{f^{*}} \\ P^{\rmc}\nCalg \ar[r]^{\incl} &P^{\ac}\nCalg }  
\end{equation} 
 is horizontally left-adjoinable.
\item\label{lphkrthertgerg} The functors  $L$ and  $\incl$ in  \eqref{goijweoifwerfwerfw}  are    morphisms of $E$-theory contexts.
  \end{enumerate}

\end{prop}
\begin{proof}
Let $0_{P}$ denote the minimal element in $P$.
The functor $L$ sends $A$ in $P^{\ac}\nCalg$ to the functor
$$L(A):p\mapsto A(p)/A(0_{P})\ .$$
One checks that it belongs to $P^{\rmc}\nCalg$.
The unit of the  left Bousfield localization is induced by the canonical projections $A(p)\to A(p)/A(0_{P})$ for all $p$ in $P$.
If $A$ is in $P^{\rmc}\nCalg$, then $A(0_{P})=0$ and the counit is the isomorphism $L(\incl(A))\cong A$.
The triangle identities are straightforward to check. This shows Assertion \ref{khopzrjtjrh}.

Since $f$ in Assertion  \ref{hkpertgretggetg} preserves minimal elements, we have
$$(f^{*}A)(p)/(f^{*}A)(0_{P})\cong A(f(p))/A(f(0_{P}))\cong L(f^{*}A)(p)\ .$$
This implies that the  square in Assertion  \ref{hkpertgretggetg} is horizontally left adjointable.

We finally show  Assertion \ref{lphkrthertgerg}. In order to show that $L$ is $E$-admissible
 we use the fact  that finite products, and filtered colimits, and exact sequences in $P^{\rmc}\nCalg$ and $P^{\ac}\nCalg$ are formed pointwise. 
If $(A_{i})_{i\in I}$ is a filtered system in $P^{\ac}\nCalg$, then
for every $p$ in $P$ the colimit of the exact sequences 
$$0\to A_{i}(0_{P}) \to A_{i}(p)\to  A_{i}(p)/A_{i}(0_{P})\to 0$$
is the exact sequence
$$0\to A(0_{P}) \to A(p)\to  A(p)/A(0_{P})\to 0\ .$$
Looking at the third terms, we conclude that 
$$\colim_{i\in I} L(A)\cong L(\colim_{i\in I} A_{i})\ .$$
The argument for finite products is similar.
In order to see that $L$ preserves exact sequences,  assume that   $$0\to A\to B\to C\to 0$$ is exact in $P^{\ac}\nCalg$.
Then for $p$ in $P$ we get a web of vertical exact sequences
$$ \xymatrix{&0\ar[d]&0\ar[d]&0\ar[d]&\\
0\ar[r]&A(0_{P})\ar[r] \ar[d]&B(0_{P})\ar[d]\ar[r]&C(0_{P})\ar[d]\ar[r]&0\\
0\ar[r]&A(p)\ar[d]\ar[r] &B(p)\ar[r]\ar[d]&C(p)\ar[r]\ar[d]&0\\
0\ar[r] &A(p)/A(0_{P})\ar[r] \ar[d]&B(p)/B(0_{P})\ar[d]\ar[r]\ar[d]&C(p)/C(0_{P})\ar[d]\ar[r]&0\\
&0&0&0&}\ .$$
Since the upper two horizontal sequences are exact by assumption,  the lower horizontal sequence is also exact.

   In order to see that $L$ is a morphism of $\cN$-modules we use that for $N$ in $\cN$ and $p$ in $P$ we have the relations
$$L(A)(p)\otimes N\cong (A(p)/A(0_{P}))\otimes N\cong (A(p)\otimes N)/(A(0_{P})\otimes N)\cong L(A\otimes N)(p)\ .$$

It is straightforward to see that $\incl$ in \eqref{goijweoifwerfwerfw} is a morphism of $E$-theory contexts.
\end{proof}

\begin{prop} \label{kophertrtgertgtrge}We have an extension of functors
$$\xymatrix{\Poset^{\pfr}_{(*)}\ar[dr]_{\incl}\ar[rr]^{(-)^{\rmc}\nCalg} &&\EC\\&\Poset^{\pafr}_{(*)}\ar@{..>}[ur]_{(-)^{\rmc}\nCalg}&}\ .$$
 \end{prop}
 \begin{proof}
 The extension is given as follows:
 \begin{enumerate}
 \item objects: It sends $P$ in $\Poset_{(*)}^{\pafr}$ to $P^{\rmc}\nCalg$.
 \item morphisms: It sends $f:P\to Q$ in $\Poset_{(*)}^{\pafr}$ to the composition
 $$Q^{\rmc}\nCalg\stackrel{f^{*}}{\to} P^{\ac}\nCalg\stackrel{L}{\to} P^{\rmc}\nCalg\ .$$
 \item Compositions: Let $g:Q\to R$ be a second morphism in $\Poset_{(*)}^{\pafr}$. We add subscripts to the symbols for  the localizations $L$.
We have a natural transformation
 $$  L_{P}f^{*}g^{*}\simeq L_{P}L_{P}f^{*} g^{*}\to  L_{P}f^{*} L_{Q}g^{*} $$ induced by the Beck-Chevalley map $L_{P}f^{*}\to f^{*}L_{Q}$ for
 the square  \eqref{vifjiowvfevsf}. The latter might not be an isomorphism since
 $f$ may not preserve the minimal elements. But one checks that the whole composition is an isomorphism.
 The pentagon relation for the composition of three morphisms is satisfied.
   \end{enumerate}%
 \end{proof}

\subsection{Adjunctions between frames}
In this section we collect some facts about adjunctions between frames for later use.
We suggest to consult \cite[Sec. 2-3]{Lehner:2026aa} for a condensed presentationof most of the definitions and facts below.

Recall that an adjunction in $\Poset$ is a property of a pair of functors.
Assume that $f:P\to Q$ is a morphism of posets.  
\begin{lem}\label{hiuqhiufhweifqwefqwefqw} If $P$ is a frame and $f$ is  a frame morphism, then we have an adjunction 
 \begin{equation}\label{jibodfgbfgbwtrg}f:P\leftrightarrows Q:f_{\sharp} 
\end{equation}
 in $\Poset^{\uc}$. If $f$ is  surjective, then $\id=f_{\sharp}f$. If $f$ is injective, then 
 $f_{\sharp}f=\id$.

\end{lem}
\begin{proof} Since $P$ is a frame
 we can set  $$f_{\sharp}(q):=\bigvee_{\{p\in P\mid f(p)\le q\}}p$$ and observe that 
 $f f_{\sharp}(q)\le q$ and $p\le f_{\sharp}  f(p)$.
  The additional statements are straightforward. 
\end{proof}

\begin{ddd}\label{herthertgergtrgtrge}
A frame morphism $f:P\to Q$ is called perfect if $f_{\sharp}$ preserves filtered joins.
\end{ddd}

Let $\Poset^{\prfr}$ be the subcategory of preframe morphisms.
\begin{kor}\label{tklphertrtgergrgret}
If $f:P\to Q$ is perfect, then we have an adjunction   \begin{equation}\label{jibodfgbfgbwtrg}f:P\leftrightarrows Q:f_{\sharp} 
\end{equation}
in $\Poset^{\prfr}$.
\end{kor}

%
%
%
%

Recall that the category of locales is defined as the opposite of the category of frames. 
Open or closed  inclusions are usually considered for maps between locales. The following are the translations to frames.
Let $f:P\to Q$ be a frame morphism   between posets.

\begin{ddd}\label{kopzjtzjrtzhrzth}\mbox{}
\begin{enumerate}
\item \label{korpberthter}$f:P\to Q$  is a closed subframe  inclusion if there exists $p_{f}$ in $P$
such that $ f$ restricts to a frame isomorphism  $P_{p_{f}/}\stackrel{\cong}{\to} Q $.
 \item\label{korpberthter1}  $f:P\to Q$  is an   open subframe   inclusion if there exists $p_{f}$ in $P$ such that 
 $f$ restricts to a frame isomorphism $P_{/p_{f}}\stackrel{\cong}{\to} Q $.
 \end{enumerate}
\end{ddd}

\begin{rem}\label{kopjrtzjzthzhtzhrtzhtz}

Note that in these definitions, $Q$ is the subframe of $P$.

If $P$ and $Q$ are the frames of open subsets of topological spaces, then
the notions of an open or closed  inclusion become the usual ones.

The map between frames induced by a proper morphism between locally compact Hausdorff spaces  
is an example of a perfect map. \hB
\end{rem}

The proofs of the following two lemmas are straightforward.
\begin{lem}\label{ojgwpoergewrfweff}
An open subframe  inclusion $f:P\to Q$ fits into a right Bousfield localization   \begin{equation}\label{gjoiowegwefrweffw}f^{\flat}:Q\leftrightarrows P: f  
\end{equation} 
in $\Poset^{\pfr}$,
 where $ff^{\flat}(q)= q$ and $f^{\flat} $ is a partial frame morphism.
\end{lem}


\begin{kor} \label{okhpetrhkertgtrgergertge} A closed subframe inclusion $f:P\to Q $ 
fits into a left 
  Bousfield localization 
$$f :P\leftrightarrows Q:f_{\sharp} $$
in $\Poset^{\afr}$,  
where  $f_{\sharp}$ is determined by  $ f_{\sharp}(f (p)) =p\vee p_{f}$  and  is  an almost frame morphism.
 \end{kor}


\subsection{Canonical descent for $E$-theory of locales}

 We consider the functor
 $$(-)\nCalg:\Locale \to \EC\ ,  \quad X\mapsto X\nCalg\ , \quad f\mapsto f_{!}$$ from \cref{lkohperghtrgertgertg9}.
   Recall that for a locale $X$, we denote  by $\cP(X)$  the same poset considered as a frame.
 
 A map $j:U\to X$ of locales is an open sublocale inclusion if $\cP(j):\cP(X)\to \cP(U)$ is an open subframe inclusion as described in \cref{kopzjtzjrtzhrzth}.\ref{korpberthter1}.

\begin{kor}\label{okhpzrhtzjrtzj}
If $j:U\to X$ is an open  sublocale inclusion, then we have a right Bousfield localization
$$j_{!}:U\nCalg \leftrightarrows  X\nCalg:  j^{*}$$
and the right-adjoint $j^{*}$ is  a morphism between $E$-theory contexts.
\end{kor}
\begin{proof}
Here $j^{*}=\cP(j)^{\flat,*}$ with $\cP(j)^{\flat}$ from \cref{ojgwpoergewrfweff}  and $\cP(j)$ is a partial frame morphism.
We now use the fact, that the functor to $E$-theory contexts $(-)^{\rmc}\nCalg$ is defined on $\Poset_{(*)}^{\pfr}$ by \cref{gwererwfwerfwrf}.
\end{proof}
A map $i:Z\to X$ is a closed  sublocale inclusion if $\cP(i):\cP(X)\to \cP(Z)$ is a closed subframe inclusion as described in \cref{kopzjtzjrtzhrzth}.\ref{korpberthter}.

\begin{kor}\label{okhpzrhtzjrtzj1}
If $i:Z\to X$ is a closed  sublocale inclusion, then we have a left Bousfield localization
$$i^{*}: X\nCalg \leftrightarrows   Z \nCalg: i_{*}:=i_{!}$$
and the left-adjoint $i^{*}$ is   a morphism between $E$-theory contexts.
\end{kor}
\begin{proof}
Here $i^{*}:=\cP(i)_{\sharp}^{*}$, where $\cP(i)_{\sharp}$ from \cref{okhpetrhkertgtrgergertge} is an almost frame morphism, and we  use the fact, that the functor to $E$-theory contexts $(-)^{\rmc}\nCalg$ is defined on $\Poset_{(*)}^{\afr}$ by \cref{kophertrtgertgtrge}.
\end{proof}

Let $X$ be a locale and $j:U\to X$ be an open sublocale inclusion corresponding to the frame morphism
$$\cP(j):P\to P_{/p}\ , \quad   p'\mapsto p\wedge p'$$ for   some uniquely determined $p$ in $P$. The map $i:Z\to X$ given by the frame morphism  $$\cP(i):P\to P_{p/} \ ,\quad p'\mapsto p\vee p'$$
is then called the complementary closed sublocale inclusion.

Recall the $E$-theory functor 
$$E:\Locale \to \Pr_{\st}^{L}$$ from \cref{ohkperrtgertgertg}.
\begin{theorem}\label{kohpertghertgtrege}
For an open sublocale $j:U\to X$ and its complementary closed sublocale $i:Z\to X$
we have a  recollement $$ \xymatrix{  E(U )\ar@/^0.8cm/[r]_{\perp}^{j_{!}}\ar@/_0.8cm/[r]^{\perp}_{j_{*} }&\ar[l]_{j^{*}} E(X )\ar@/^0.8cm/[r]_{\perp}^{i^{*}}\ar@/_0.8cm/[r]^{\perp}_{ j^{!}}&\ar[l]^{i_{*}} E(Z )}\ .$$
\end{theorem}
\begin{proof}
The left Bousfield localization $j_{!}\dashv j^{*}$ from  \cref{okhpzrhtzjrtzj} and the right Bousfield localization
$i^{*}\dashv i_{*}$ from  \cref{okhpzrhtzjrtzj1}   induce corresponding localizations after applying the $E$-theory functor. The additional right-adjoints $j_{*}$ and $i^{!}$ exist since $E$ takes values in $\Pr_{\st}^{L}$ and sends partial or almost frame morphisms to
left-adjoint functors.

For every $A$ in $X\nCalg$ we have an exact sequence
$$0\to j_{!}j^{*}A\to A\to i_{*}i^{*}A\to 0\ .$$
 Indeed, let $p'$ be in $P$.
Then this  sequence evaluates  to 
$$0\to  A(p\wedge p')\to A(p')\to  A(p\vee p')/A(p)\to 0$$
which is exact, since
$A(p\wedge p')=A(p)\cap A(p')$ and $ A(p\vee p')=A(p)+A(p')$.
The functors
$$\Phi:A\mapsto  \ee_{X }(j_{!}j^{*}A)$$ and
$$\Psi:A\mapsto  \Fib(  \ee_{X }(A)\to  \ee_{X }(i_{*} i^{*}A))$$
from $X\nCalg$ to $E(X )$ are both  homological. The exactness of $\ee_{X}$
implies an equivalence
$\Phi\to \Psi$.  
By the universal property of $\ee_{X}$ this equivalence
  extends to an equivalence on the level of $E$-theory. More precisely,   we have a fibre sequence
of endofunctors
\begin{equation}\label{twuiotjioergergw}
 j_{!}j^{*} \to \id \to  i_{*}i^{*}
\end{equation}
of $E(X )$.
\end{proof}

\begin{theorem}  \label{okphterthertg}Let $X$ be a locale  and $(U_{i})_{i\in I}$ be a filtered system of open sublocales such  that $X=\bigcup_{i\in I} U_{i}$. Then, we have an equivalence
$$ E(X )\simeq \lim_{  i\in I} E(U_{i})\ .$$
 \end{theorem}
\begin{proof}
 Let $j_{i}:U_{i}\to X$ denote the inclusions. We are going to  check    that for every $A$ in $X\nCalg$ we have  
\begin{equation}\label{vwfedfvsdfvfdsvs} \colim_{i\in I}   j_{i,!}j_{i}^{*}A \cong A \ . \end{equation} 
Let $P=\cP(X)$ be the frame corresponding to the locale $X$.
For $p$ in $P$ we have 
$ (j_{i,!}j_{i}^{*}A)(p)=A(p\wedge p_{i})$, where $p_{i}$ is determined by $U_{i}$.
The condition $X=\bigcup_{i\in I}U_{i}$ translates to $\bigvee_{i\in I}p_{i}=\infty_{\cP(X)}$.
Then $$\colim_{i\in I}   j_{i,!}j_{i}^{*}A(p)\cong \colim_{i\in I} A(p\wedge p_{i})\cong
 A(\bigvee_{i\in I}(p\wedge p_{i}))\cong 
 A(  p\wedge \bigvee_{i\in I} p_{i} )\cong A(p)$$
 using the continuity of $A$ and distributivity of $P$.
 Since $\ee_{X}$ preserves colimits the equivalence  \eqref{vwfedfvsdfvfdsvs}  induces an equivalence of functors
 $$   \colim_{i\in I}   j_{i,!}j_{i}^{*} \ee_{X}\stackrel{\simeq}{\to} \ee_{X}:X\nCalg\to E(X)\ .$$
    By the universal property of $\ee_{X}$ 
 this equivalence   extends to
an equivalence of endofunctors
\begin{equation}\label{hguhiergwerfwerfwerfw}
\colim_{i\in I\ }   j_{i,!}j_{i}^{*} \simeq \id
\end{equation} of $E(X )$.
We conclude that the functor 
\begin{equation}\label{gkweropfwerfwerfw}E(X )\to \lim_{i\in I} E(U_{i} )
\end{equation} 
is fully faithful. We now show that it is also essentially surjective.
  For $(A_{i})_{i\in I}$ in $\lim_{i\in I } E(U_{i} )$
 we set
 $$A:=\colim_{i\in I}j_{i,!}A_{i}\ .$$
We then check that 
 $j_{k}^{*} A\simeq A_{k}$ for all $k$ in $I$.  
 To this end, we use the fact $j_{k}^{*}$ preserves colimits and the 
equivalences  $j_{k}^{*}j_{i,!}A_{i}\simeq j_{U_{k}\to U_{i}}^{*}A_{i}\simeq A_{k}$ provided that $k\le i$.
  \end{proof}

\subsection{Comparison with cosheaves}

The this section we relate the category $E(X)$ for a locale $X$ with the $\EE$-valued  cosheaves on $X$. Under finiteness assumptions on $X$ we get an equivalence. The main result is \cref{khoperthergtrge}.  

Let $X$ be  a  locale and let $\cC$ be a   presentable  stable $\infty$-category.
 \begin{ddd} \label{ohipeorthretgrtgrteg}A cosheaf on $X$ with values in $\cC$ is   
a functor $F:\cP(X)\to \cC$  such that:
\begin{enumerate}
\item $F(\emptyset)=0$
\item For every $p,q$ in $\cP(X)$, we have a  pushout
$$\xymatrix{F(p\wedge  q)\ar[r]\ar[d] & F(p) \ar[d] \\ F(q) \ar[r] & F(p\vee q)}\ . $$
\item For every filtered family $(p_{i})_{i}$ in $\cP(X)$, we 
have an equivalence
$$ \colim_{i} F(p_{i})  \stackrel{\simeq}{\to} F(\vee_{i}p_{i})\ .$$
 \end{enumerate}We let 
$\CoSh(X,\cC)$ denote the stable $\infty$-category of cosheaves on  $X$.

 \end{ddd}

A morphism of locales $f:X\to Y$ gives rise to a  colimit  preserving functor  $$f_{!}:\CoShv(X,\cC)\to \CoShv(Y,\cC)$$
by restricting the functor $\Fun(\cP(X),\cC)\to \Fun(\cP(Y),\cC)$ induced by the morphism
$\cP(f):\cP(Y)\to \cP(X)$. In the case of an open inclusion, it is fully faithful and  strongly cocontinuous and has two further right-adjoints $f^{*}$ and $f_{*}
$. In the case of a closed inclusion,   it is fully faithful and has a further left-adjoint $f^{*}$ and a right-adjoint $f^{!}$, and we set $f_{!}=f_{*}$.
These facts can be seen similarly to \cref{okhpzrhtzjrtzj} and \cref{okhpzrhtzjrtzj1}.

 We recall the following well-known fact.
 \begin{prop}\label{jiohrthetrhterhreehg93q}
Assume that $j:U\to X$ and $i:Z\to X$ are complementary inclusions of an open and a closed   sublocale.
Then we have a recollement
$$ \xymatrix{ \CoSh(U,\cC)\ar@/^0.8cm/[r]_{\perp}^{j_{!}}\ar@/_0.8cm/[r]^{\perp}_{ j_{*}}&\ar[l]_{j^{*}}\CoSh(X,\cC)\ar@/^0.8cm/[r]_{\perp}^{i^{*} }\ar@/_0.8cm/[r]^{\perp}_{ i^{!}}&\ar[l]^{i_{*}}\CoSh(Z,\cC)}\ .$$
 \end{prop}
 \begin{proof}
 The proof is very similar to the one  for \cref{kohpertghertgtrege}.
 We again use the  fibre sequences for cosheaves $F$ 
 $$ j_{!}j^{*} F\to F\to  i_{*}i^{*}F\ .$$
 
 \end{proof}
 
We further recall the following well-known continuity property of cosheaves.
 \begin{prop} \label{kopttrherthtr}Let $X$ be a locale  and $(U_{i})_{i\in I}$ be a filtered system of open sublocales such  that $X=\bigcup_{i\in I} U$. Then we have an equivalence  $$\CoShv(P ,\EE)\simeq \lim_{i\in I\ } \CoShv(U_{i} ,\EE)\ .$$
\end{prop}
\begin{proof}
This has the same proof as  \cref{okphterthertg} starting with the observation that \eqref{hguhiergwerfwerfwerfw} holds for cosheaves.
\end{proof}

\begin{lem}\label{herthertgrtger}
There exists a canonical  natural transformation of functors \begin{equation}\label{trwet42t4t3}
s:E\to \Fun(\cP(-),\EE):\Locale\to \Pr_{\st}^{L}
\end{equation}  
which sends $A$ in $E(X)$ to the cosheaf $s_{X}(A):\cP(X)\ni p\mapsto \ee(A(p))$.
\end{lem}
\begin{proof}

Let $X$ be a locale. Since the right-down composition in  $$
\xymatrix{X\nCalg\ar[r]\ar[d]^{\ee_{X }} &\Fun(\cP(X),\nCalg) \ar[d]^{\ee\circ -} \\E(X ) \ar@{..>}[r]^-{s_{X}} &\Fun(\cP(X),\EE) }$$
is homological, we get the colimit-preserving factorization  $s_{X}$.
Applying the same argument to the whole diagram over $\Locale$ we get the 
  transformation of functors  \eqref{trwet42t4t3}.
 \end{proof}
 
 Let $X$ be a locale.
 \begin{theorem}\label{gweoirgjoergeferggr}\mbox{}
   \begin{enumerate} \item  \label{jiogwrtgergwerfwerf}
   
The functor
$s_{X}$ corestricts to a colimit-preserving  functor
$$s_{X}:E(X )\to\CoSh(X,\EE)\ .$$
\item  \label{jiogwrtgergwerfwerf1}
 If $X$ is finite, then $s_{X}$ is an equivalence. \end{enumerate}
\end{theorem}
\begin{proof}
 If $A$ is in $X\nCalg$, then $\ee(A(\emptyset)) \simeq \ee(0)\simeq 0$.
Since $A$ is continuous,  for $p$ and $q$ in $\cP(X)$ we  have, using the exactness of $\ee:\nCalg\to \EE$,  
a pushout
$$\xymatrix{\ee(A(p\wedge  q))\ar[r]\ar[d] & \ee(A(p)) \ar[d] \\ \ee(A(q)) \ar[r] & \ee(A(p\vee q))}\ . $$
 Similarly, using the facts that $\ee$ preserves filtered colimits and that $A$ is continuous,
for a filtered  family $(p_{i})_{i}$  in $\cP(X)$ we have
$$\ee(A(\vee_{i}p_{i}))\simeq \colim_{i} \ee(A(p_{i}))\ .$$
 These relations show that $s_{X}(A)$ is a cosheaf for every $A$ in $X\nCalg$.
 Using the universal property of $\ee_{X }$ these
 relations extend  to all objects of $E(X ) $. We conclude that $s_{X}$ takes values in cosheaves.
This finishes the verification of Assertion \ref{jiogwrtgergwerfwerf}.

We now show Assertion \ref{jiogwrtgergwerfwerf1} by induction on the number of elements of $P$. The start of the induction is at the locale corresponding to the frame $[1]$, where the assertion is obvious.

Consider now a finite locale $X$.
Let $p$ be a minimal non-minimal element of $\cP(X)$. It determines   an open sublocale inclusion $U\to X$.
We let $Z\to X$ be the closed complement.
 
Combining \cref{jiohrthetrhterhreehg93q} and \cref{kohpertghertgtrege}
 we have a map of cofibre sequences in $\Cat^{\exa}_{\infty}$
$$\xymatrix{  E(U )  \ar[r]\ar[d]_{s_{U}}^{\simeq} & E(X ) \ar[r]\ar[d]_{s_{X}} & E(Z  )\ar[d]_{s_{Z}}^{\simeq}   \\ \CoSh(U,\EE) \ar[r]  &\CoSh(X,\EE)\ar[r]& \CoSh(Z,\EE)    } \ .$$
By the induction hypothesis, the outer vertical maps are equivalences as indicated.
This implies that the middle map is an equivalence.
\end{proof}

\begin{rem}
The problem of calculating mapping  groups  in the $KK$ or $E$-theory of $X\nCalg$  for finite locales using methods from triangulated category theory
 has been considered  in 
\cite{filtrated}, \cite{zbMATH05657129}, \cite{Bentmann_2013}. In the case of $E$-theory the following provides a complete homotopy theoretic description of the whole mapping space. 
\begin{kor}\label{kohpreherthrtge} If $X$ is finite, then  
for $A,B$ in $E(X )$ we have
$$\map_{E(X )}(A,B)\simeq \map_{\CoShv(X,\EE)}(s(A),s(B))\ .$$
\end{kor} \hB
\end{rem}
 
 \begin{ex} Every 
  finite topological space $X$ gives rise to a finite locale $\Open(X)$.
 Therefore  \cref{gweoirgjoergeferggr} and \cref{kohpreherthrtge} apply to $X$-$C^{*}$-algebras for finite topological spaces. \hB
 \end{ex}

 Combining \cref{kopttrherthtr}, \cref{okphterthertg} with \cref{gweoirgjoergeferggr} we get.
 \begin{kor}\label{khoperthergtrge}
   Let $X$ be a locale which is  a union of a filtered family of 
   finite open sublocales; then we have an equivalence
   $$s_{X}:E(X )\stackrel{\simeq}{\to} \CoShv(X,\EE)\ .$$
 \end{kor}
 
 \begin{ex} The locale corresponding to the poset  $\nat$ satisfies the assumption of \cref{khoperthergtrge}. \hB
 \end{ex}

 \section{Locally compact Hausdorff spaces}\label{kohperthertgertgertg}

\subsection{Statement of the main theorem}

%

 Let $\LCH$ denote the category of locally compact Hausdorff spaces.
 By restriction along $\LCH\to \Locale$,  the functor from \cref{lkohperghtrgertgertg9} gives rise to a functor  
 $$(-)\nCalg:\LCH\to \EC\ ,
\quad X\mapsto  X\nCalg\ , \quad  f\mapsto f_{!}\ .$$
In order to construct a six-functor formalism, we need a contravariant functor.
\begin{theorem}\label{kophejhthgerthetrh}
There is   a lax symmetric monoidal functor into $E$-theory contexts 
 $$(-)\nCalg:\LCH^{\op}\to  \Cat
\ , \quad X\mapsto  X\nCalg\ , \quad  f\mapsto f^{*}\ .$$
\end{theorem}
The proof of this theorem will be given in \cref{kophejhthgt5t5erthetrh1}.
The interaction between the covariant functoriality  with the   contravariant functoriality  and the  tensor product 
will be formulated as follows:
  \begin{theorem}\label{okpherthrtgetgetg}
The lax symmetric monoidal functor \begin{equation}\label{rvewiojvofvdfvsdfvsdfv}  (-) \nCalg:\LCH^{\op} \to \Cat  \end{equation} from \cref{kophejhthgerthetrh}
 is a three-functor formalism in the sense of \cref{kojpjrtzjzhrtzh}.
\end{theorem}
 
 The proof of this theorem is  given in \cref{okpherthrtgetgetg1}.

\subsection{Equivalent pictures of $X\nCalg$}
  
 Let $X$ be a locally compact Hausdorff space and recall that $X\nCalg$ is defined,  by combining   \cref{lkohperghtrgertgertg9} and \cref{zehjherthertherth9}, as the category of pairs $$(A,A(-):\Open(X)\to I(A))$$ such that $A(-)$ is a frame morphism.

\begin{ex}\label{okhperhrethregtrg}
We have the preferred object $C_{0}(X)_{X}$ in $X\nCalg$.  Its underlying $C^{*}$-algebra is $C_{0}(X)$ and the corresponding frame map
$\Open(X) \to I(C_{0}(X))$ sends $U$ to the ideal
$ C_{0}(U)$.    \hB
\end{ex}


The following proposition states a fact that is well-known in the literature on $C_{0}(X)$-algebras and essentially stated in   \cite{zbMATH05657129}. Since it is crucial for the understanding of the categorical properties of
$X\nCalg$ for locally compact Hausdorff spaces shown in \cref{okprherhrtgertge},
we give the  argument as a service to the reader.

  \begin{prop} \label{kophertgetrgtreeg} For a  locally compact Hausdorff space $X$, providing an  $X$-$C^{*}$-algebra
 structure on $A$ is equivalent to providing   a pair $(A,m)$ of  a $C^{*}$-algebra $A$ together with
 a surjective $A$-bilinear homomorphism
$m:C_{0}(X)\otimes A\to A$.\end{prop}
\begin{proof} Recall the topological space $\Prim(A)$ from \cref{kohperrgrgrgretge}.
In this proof, we will use the following constructions and facts:
Let $A$ be in $\nCalg$
and  $\rho:A\to B(H)$ be in $\Prim(A)$.  Then $\rho(A)'\cong \C$, where $\rho(A)'$ is the  commutant of $\rho(A)$ in $B(H)$.
Since $\rho$ is irreducible, it   is essential in the sense that $\rho(A)H$ is dense in $H$. The representation $\rho$ therefore  has a unique  extension  $\hat \rho:M(A)\to B(H)$   to the multiplier algebra $M(A)$ of $A$.
 We get a homomorphism
$$ Z(M(A)) \stackrel{\hat \rho}{\to} \rho(A)'\cong \C\ ,$$
where $Z(M(A))$ denotes the center of $M(A)$.
This induces a homomorphism
\begin{equation}\label{gerfwerfewrfrggwergewr}Z(M(A))\to C_{b}(\Prim(A))\ , \quad  z\mapsto (\rho\mapsto  \hat \rho(z))\ .
\end{equation} 
We collect the following facts for later reference.
\begin{fact}\label{gjroepgfwergwerferfw}\mbox{}{\em \begin{enumerate}
\item  \label{klphprethtrherthtr}The  Dauns-Hofmann theorem states that the homomorphism in \eqref{gerfwerfewrfrggwergewr} is an isomorphism. 
\item\label{klphprethtrherthtr1} For every $a$ in $A$ we have $\|a\|=\sup_{\rho\in \Prim(A)} \|\rho(a)\|$  \cite[3.3.6]{zbMATH04073740}.  
\item \label{klphprethtrherthtr2} For $a$ in $A$ and $r $ in $(0,\infty)$  the subset $\{\rho\in \Prim(A)|\rho(a)\ge r\}$
is (quasi)compact  \cite[3.3.7]{zbMATH04073740}.   \end{enumerate}}\end{fact}

\begin{lem}[{\cite[Lem.2.25]{zbMATH05657129}}] \label{ojhpertherthrtegtrg0}
For every sober topological space,
the data of an $X$-$C^{*}$-algebra structure on $A$ is equivalent to the data  of a map  $\tilde \psi:\Prim(A)\to X$ of topological spaces.
\end{lem}
\begin{proof}

By definition, an $X$-$C^{*}$-algebra is a frame morphism $A:\Open(X)\to I(A)$.
 We have $I(A)\cong \Open(\Prim(A))$; hence,
this map is the same as a  frame morphism
$\Open(X)\to \Open(\Prim(A))$.
Since $X$ is sober, the data of such a  map is the same as a continuous map $\tilde \psi:\Prim(A)\to X$ of topological spaces.
\end{proof}
\begin{lem} [{\cite[Sec. 2.1]{zbMATH05657129}}]\label{ojhpertherthrtegtrg1}
If $X$ is locally compact Hausdorff, then the  data of  a continuous map $\tilde \psi:\Prim(A)\to X$ of topological spaces
  is equivalent the data of  an essential homomorphism
$C_{0}(X)\to Z(M(A))$.  
\end{lem}
\begin{proof}
 The map $\tilde \psi $ gives map $\psi:C_{0}(X)\to C_{b}(\Prim(A))\cong Z(M(A))$. In order to see that 
it is essential,  let $b$ be in $Z(M(A))$ and assume that $\psi(f)b=0$ for all  $f$ in $C_{0}(X)$.
Let $\rho:A\to B(H)$ be in $\Prim(A)$ and  consider $x:=\tilde \psi(b)$ in $X$.
 We then have $\hat \rho(\psi(f)b)=f(x)\hat \rho(b)$.
Since $X$ is locally compact Hausdorff,  we can find $f$ in $C_{0}(X)$ with $f(x)\not=0$. 
We can therefore conclude that  $\hat \rho(b)=0$ for all $\rho$ in $\Prim(A)$. This implies $b=0$ by the injectivity part of the Dauns-Hofmann theorem  \cref{gjroepgfwergwerferfw}.\ref{klphprethtrherthtr}.

Concersely,  assume that $\psi:C_{0}(X)\to   Z(M(A))$ is an essential homomorphism.
Then we get a continuous map $\tilde \psi:\Prim(A)\to X$. Its  Gelfand dual   is the composition  $$
\tilde \psi(\rho):C_{0}(X)\stackrel{\psi}{\to} Z(M(A))\stackrel{\hat \rho}{\to}
\rho(A)'\cong \C\ .$$  This homomorphism is indeed non-zero since $\psi$ is essential.  \end{proof}
The following  lemma is essentially explained in \cite{zbMATH05657129}.
\begin{lem} \label{ojhpertherthrtegtrg2}
If $X$ is a locally compact Hausdorff,
 then essential homomorphisms $\psi:C_{0}(X)\to Z(M(A))$ are in bijection with  surjective  $A$-bilinear homomorphisms
$m:C_{0}(X)\otimes A\to A$. \end{lem} 
\begin{proof}Given $\psi$ we get a homomorphism
$C_{0}(X)\otimes^{\alg} A\to A$ which induces $m$ by the universal property of the tensor product.
We first show that $m$ is essential. Let $a$ be in $A$ and assume that $m(f\otimes b)a=0$ for all $f$ in $ C_{0}(X)$ and $b$ in $A$.   Then    $b\psi(f)a=0$ for all $b$ in $A$ and hence $\psi(f)a=0$ for all $f$. This gives
$$\sup_{\rho\in \Prim(A)} \|f( \tilde \psi(\rho)) \rho(a)\|=0\ .$$
Since again for every $x$ in $X$ (so in particular for $x=\tilde \psi(\rho)$) there exists a function $f$ in $C_{0}(X)$ with $f(x)\not=0$ we conclude that 
$\sup_{\rho\in \Prim(A)} \| \rho(a)\|=0$ which implies $a=0$ by \cref{gjroepgfwergwerferfw}.\ref{klphprethtrherthtr1}.


 We now show that 
  $m:C_{0}(X)\otimes A\to A$ is surjective.
  It suffices to show that $m$ has a dense image. Let $a$ be in $A$ and fix $\epsilon$ in $(0,\infty)$.
Then $K:=\{\rho\in \Prim(A)\mid \|\rho(a)\|\ge \epsilon\}$ is compact by \cref{gjroepgfwergwerferfw}.\ref{klphprethtrherthtr2}.
Hence $\tilde \psi(K)$ is a compact subset of $X$. Using that $X$ is locally compact Hausdorff, we can 
 choose $f$ in $C_{0}(X)$ such that $\|f\|\le 1$ and $f_{|K}\equiv 1$.
 We then have
 $$\|a- m(f\otimes a)\|=\sup_{\rho\in \Prim(A)} \|\rho(a- \psi(f)a)\|
 =\sup_{\rho\in \Prim(A)} \|(1- f(\tilde  \psi(\rho)) \rho(a) \|\le \epsilon \ .$$

Vice versa, a surjective $A$-bilinear homomorphism $m:C_{0}(X)\otimes A\to A$ induces a homomorphism
$\psi:C_{0}(X)\to Z(M(A))$ by $\psi(f) a:=m(f\otimes a)$. The central multiplier property follows from $A$-bilinearity of $m$ by 
$\psi(f)(ab)=m(f\otimes ab)=m(f\otimes a)b=\psi(f)(a)b$ and  $\psi(f)(ab)=m(f\otimes ab)=am(f\otimes b)=a\psi(f)(b)$.
 We must check that $\psi$ is essential.
Let $z$ be in $Z(M(A))$ and assume that $\psi(f) z=0$ for all $f$ in $C_{0}(X)$. We must show that $z=0$. The assumption implies that 
$\psi(f)zab=\psi(f)azb= m(f\otimes a)zb=0$ for all $f$ in $C_{0}(X)$ and $a,b$ in $A$.
Since $m$ is surjective, we conclude that $zb=0$ for all $b$ in $A$, and thus $z=0$.
\end{proof}
%
%
%
%
%
%
%
%
%
%
%
%
%
%
%
%
%
%
%
%
%
%
%
%
%
%
%
%
%
%
%
%
%
%
The combination of \cref{ojhpertherthrtegtrg0}, \cref{ojhpertherthrtegtrg1}, and \cref{ojhpertherthrtegtrg2}
proves \cref{kophertgetrgtreeg}.
\end{proof}

\begin{rem}
In view of \cref{ojhpertherthrtegtrg2} and \cref{kophertgetrgtreeg},
our definition of $X\nCalg$ for a locally compact Hausdorff space $X$ coincides with \cite[Def. 2.1]{zbMATH05657129}.
The combination of \cref{ojhpertherthrtegtrg1} and \cref{ojhpertherthrtegtrg0} furthermore shows that,
in this case, our definition of $X\nCalg$ coincides with the standard definition of a $C_{0}(X)$-algebra as an algebra $A$ equipped 
with an essential homomorphism  $C_{0}(X)\to Z(M(A))$.
By \cref{ojhpertherthrtegtrg0}, our definition of $X\nCalg$ for a sober topological space $X$ coincides with
 \cite[Def. 2.3]{zbMATH05657129}. \hB
\end{rem}

\begin{rem}\label{herhttrgrtgertge}
In view of \cref{kophertgetrgtreeg}
we consider objects of  $X\nCalg$ equivalently  as pairs  $(A,m)$
of a $C^{*}$-algebra and an $A$-bilinear surjective homomorphism $
m:C_{0}(X)\otimes A\to A$. In this picture, the value of the functor $A:\Open(X)\to I(A)$ 
is given by \begin{equation}\label{wefqfoijoqwef}A(U)=\im(m:C_{0}(U)\otimes A\to A)\ .\end{equation}
 A morphism $f:(A,m)\to (A',m')$ is then  a homomorphism $f:A\to A'$ 
in $\nCalg$ such that the following diagram commutes: $$\xymatrix{C_{0}(X)\otimes A\ar[r]^{\id\otimes f}\ar[d]^{m} & C_{0}(X)\otimes A'\ar[d]^{m'} \\ \ar[r]^{f} A&A' }\ .$$ 
  A continuous map
$f:X\to Y$ between locally compact Hausdorff spaces induces the functor 
$$f_{!}:X\nCalg\to Y\nCalg\ ,  \quad (A,m)\mapsto (A,f_{!}m)\ ,$$
where $$f_{!}m:C_{0}(Y)\otimes A\to C_{b}(X)\otimes A \stackrel{\hat m}{\to} A\ ,$$
with 
$ \hat m$ being the natural extension of $m$.

 We have an inclusion
\begin{equation}\label{gerfwrgegrg}\Hom_{X\nCalg}((A,m),(A',m'))\subseteq \Hom_{\nCalg}(A,A')\ .
\end{equation} 
   \hB
\end{rem}

  \begin{rem}
  If $(A,m)$ is in $X\nCalg$, then the $A$-bilinear homomorphism
  $m:C_{0}(X)\otimes A\to A$ satisfies
  \begin{equation}\label{ijiofjwoerfwerf}
m(f\otimes m(g\otimes a))=m(fg\otimes a)\ .
\end{equation}
 To see this, we     use an approximate  unit $u$ of $A$ and 
\begin{eqnarray*}
m(f\otimes m(g\otimes a))&\sim& m(f\otimes u m(g\otimes a))\\&=&
m(f\otimes u)m(g\otimes a)\\
&=&m(fg\otimes ua)\\&\sim& m(fg\otimes a)\ .
\end{eqnarray*}
\hB
  \end{rem}

In the following proposition, we describe filtered colimits, kernels, and quotients in $ X\nCalg$ 
in terms of the pairs $(A,m)$.
For the the second and third statements we consider the squares 
 \begin{equation}\label{243rt243r342t2t4}\xymatrix{(A,m)\ar[r]\ar[d] &(B,n) \ar[d] \\ 0\ar[r] &(C,r) } 
\end{equation} and 
\begin{equation}\label{gewrfewfewferfreferf}\xymatrix{A(X)\ar[r]\ar[d] &B(X) \ar[d] \\ 0\ar[r] &C(X) } \ .\end{equation}
\begin{prop}\label{kophrtgretgggrethrthe} \mbox{}
\begin{enumerate} \item \label{jiggtre}
A filtered colimit  $(A,m)$ of a system $(A_{i},m_{i})_{i\in I}$ in $X\nCalg$  is given by   $A:=\colim_{i\in I}A_{i}(X)$ and the  canonically induced $A$-bilinear multiplication map $m:C_{0}(X)\otimes A\to A$. 
\item \label{jiggtre2} A square  \eqref{243rt243r342t2t4} in $X\nCalg$ is cartesian  
 if and only if the square \eqref{gewrfewfewferfreferf}
 is cartesian  in $\nCalg$.
  \item  \label{jiggtre3} A square  \eqref{243rt243r342t2t4} in $X\nCalg$ is cocartesian  
 if and only if the square \eqref{gewrfewfewferfreferf}
 is cocartesian  in $\nCalg$.
\end{enumerate}
\end{prop}
\begin{proof}
 Using that the maximal tensor product in $\nCalg$  preserves filtered colimits, we define $m$ by
$$C_{0}(X)\otimes A\cong C_{0}(X)\otimes \colim_{i\in I}A_{i}\cong \colim_{i\in I}  C_{0}(X)\otimes  A_{i} \stackrel{(m_{i})_{i}}{\to } \colim_{i\in I}   A_{i} \cong A\ .$$
A filtered colimit of surjective maps is again surjective. This shows that
$(A,m)$ is in $X\nCalg$. We now show that it represents the colimit.
Let $(B,n)$ be in $X\nCalg$.
We have $$\Hom_{\nCalg}(A,B)\cong\lim_{i\in I} \Hom_{\nCalg}(A_{i},B)\ .$$
Using \eqref{gerfwrgegrg}, this bijection  induces an injective map
$$\Hom_{X\nCalg}((A,m),(B,n))\to \lim_{i\in I} \Hom_{X\nCalg}((A_{i},m_{i}),(B,n))\ .$$
It remains to show that it is surjective.
Let $(f_{i})_{i\in I}$ be a compatible system in the limit $$ \lim_{i\in I} \Hom_{X\nCalg}((A_{i},m_{i}),(B,n))$$ and $f$ be its preimage in $\Hom_{\nCalg}(A,B)$.  We must show that $f$ actually belongs to the subset
$\Hom_{X\nCalg}((A,m),(B,n))$.
  For $i$ in $ I$ we consider the diagram
$$\xymatrix{\ar@/^-2cm/[ddd]A_{i}\ar[r]^{f_{i}}&\ar@{=}@/^2cm/[ddd]B\\ \ar[u]C_{0}(X)\otimes A_{i}\ar[r]^{\id\otimes f_{i}}\ar[d]&\ar[u]C_{0}(X)\otimes B \ar[d] \\ C_{0}(X)\otimes A\ar@{..>}[r]^{\id\otimes f}\ar[d]^{m} & C_{0}(X)\otimes B\ar[d]^{n} \\ \ar[r]^{f} A&B}\ .$$ 
We must show that the lower square commutes.
The outer square commutes by the definition of $f$.
By assumption, the upper square commutes for all $i$.
Hence, the composition of the lower two squares commutes for all $i$. 
  It follows from the universal property of  $C_{0}(X)\otimes A$ being the colimit of the system $(C_{0}(X)\otimes A_{i})_{i}$
  that the lower square commutes.
This finishes the proof of Assertion \ref{jiggtre}.
%
%
 
 We now show Assertions  \ref{jiggtre2} and   \ref{jiggtre3}.
  We already know   that   fibre and cofibre sequences in   $X\nCalg$
    are calculated   pointwise.      
    If the square \eqref{243rt243r342t2t4} is cartesian or cocartesian, then evaluating at $X$ we see that the square \eqref{gewrfewfewferfreferf} is cartesian or cocartesian, respectively.
    
For the converse   assume that \eqref{gewrfewfewferfreferf} is cartesian. 
We must show that $A(U)=\ker(B(U)\to C(U))$ for all $U$ in $\Open(X)$.
 We  have $\ker(B(U)\to C(U))=A\cap B(U)$. 
It remains to show that $A(U)=  A\cap B(U) $.
It is clear that $A(U)\subseteq  A\cap B(U)$.
 For surjectivity,
  consider $b$ in $A\cap B(U)$.  Then we can approximate $b$ by elements $n(f\otimes b)$ with $f$ in $C_{0}(U)$.
  We now note that $n(f\otimes b)\in A(U)$, so  also $b\in A(U)$.


We now assume that  \eqref{gewrfewfewferfreferf} is cocartesian.
 We must show that  for 
every $U$ in $\Open(X)$  the projection
$B(U)/A(U)\to C(A)$ is an isomorphism.
  It suffices to show that $A(U)$ maps onto the kernel of the second map $B(U)\to C(U)$ and that this map has a dense range.
 For the latter, note that 
 the subspace of elements $r(g\otimes c)$ with $c$ in $C$  and $g$ in $C_{0}(U)$ is dense in $C(U)$. For such elements, 
we find $b$ in $B$ such that $b\mapsto c$. Then $n(g\otimes b)\in B(U)$ and $n(g\otimes b)\mapsto r(g\otimes c)$.
 
Finally, assume that $b$ is in $\ker(B(U)\to C(U)$. We can approximate
$b$ by elements of the form
$n(f\otimes b)$ with $f$ in $C_{0}(U)$. These elements   are in the image of $A(U)\to B(U)$. Hence,
the image of $A(U)\to B(U)$ is the kernel of $B(U)\to C(U)$.
\end{proof}

\subsection{Presentability of  $X\nCalg$ for locally compact Hausdorff spaces}

For a general pointed frame $P$ we can define the $E$-theory context $P^{\rmc}\nCalg$ of continuous $P$-$C^{*}$-algebras,
but in contrast to  $P^{?}\nCalg$ for $?\in \{-,\uc\}$,  or also for $?=\reg$  (assuming that $P$ is stably locally compact), we have not yet shown  that the category $P^{\rmc}\nCalg$  is presentable. In this section, we show this presentability when $P=\Open(X)$ for a locally compact Hausdorff space.

%

 \begin{rem}
 If $X$ is a locally compact Hausdorff space, then $\Open(X)$ is a stably locally compact frame. For $U,V$ in $\Open(X)$ we write the way below relation as 
    $V\Subset  U$.  
  We have   $V\Subset  U $ 
    if and only  if there exists a function $\phi$ in $C_{c}(U)$ taking values in $[0,1]$ with 
  $\phi_{|V}\equiv 1$. We call $\phi$ a witness of the relation $V\Subset U$.
  
  \hB\end{rem}
%

The following notions will be used frequently below. Let $X$ be a topological space and $A$ be in $\Open(X)\nCalg$.

\begin{ddd}\label{hertrtgetrgertge} For a closed point $x$ in $X$,
we call $$A(x):=A(X)/A(X\setminus \{x\})$$ the fibre of $A$ at $x$ and
let $\ev_{x}:A\to A(x)$ denote the projection map.
\end{ddd}

\begin{theorem}\label{okprherhrtgertge} For a locally compact Hausdorff space $X$, 
we have an accessible  left  Bousfield localization
\begin{equation}\label{gioujgiotrgwergwerfwerf}c:\Open(X)^{\reg}\nCalg\leftrightarrows X\nCalg:\incl\end{equation}
and $X\nCalg$ is presentable.

\end{theorem}
\begin{proof}
We construct the adjunction in \eqref{gioujgiotrgwergwerfwerf} explicitly by describing the functor $c$ and the unit and counit. The functor $c$ is actually the restriction of a functor
\begin{equation}\label{gjiojowiejroijfwer}c:\Open(X)\nCalg\to X\nCalg
\end{equation}
to the subcategory $\Open(X)^{\reg}\nCalg$.
Let $A$ be in $\Open(X) \nCalg$. We use $A$ also to denote the $C^{*}$-algebra $A(X)$. 
Then we can form the  tensor product  $C_{0}(X)\otimes  A$.   We have a canonical isomorphism
$C_{0}(X)\otimes A\cong C_{0}(X,A)$. For every $x$ in $X$, we consider the evaluation
$$e_{x}: C_{0}(X,A)\stackrel{\phi\mapsto \phi(x)}{\to} A\stackrel{a\mapsto \ev_{x}(a)}{\to}  A( x)\ .$$ 
We let $I$ be the kernel of
$$(e_{x})_{x\in X}: C_{0}(X)\otimes A  \to \prod_{x\in X} A(x)\ .$$
 The elements of $I$ are thus functions $\phi$  in $C_{0}(X,A)$ with $\phi(x)\in A(X\setminus \{x\})$ for every $x$ in $X$.
We then define the underlying $C^{*}$-algebra of $c(A)$ by
$$c(A):=(C_{0}(X)\otimes A)/I\ .$$
We have a map of exact sequences
$$
\xymatrix{
0\ar[r]&C_{0}(X)\otimes I\ar[d]\ar[r]&C_{0}(X)\otimes C_{0}(X)\otimes A\ar[r]\ar[d]&C_{0}(X)\otimes c(A)\ar[r]\ar[d]^{m'}&0\\
0\ar[r]&I\ar[r]&C_{0}(X)\otimes A\ar[r]&c(A)\ar[r]&0}
$$
where the middle vertical map is induced by $C_{0}(X)\otimes C_{0}(X)\cong C_{0}(X\times X)\stackrel{\diag^{*}}{\to}
C_{0}(X)$. 
An element of $C_{0}(X)\otimes I$ can be considered as a function $\phi:X\times X\to A$ such that
$\phi(x,y)\in A(X\setminus \{y\})$ for every $x,y$ in $X$.
The middle vertical map sends this function to $x\mapsto \phi(x,x):X\to A$, which obviously belongs to $I$.
 Hence the  left vertical map is well-defined. 
We conclude the existence of the right vertical map, which is surjective
 since the middle map is surjective.  It is the $c(A)$-bilinear multiplication map
 witnessing  $(c(A),m')$ as an object of $X\nCalg$.

 We now construct the counit 
 $$\eta:c\circ \incl\stackrel{\cong}{\to} \id:X\nCalg\to X\nCalg\ .$$
   
\begin{lem}\label{lpherhrtgert} If $A$ is in $X\nCalg$, then   \begin{equation}\label{gergfwerferferfw}
\ev:=(\ev_{x})_{x\in X}:A  \to \prod_{x\in X} A (x)
\end{equation}
 is injective. \end{lem}
 \begin{proof}
Let $\tilde \psi:\Prim(A)\to X$ be the   map from \cref{ojhpertherthrtegtrg0} corresponding to the structure of $A$ and $m:C_{0}(X)\otimes A\to A$ the structure homomorphism from  \cref{kophertgetrgtreeg}.
For every $\rho$ in $\Prim(A)$ we have $A(X\setminus \{\tilde \psi(\rho)\})\subseteq \ker(\rho)$.
 To see this, note that,   by \eqref{wefqfoijoqwef}, we must show that
$\rho(m(f\otimes a))=0$ for all $ a$ in $A(X)$ and $f$ in $C_{0}(X\setminus \{\tilde \psi(\rho)\})$.
We have  $ \rho(m(f\otimes a))=f(\tilde \psi(\rho))\rho(a)=0$ since $f(\tilde \psi(\rho))=0$.
Hence we get a homomorphism $A( \tilde \psi(\rho))\to A(X)/\ker(\rho)$.
If $(\ev_{x})_{x\in X}(a)=0$, then also $0=([a]_{A(X)/\ker(\rho)})_{\rho\in \Prim(A(X))}$.
This implies, by \cref{gjroepgfwergwerferfw}.\ref{klphprethtrherthtr1}, that $a=0$.
  \end{proof}
    Assume   that  $A$ is in $X\nCalg$.   By \cref{kophertgetrgtreeg}
  we have the structure  morphism
$$m:C_{0}(X)\otimes A\to A\ .$$  

  \begin{lem} \label{herthertgetrgtr}
  The  kernel of $m$ is precisely the ideal $I$. \end{lem}
  \begin{proof}
 Using \cref{lpherhrtgert} we can  conclude  that
   the kernel of $m$ is precisely the kernel of
 $$e:C_{0}(X)\otimes A\stackrel{m}{\to} A\stackrel{(\ev_{x})_{x\in X}}{\to}  \prod_{x\in X} A (x)\ .$$
 We now show that the following   square
 $$\xymatrix{C_{0}(X)\otimes A\ar[ddrr]^{e}\ar[rr]^-{m}\ar[d]_{\cong }  &&A\ar[dd]^{(\ev_{x})_{x\in X}} \\ C_{0}(X,A)\ar[d]_{\phi\mapsto (\phi(x))_{x\in X}}&&\\ \prod_{x\in X} A \ar[rr]^{\prod_{x\in X}\ev_{x}}  &&  \prod_{x\in X} A (x)} $$ commutes.
It suffices to check  commutativity on elements of the form  $f\otimes a$ in $C_{0}(X)\otimes A$,  and this follows from \begin{equation}\label{gweropjiowergwe}\ev_{x}(m(f\otimes a))=f(x)\ev_{x}(a)=\ev_{x}(f(x)a)
\end{equation} in $A(x)$ for every $x$ in $X$.
 The ideal $I$ is by definition the kernel of the down-right composition, which coincides
 with the kernel of the right-down composition. This finishes the verification   that $I$ is the kernel of $m$.\end{proof}

By   \cref{herthertgetrgtr}  for $A$ in $X\nCalg$ we have a canonical isomorphism $$\eta_{A}:c\circ \incl(A) \stackrel{\cong}{\to} A\ , \qquad [f\otimes a]\mapsto m(f\otimes a)$$
of plain $C^{*}$-algebras.
We must check that this is a morphism in $X\nCalg$. To this end, we must check that   the square
$$\xymatrix{C_{0}(X)\otimes c\circ \incl(A)\ar[r]\ar[d]^{m'}&C_{0}(X)\otimes A\ar[d]^{m}\\ c\circ \incl(A)\ar[r]&A}$$ commutes, see \cref{herhttrgrtgertge}.
Let $f$ be in $C_{0}(X)$ and $[g\otimes a]$ be in $c\circ \incl(A)$. Then the square gives
$$\xymatrix{f\otimes [g\otimes a]\ar@{|->}[r]\ar@{|->}[d]&f\otimes m(g\otimes a)\ar@{|->}[d]\\ [fg\otimes a]\ar@{|->}[r]&m(fg\otimes a)=m(f\otimes m(g\otimes a))}$$
The equality at the lower corner is \eqref{ijiofjwoerfwerf}.  

We now define the unit 
$$\epsilon:\id\to \incl\circ c$$   as follows.
Let $A$ be in $\Open(X)^{\reg}\nCalg$. Then the  component at $A$ of the unit  of the adjunction   \eqref{gwerfwerfwerfwerferf2}  is an isomorphism
$$A \cong R( A ) \ ,$$ where we consider $R$ as defined in $\Open(X)\nCalg$ and omit to write inclusions explicitly.
Therefore for every
 $U$  in $\Open(X)$ we have 
 $$ \overline{\bigcup_{V\Subset  U}A(V) }=A(U)  \ .$$
 For 
$a$  in $A(U) $
we must define $\epsilon_{A}(a)$ in $c(A)(U)$.
Assume first that   $a\in \bigcup_{V\Subset  U}A(V)$. 
Then $a\in A(V)$ for some 
  $V\Subset  U$  witnessed by $\phi$.  
 In this case, we define  $\eta_{A}(a):=[\phi\otimes a]$ in $c(A)$.
  This is independent of the choice of $\phi$ and multiplicative. To see independence,
let $\phi'$ be a second choice. Then  for every $x$ in $X$, we have
$\ev_{x}((\phi-\phi')  a)\in A(X\setminus \{x\})$. Indeed, 
 $ \ev_{x}((\phi-\phi')  a)=0$ for $x$ in $V$ and $\ev_{x}((\phi-\phi')  a)\in A(V)$ for $x\not\in V$ and hence $V\subseteq X\setminus \{x\}$.
   In a similar manner, we show that this describes a $*$-algebra homomorphism.    We then extend this map to all of  $A(U)$ by continuity.

We next check the triangle identities.
We argue that for $A$ in $X\nCalg$ the composition
  $$c(A)\stackrel{c(\epsilon_{A})}{\to}  c(\incl(c(A)))\stackrel{\eta_{c(A)}}{\to} c(A)$$   
   is the identity.
     Let $[\psi\otimes a]$   be in $c(A)(U)$ with 
  $a$ in $A(V)$ for     
    $V\Subset   U$ witnessed by $\phi$, and $\psi\in C_{0}(U)$.
    The elements of this form generate $c(A)(U)$.
 The composition sends this to $[\psi \phi\otimes a]$, but
   $[\psi \phi\otimes a]=[\psi\otimes a]$. 
      
    Furthermore, for $B$ in $\Open(X)^{\reg}\nCalg$ we show that the composition
    $$\incl(B)\stackrel{\epsilon_{\incl(B)}}{\to} \incl(c(\incl(B))) \stackrel{\incl(\eta_{B})}{\to} \incl(B)$$
    is the identity.  Note that $\incl(B)(U)$ is generated by 
     $b$  in $B(V)$ for
    $V\Subset  U$ witnessed by $\phi$. The composition sends
    $b$ to $m(\phi\otimes b)$. We have
    $\ev_{x}(m(\phi\otimes b))=\phi(x) \ev_{x}(b)=\ev_{x}(b)$ since $\phi(x)=1$ for $x$ in $V$ and
    $\ev_{x}(b)=0$ for $x\not\in V$.
    This finishes the verification that \eqref{gioujgiotrgwergwerfwerf} is a left Bousfield localization. 
    
    We already know from a combination of \cref{kohperthertgertgetrg1} and \cref{kohperthertgertgetrg}  that $\incl$ preserves filtered colimits. Since $X\nCalg$ is an accessible left Bousfield localization
    of   $\Open(X)^{\reg}\nCalg$, which is presentable by \cref{jioretherthertgetrgertg}, we conclude that $X\nCalg$ is also presentable.
     \end{proof}

For later  use, we record the following corollary of the proof of \cref{okprherhrtgertge}.
 
\begin{kor} \label{lkopherthertgetg}If $X$ is a locally compact Hausdorff space and  $(A,m)$ is in $X\nCalg$, then
we have an exact sequence
$$0\to I\to C_{0}(X)\otimes A\stackrel{m}{\to} A\to 0$$
where the ideal $I$ consists   precisely of the functions $\phi$ in $C_{0}(X,A)\cong C_{0}(X)\otimes A$ with
$\phi(x)\in A(X\setminus\{x\})$ for all $x$ in $X$.
\end{kor}

\subsection{Fell bundles}\label{jiojgowerwregerfwrf}

Recall that for a locally compact Hausdorff space, we can equivalently describe objects of $X\nCalg$ as
pairs $(A,A(-))$ of a $C^{*}
$-algebra  equippped with a frame morphism $\Open(X)\to I(A)$, or as pairs $(A,m)$ consisting of a $C^{*}$-algebra $A$ equipped with a   surjective homomorphism $C_{0}(X)\otimes A\to A$. In this section,  we provide a third alternative picture of $X\nCalg$ in terms of upper semicontinuous $C^{*}$-bundles or Fell bundles (see \cref{hetrertgrtgrtegetg} below). This equivalence  is
well-known in the $C^{*}$-literature and  is essentially due to \cite{Nilsen_1996}. 
As the notation and some of the details are relevant for later constructions, we  
 reproduce the arguments for completeness. 
We  apply the Fell-bundle picture in order to describe the contravariant functorality
$$\LCH^{\op}\to \Cat\ , \quad  X\mapsto X\nCalg\ ,  \quad f\mapsto f^{*}\ .$$
 Let $X$ be a locally compact Hausdorff space, consider $A$ in $\Open(X)\nCalg$, and recall \cref{hertrtgetrgertge}.  
The $C^{*}$-algebra $C_{b}(X)$ acts by multipliers on the $C^{*}$-algebra $\prod_{x\in X} A(x)$.
 We have a canonical map $$\ev:A\to \prod_{x\in X} A(x)\ , \quad a\mapsto \ev(a):=(\ev_{x}(a))_{x\in X}\ .$$
The elements in the product  are called sections, and for 
 $a$ in $A$, we call $\ev(a)$ the constant section  associated to $a$. 


\begin{lem}\label{kopherhrtgrgrtegertgrgrtgerge}
Assume that  $A$ belongs to $\Open(X)^{\reg}\nCalg$. Then the following assertions are equivalent:
\begin{enumerate}
\item \label{wzerguwerg89989893} We have $A\in X\nCalg$.
\item \label{wzerguwerg899898931}  The map 
  $\ev:A\to \prod_{x\in X} A(x)$ is injective, its image   is preserved by
multiplication with elements of $C_{0}(X)$, and for every $U$ in $\Open(X)$ the  ideal
$\ev(A(U))$ is precisely the ideal generated by products with $C_{0}(U)$.
\end{enumerate} \end{lem}
\begin{proof} If $A$ is in $X\nCalg$  corresponding to a pair $(A,m)$ as in \cref{kophertgetrgtreeg}, then 
$\ev$ is injective by \cref{lpherhrtgert}. 
The algebra $A$ is identified by $\ev$ with   the algebra  of  constant sections.
In view of $\ev_{x}(m(f\otimes a))=f(x)\ev_{x}(a)$ (see \eqref{gweropjiowergwe})
and the surjectivity of $m$
 the image of $\ev$ is also the algebra generated by sections of the form $f \ev(a)$ for $f$ in $C_{0}(X)$ and $a$ in $A$. In particular, it is preserved by  the multiplication by $C_{0}(X)$.
Finally, for every $U$ in $\Open(X)$, the ideal  $\ev(A(U))$ is precisely the ideal generated by products with $C_{0}(U)$.
Hence Assertion  \ref{wzerguwerg89989893} implies Assertion  \ref{wzerguwerg899898931}.

Conversely,  assume Assertion  \ref{wzerguwerg899898931}. The sub-$C^{*}$-algebra  of $ \prod_{x\in X}A(x)$
generated by   products of constant sections  with elements of  $C_{0}(X)$
is precisely $c(A)$ considered in \cref{okprherhrtgertge}. Let $m:C_{0}(X)\otimes c(A)\to c(A)$ denote the obvious multiplication map. Then $(c(A),m)$ belongs to 
 $X\nCalg$.

 By assumption, the map $\ev$ induces an injective map $A\to c(A)$ of $C^{*}$-algebras.  But this map is also surjective by the definition of $c(A)$. 
 The conditions in \ref{wzerguwerg899898931} imply that $A\to c(A)$ is actually an isomorphism in $\Open(X)\nCalg$, and hence in $X\nCalg$.
%
%
%
%
%
%
%
%
%
%
%
%
%
%
 Hence Assertion  \ref{wzerguwerg899898931} implies Assertion \ref{wzerguwerg89989893}.
\end{proof}

%
%
%
%
\begin{constr} \label{okhptrhertgergergretg}{\em 
Assume   that $(A_{x})_{x\in X}$ is a family of $C^{*}$-algebras indexed by the points of a locally compact Hausdorff space $X$. We consider a subalgebra $A$ of $\prod_{x\in X}A_{x}$. We assume that $A$ is preserved by the multiplication with elements of $C_{0}(X)$ and   that the  $A$-bilinear multiplication map $m:C_{0}(X)\otimes A\to A$ is surjective.
For $U$ in $\Open(X)$, we define $A(U)$ as the subalgebra generated by products of $A$ with elements of $C_{0}(U)$. This is a closed $*$-ideal.  We let $\pi_{x}:A\to A_{x}$ denote the restriction of the projection to the factor with index $x$.}\hB\end{constr}
 \begin{lem}\label{hlkperthrtgertgertg}\mbox{} In the situation of \cref{okhptrhertgergergretg}, we have the following assertions.
 \begin{enumerate} \item  \label{gojpwegerfwerfwerfw} $A$ with the structure $\Open(X)\ni U\mapsto A(U)\in I(A)$ belongs to $X\nCalg$. \item \label{gojpwegerfwerfwerfw1} For every $x$ in $X$ and $\phi$ in $C_{0}(X)\otimes A\cong C_{0}(X,A)$ we have
 $$\pi_{x}(m(\phi))=\pi_{x}(\phi(x))\ .$$  
 \item  \label{gojpwegerfwerfwerfw2}  For every $a$ in $A$, the function $X\ni x\mapsto  \|\pi_{x}(a)\|_{A(x)}$ vanishes at $\infty$.
\end{enumerate}
\end{lem}
\begin{proof}
The Assertion  \ref{gojpwegerfwerfwerfw} is 
precisely the description of the functor $\Open(X)\to I(A)$ obtained from the pair $(A,m)$ via the equivalence 
asserted in 
 \cref{kophertgetrgtreeg}.

  It suffices to check the identity claimed in  Assertion  \ref{gojpwegerfwerfwerfw1} on elements of the form $\kappa\otimes a$ for  $\kappa$   in $C_{0}(X)$ and $a$   in $A$. By the definition of the multiplication map, we have $$\pi_{x}(m(\kappa\otimes a))=\kappa(x)\pi_{x}(a)\ .$$ On the other hand, $$\pi_{x}((\kappa\otimes a)(x))=\pi_{x}(\kappa(x) a)=\kappa(x) \pi_{x}(a)\ .$$
  
 We now show  Assertion \ref{gojpwegerfwerfwerfw2}.
  The function $x\mapsto  \|\pi_{x}(a)\|_{A_{x}}$ is  bounded by $\|a\|_{A}$.
We can approximate $a$ by elements of the form $m(f\otimes a)$ for $f\in C_{0}(X)$. Then by Assertion \ref{gojpwegerfwerfwerfw}
$$x\mapsto  \| \pi_{x}(m(f\otimes a))\|_{A_{x}}= |f(x)|\|
 \pi_{x}(a)\|_{A_{x}}$$ vanishes at $\infty$. It follows that the original function also vanishes at $\infty$.
 \end{proof}

The algebra $A$ described in  \cref{okhptrhertgergergretg} is an object of $X\nCalg$, but in this generality
the algebras $A_{x}$ are not necessarily the fibres defined in \cref{hertrtgetrgertge}.
The following result  provides a sufficient criterion for this property and clarifies the role of the condition that the functions 
 $x\mapsto \|\pi_{x}(a)\|_{A_{x}}$ are upper semi-continuous for all $a$ in $A$.

 \begin{prop}\label{jgwrieogjwefwefwrf}In the situation of \cref{okhptrhertgergergretg} the following assertions are equivalent:
\begin{enumerate}
 \item \label{ohkpetrhetrhe1}For every $x$ in $X$, the map $\pi_{x}$ induces an isomorphism $A(x)\stackrel{\cong}{\to} A_{x}$.
 \item\label{ohkpetrhetrhe2} For every $a$ in $A$, the function $x\mapsto \|\pi_{x}(a)\|_{A_{x}}$ is upper semi-continuous and $A\to A_{x}$ is surjective for all $x\in X$.
\end{enumerate}

\end{prop}
\begin{proof}
By the definition of $A(x)$, we have an exact sequence
$$0\to A(X\setminus \{x\})\to A \stackrel{\ev_{x}}{\to} A(x)\to  0\ .$$ Furthermore, $A(X\setminus \{x\})\subseteq \ker(\pi_{x})$ since it is generated by elements $m(\kappa\otimes a)$ with $\kappa(x)=0$, and using \cref{hlkperthrtgertgertg}.\ref{gojpwegerfwerfwerfw1}, we have  $\pi_{x}(m(\kappa\otimes a))=\kappa(x) \pi_{x}(a)=0$.
Therefore, the map $A(x)\to A_{x}$ induced by $\pi_{x}$  is well-defined.

We show that Assertion  \ref{ohkpetrhetrhe1} implies Assertion  \ref{ohkpetrhetrhe2}.
We consider $a$  in $A$ and must show that  $x\mapsto \|\pi_{x}(a)\|_{A_{x}}$ is upper semi-continuous.
Since the multiplication map $m:C_{0}(X)\otimes A\to A$ is surjective there exists $\phi$ in $C_{0}(X,A)$ such that $m(\phi)=a$.  By \cref{hlkperthrtgertgertg}.\ref{gojpwegerfwerfwerfw1} 
   for every $x$ in $X$,
 we have $$\pi_{x}(a)=\pi_{x}(m(\phi)) =\pi_{x}(\phi(x)) \ .$$ We now use that $A_{x}\cong A(x)\cong A/A(X\setminus \{x\})$.
  Given $\epsilon$ in $(0,\infty)$ and $x_{0}$ in $X$, we  find   $b$ in $A(X\setminus \{x_{0}\})$  such that $$ \|\phi(x_{0})+b\|_{A}  \le  \|\pi_{x_{0}}(a)\|_{A_{x_{0}}}+\epsilon/3    \ .$$  
Since $A(X\setminus \{x_{0}\})$ is the image of the multiplication map $C_{0}(X\setminus \{x_{0}\})\otimes A\to A$,
 we can find $\beta$ in $C_{0}(X\setminus \{x_{0}\})\otimes A$ such that
 $m(\beta)=b$. We can  find an open  neighborhood $V$ of $x_{0}$ and $\kappa $ in $ C_{0}(X\setminus \bar V  )\otimes A$ such that $\|\beta- \kappa \beta\|\le \epsilon/3$. We  have $ \|b-m(\kappa \beta)\|\le \epsilon/3$ and $m(\kappa \beta)\in A(X\setminus \bar V)$. Now let $\chi$ be in $C_{0}(V)$ such that $\chi(x_{0})=1$. Then $\chi \otimes m(\kappa \beta)\in \ker(m)$ since $\pi_{x}(m(\chi\otimes \kappa\beta))=\chi(x)\pi_{x}(\kappa\beta)=0$ for all $x$ in $X$. Indeed, 
 if $\chi(x)\not=0$, then $x\in V$ and $\pi_{x}(\kappa\beta)=0$ since $\kappa\beta\in A(X\setminus \bar V)\subseteq A(X\setminus\{x\})$.
 
 We can  replace $\phi$ by 
$ \phi+\chi \otimes m(\kappa \beta)$. Since $m(\chi \otimes m(\kappa \beta))=0$ we 
  still have  $m(\phi)=a$, but since $ \chi(x_{0})  m(\kappa\beta)= m(\kappa \beta)$ we now have 
$$\|\phi(x_{0})\|_{A}  \le \| \pi_{x_{0}}(a)\|_{A_{x_{0}}}+2\epsilon/3\ .$$

 By the continuity of $y\mapsto \|\phi(y)\|_{A}$  there exists a neighborhood $U$ of $x_{0}$ such that for all $x$ in $U$, we have
 $$ \| \pi_{x}(a)\|_{A_{x}}\le \|\phi(x)\|_{A }\le \|\phi(x_{0})\|_{A}+\epsilon/3\le   \|\pi_{x_{0}}(a)\|_{A_{x_{0}}}+\epsilon\ .$$
This finishes the verification that $x\mapsto \|\pi_{x}(a)\|_{A_{x}}$ is upper semi-continuous. 

We now show that Assertion  \ref{ohkpetrhetrhe2} implies Assertion   \ref{ohkpetrhetrhe1}.
We must show that $A(X\setminus \{x_{0}\}) =\ker(\pi_{x_{0}})$ for every $x_{0}$ in $X$.
As seen at the beginning of this proof, we have   $A(X\setminus \{x_{0}\})\subseteq \ker(\pi_{x_{0}})$. In order to show the converse
inclusion, consider $a$ in $\ker(\pi_{x_{0}})$. Then $\|\pi_{x_{0}}(a)\|_{A_{x_{0}}}=0$. Fix $\epsilon$ in $(0,\infty)$. Since
$x\mapsto \|\pi_{x}(a)\|_{A_{x}}$ is upper semi-continuous and vanishes at $\infty$ by \cref{hlkperthrtgertgertg}.\ref{gojpwegerfwerfwerfw2}, we can find $\kappa$ in $C_{0}(X\setminus \{x\})$ such that $$ \|m(\kappa\otimes a)-a\|_{A}=\sup_{x\in X}\| \kappa(x)\pi_{x}(a)- \pi_{x}(a) \|_{A_{x}}\le \epsilon\ .$$ But $m(\kappa\otimes a)\in A(X\setminus \{x\})$. Since $\epsilon$  can be chosen arbitrarily small and the ideal $A(X\setminus \{x\})$ is closed, we conclude that $a\in A(X\setminus \{x\})$.
\end{proof}

\begin{rem}
 In the situation of \cref{okhptrhertgergergretg},  assume that $A$ and $A'$ are subalgebras of 
 $\prod_{x\in X}A_{x}$ preserved by $C_{0}(X)$ with a surjective multiplication map.  
 
   \begin{lem}[{\cite[Prop. 3.1]{zbMATH01289894}}] If $A$ and $A'$  satisfy the equivalent conditions in \cref{jgwrieogjwefwefwrf}, and if
$A\subseteq A'$, then $A=A'$.
 \end{lem}
 \begin{proof}
 Let $a'$ be in $A'$. 
 We must show that $a'\in A$. 
We can approximate $a'$ by elements satisfying  $a'=\kappa a'$ for some $\kappa$ in $C_{c}(X)$. Then $\ev_{x}(a')=0$ for $x\not\in \supp(\kappa)$.
For every $x$ in $X$, we can find $a_{x}$ in $A$ such that $\ev_{x}(a_{x})=\ev_{x}(a')$. Fix $\epsilon$ in $(0,\infty)$.
Then there exists a neighborhood $V_{x}$ of $x$ such that
$\|\ev_{y}(a'-a)\|_{A_{y}}\le \epsilon$ for all $y$ in $V_{x}$.
We choose a partition of unity $(\chi_{x})_{x\in X}$ subordinated to the covering $(V_{x})_{x\in \sup(\kappa)}$.
We choose a finite family  $(x_{i})_{i\in I}$ of points in  $X$ such that
$\supp(\kappa)\subseteq \bigcup_{i\in I} V_{i}$.
We then consider the element
$$\sum_{i\in I} \chi_{x_{i}} a_{x_{i}}$$ in $A$.
We have
$$ \|\sum_{i\in I} \chi_{x_{i}} a_{x_{i}}-a'\| \le \sup_{y\in X}\|\ev_{y} (\sum_{i\in I} \chi_{x_{i}} (a_{x_{i}} -a' ) )\|_{A_{y}}
\le \sum_{i\in I} \chi_{x_{i}}(x)   \sup_{y\in V_{x_{i}}}   \|\ev_{y}(a_{x})-a')\|_{A_{y}}\le \epsilon\ .$$
Since $\epsilon$ is arbitrary, we conclude that $a'\in A$.
 \end{proof}
 \hB
\end{rem}

Combining the   results  above,  we arrive at our third  characterization of $X$-$C^{*}$-algebras in terms of Fell bundles (compare with \cite[Thm. 2.3]{Nilsen_1996}).
\begin{kor}\label{hetrertgrtgrtegetg}
For a locally compact Hausdorff space and a   family $(A_{x})_{x\in X}$ of $C^{*}$-algebras,
the following data are equivalent:
\begin{enumerate}
\item An object $A$ in $X\nCalg$ with $A_{x}\cong A(x)$ for all $x$ in $X$.
\item  \label{wrgwerwfwerf} A subalgebra
$A\subseteq \prod_{x\in X}A_{x}$ that is invariant under and generated by multiplication by  functions in $C_{0}(X)$
and such that the map $x\mapsto \|\pi_{x}(a)\|_{A_{x}}$ is upper semi-continuous for every $a$ in $A$ and $\pi_{x}:A\to A_{x}$ is surjective for every $x$ in $X$.
\end{enumerate}
\end{kor}
\begin{proof}
Given an object  $A$ in $X\nCalg$ with  $A_{x}\cong A(x)$ for all $x$ in $X$, we can view
$A$ as the subalgebra of constant sections of $ \prod_{x\in X}A_{x}$.

Conversely, let $A$ be in as in Assertion \ref{wrgwerwfwerf}. Then $A$ is in $X\nCalg$  and  $A_{x}\cong A(x)$ for all $x$ in $X$.
\end{proof}

\subsection{The $f^{*}$-functoriality and adjunctions}

We first describe the covariant  functoriality of $X\nCalg$ in the Fell bundle picture.  
Let $f:X\to Y$ be a continuous map between locally compact Hausdorff spaces, and $A$ be in $X\nCalg$.
\begin{lem}
The object $f_{!}A$ in $Y\nCalg$ can be identified with the subalgebra  $A$ of constant sections of 
$$\prod_{x\in X}A(x)\cong \prod_{y\in Y}\left( \prod_{x\in f^{-1}(\{y\})} A(x)\right)$$  
with the structure map $C_{0}(Y)\otimes f_{!}A\to A$ determined by this decomposition of the product.
\end{lem}
\begin{proof}
This follows from unfolding definitions.
\end{proof}

Note that here we use the homomorphism
$f^{*}:C_{0}(Y)\to C_{b}(X)$ and that $C_{b}(X)$ acts on $A$ by central multipliers.

\begin{kor}\label{kookphertgtrgertg}  The   fibre  $(f_{!}A)(y)$  is precisely the subalgebra of $ \prod_{x\in f^{-1}(\{y\})} A(x)$ generated by the constant sections
$f^{-1}(\{y\})\ni x\mapsto \ev_{x}(a)$ for $a$ in $A$.\end{kor}

Now  let $f:X\to Y$ be  a map, and $A$ be in $Y\nCalg$.
We use the Fell bundle picture from \cref{hetrertgrtgrtegetg} in order to construct
$f^{*}A$ such that $(f^{*}A)(x)=A(f(x))$ for all $x$ in $X$.

\begin{ddd}\label{kopghwthrwegregewrfref}
We define
 $f^{*}A$ as the subalgebra of 
$ \prod_{x\in X} A(f(x))$ generated by the products of  constant sections  with elements of  $C_{0}(X)$.
 \end{ddd}
 
 
This definition is justified  as follows. Recall  
 that a constant section  is a section of the form  $x\mapsto \ev_{f(x)}(a)$ for $a$ in $A$.
 By \cref{hlkperthrtgertgertg}.\ref{gojpwegerfwerfwerfw} we get the structure of $f^{*}A$ as an object of $X\nCalg$.
  Observe that $x\mapsto \|\ev_{f(x)}(a)\|_{A(f(x))}$ is upper semi-continuous.  Furthermore, 
 given $x$ and $b$ in $ A(f(x))$ there exists $a$ in $A$ such that $\ev_{f(x)}(a)=b$. 
 Hence $f^{*}A\to  A(f(x))$ is surjective. In view of \cref{jgwrieogjwefwefwrf} we get  \begin{equation}\label{fqweihiquwoefqwef}(f^{*}A)(x)\cong A(f(x))\ .
\end{equation}  
Thus, $f^{*}A$ has the desired fibres.

\begin{rem}
We can express  $f^{*}A$ in terms of the balanced tensor product 
$$f^{*}A\cong C_{0}(X)\otimes_{C_{0}(Y)}A\ ,$$ however, the above definition
is more amenable to verifying the functoriality,  providing the coherences  below, and describing the lax symmetric monoidal 
structure introduced in \cref{opkhetrhrtgergertg}.
\end{rem}


  Assume that $g:Y\to Z$ is a second map.
 \begin{lem}\label{kogpwererfwef}
 We have a canonical isomorphism $ \alpha_{g,f}: (g\circ f)^{*}\stackrel{\cong}{\to}f^{*}g^{*} $.
 \end{lem}
 \begin{proof} Let $A$ be in $Z\nCalg$. Then
   $g^{*}A$ is by definition the subalgebra of $\prod_{y\in Y} A(g(y))$ generated by products of constant sections with elements of $C_{0}(Y)$. Furthermore,
 $f^{*}g^{*}A$ is by definition the subalgebra of
 $\prod_{x\in X} (g^{*}A)(f(x))$ generated by 
  products of constant sections  associated to  elements of $g^{*}A$ with functions in $C_{0}(X)$.
For every $x$  in $X$ by  \eqref{fqweihiquwoefqwef}
we have the isomorphism
   \begin{equation}\label{gjeorgwegwerfwerfwre} (g^{*}A)(f(x))\cong A(g(f(x))\ .
\end{equation}
   Thus, we can identify
   $f^{*}g^{*}A$ with a subalgebra of
   $\prod_{x\in X} A(g(f(x))) $
   generated by the products of constant sections associated to elements of $A$
 with products of the form $ f^{*}h\cdot l$ for $l$ in $C_{0}(X)$ and $h$ in $C_{0}(Y)$.
 But this is the same as the subalgebra generated by the products of constant sections associated to elements of $A$ with elements of $C_{0}(X)$.
The latter is precisely the  definition of $(g\circ f)^{*}A$. 
Consequently, $\alpha_{g,f}$ is the identification of the corresponding subalgebras
of $\prod_{x\in X} (g^{*}A)(f(x))$ and $\prod_{x\in X}A(g(f(x)))$, where these products are identified by   the family of isomorphisms 
\eqref{gjeorgwegwerfwerfwre} for all $x$ in $X$.
   \end{proof}
   
   Note  that $\Cat$ is a $2$-category.  Since we implicitly always work with $\infty$-categories, a
   functor from a $1$-category to $\Cat$ is what  is  classically  called a pseudofunctor.

\begin{prop}\label{kogpwtegrwref}
We have constructed a functor
$$\LCH^{\op}\to \Cat\ , \quad  X\mapsto \nCalg\ , \quad f\mapsto f^{*}\ .$$
\end{prop}
\begin{proof} By \cref{kogpwererfwef}
for every composable pair $f,g$ of morphisms in $\LCH$  we have defined an isomorphism 
$\alpha_{g,f}:f^{*}\circ g^{*}\stackrel{\cong}{\to} (g\circ f)^{*}$.
These isomorphisms satisfy the associator relation expressed by the commutativity of the diagram
\begin{equation}\label{}\xymatrix{&\ar[dl]_{\alpha_{hg,f}}(hgf)^{*}\ar[dr]^{\alpha_{h,gf}}&\\ \ar[dr]^{f^{*}\alpha_{h,g}}f^{*}(hg)^{*}&&(gf)^{*}h^{*}\ar[dl]_{\alpha_{g,f}h^{*}}\\ &f^{*}g^{*}h^{*}&}\ ,
\end{equation}
which is checked by unfolding definitions.
We also have $\id^{*}=\id$ and the compatibility of unit relations holds.
\end{proof}

  The following result extends the adjunction in  \cref{okhpzrhtzjrtzj1} from closed embeddings to proper maps between locally compact Hausdorff spaces.
  \begin{lem}\label{okhprethetgertgtge}
  If $f:X\to Y$ is a proper map between locally compact Hausdorff spaces, then we have an adjunction
  \begin{equation}\label{gerfwerfrebgw}f^{*}:Y\nCalg\leftrightarrows X\nCalg:f_{!}\ .
\end{equation} 
If $f$ is injective, then $f^{*}f_{!}\cong  \id$.
  \end{lem}
\begin{proof}
We describe  the unit and counit and then check the triangle identities. 
For $A$ in $X\nCalg$ we consider the map
$$\prod_{x\in X}  \prod_{x'\in f^{-1}(\{x\})} A(x')\to \prod_{x\in X} A(x)$$
given by the evaluation at $x'=x$. It restricts to the counit
$f^{*}f_{!}A\to A$.
To this end, we observe that it preserves constant sections and is $C_{0}(X)$-linear.
If $f$ is injective, then it is an isomorphism.
 
For $B$ in $Y\nCalg$ we consider the map
$$ \ \prod_{y\in B} B(y) \to \prod_{y\in Y}   \prod_{x\in f^{-1}(\{y\})}B(y) $$
given by the obvious diagonal map. Since $f$ is proper, it sends $C_{0}(Y)$ to $C_{0}(X)$ via pull-back. Using this,  
we see that the map restricts to 
$$ B\to f_{!}f^{*}B\ .$$
The triangle identities are straightforward.  
\end{proof}

\begin{rem}
It is an interesting question whether, in the statement of \cref{okhprethetgertgtge}, one can
  show the existence of a left-adjoint of $f_{!}$ under weaker conditions on $X$.
But one  can consider other assumptions on $Y$ and $f$ (see \cref{rgrokpowergwergwerferwf}). \hB \end{rem} 
 
\begin{lem} \label{kopherhrtgertge}If $f:X\to Y$ is an open inclusion, then we have a right Bousfield localization
$$f_{!}:X\nCalg\leftrightarrows Y\nCalg:f^{*}\ .$$
\end{lem}
\begin{proof}
We know  by 
\cref{okhpzrhtzjrtzj} that $f_{!}$ has a right-adjoint.
The point of the lemma is that this right-adjoint is given by the functor from \cref{kogpwtegrwref}.
To this end we write down the unit and counit of the adjunction explicitly.
The unit $A\to f^{*}f_{!}A$ for $A$ in $X\nCalg$ is  
 the isomorphism obtained by restricting the isomorphism 
$$\prod_{x\in X} A(x) \to \prod_{x\in X} \prod_{y\in f^{-1}(\{f(x)\}}A(x)   $$
to $A$.  The counit $f_{!}f^{*}B\to B$  for $B$ in $Y\nCalg$ the  inclusion of a subproduct obtained by restricting
$$  \prod_{y\in Y}  \prod_{x\in f^{-1}(\{y\})} B(y) \to  \prod_{y\in Y} B(y)$$
to $B$. Note   that $f^{-1}(\{y\})$ is either empty or consists of precisely one point.
The triangle identities are straightforward to check.
 \end{proof}
%
%

  
For the proof of \cref{okpherthrtgetgetg} we must verify various base-change properties relating the contravariant and covariant functorialities. 
  We  consider a cartesian diagram
\begin{equation}\label{gerwljewrg}\xymatrix{X\ar[r]^{j}\ar[d]^{g} &Z \ar[d]^{f} \\Y \ar[r]^{i} &W } 
\end{equation} in $\LCH$.

\begin{lem}\label{jptzlkphtzhrzhz}
 We have a canonical isomorphism 
\begin{equation}\label{fiqjiwoefqwefqewfqfw}   j_{!}g^{*}\stackrel{\cong}{\to} f^{*}i_{!}:  Y\nCalg \to Z\nCalg\ .
\end{equation}
\end{lem}
\begin{proof} Let $A$ be in $Y\nCalg$.
We consider the map \begin{eqnarray*}
\prod_{z\in Z} \prod_{x\in j^{-1}(\{z\})} A(g(x))&\to   &     \prod_{z\in Z}  \prod_{x\in j^{-1}(\{z\})} \prod_{y\in i^{-1}(\{i(g(x))\}} A(g(x))\\&\cong &
    \prod_{z\in Z}  \prod_{x\in j^{-1}(\{z\})} \prod_{y\in i^{-1}(\{f(z)\}} A(y)
\\&\to& 
\prod_{z\in Z}   \prod_{y\in i^{-1}(\{f(z)\})} A(y)
\end{eqnarray*}
The first map is the obvious diagonal.
The component of the second map at $z$  is the restriction along $y\mapsto (x_{y},y)$, where
$y$ is in $i^{-1}(\{f(z)\})$ and $x_{y}$ is the unique point in $X$ with $g(x_{y})=y$ and $j(x_{y})=z$. Here we use that the square \eqref{gerwljewrg} is cartesian.
For every $z$ the map $y\mapsto x_{y}$ induces an identification of the index sets
of the inner products in the domain and target. Under this identification,
the whole composition becomes the identity.
 One now checks that
the composition restricts to an isomorphism
in $Z\nCalg$.
To this end we observe that it preserves constant sections and is compatible with the multiplication
by $C_{0}(Z)$.
\end{proof}
\begin{rem}
The isomorphism found in \cref{jptzlkphtzhrzhz} is not yet determined by
adjunction data. But the point of the following
\cref{oigjworiegrefwerfewfref} and \cref{joihetrgertgertgerg} is that this isomorphism is the Beck-Chevalley morphism in the cases where the adjunction data exists. \hB
\end{rem}

 \begin{lem}\label{oigjworiegrefwerfewfref} Consider a cartesian square \eqref{gerwljewrg} and assume that  $i$ and hence $j$ are open embeddings.
Then the square 
$$\xymatrix{W\nCalg\ar[r]^{i^{*}}\ar[d]^{f^{*}} & Y\nCalg\ar[d]^{g^{*}} \\ Z\nCalg\ar[r]^{j^{*}} &X\nCalg } $$
is horizontally left-adjoinable. 
\end{lem}
\begin{proof}
One checks,  using the proof of \cref{kopherhrtgertge}, that the isomorphism  in \eqref{fiqjiwoefqwefqewfqfw} is precisely 
the composition of left-adjoints 
$$j_{!}g^{*}\to j_{!}g^{*}i^{*}i_{!}\cong j_{!}j^{*}f^{*}i_{!}\to f^{*}i_{!} $$
induced by the respective unit and counit. 
%
%
\end{proof}


  We consider a   cartesian diagram \begin{equation}\label{rwefwerfwer}\xymatrix{X\ar[r]^{g}\ar[d]^{q} &Y \ar[d]^{p} \\Z \ar[r]^{f} &W } 
\end{equation}
 in $\LCH$.

\begin{lem}\label{joihetrgertgertgerg} Assume that  $p,q$ in the cartesian square \eqref{rwefwerfwer} are proper. Then
the square \begin{equation}\label{}\xymatrix{W\nCalg\ar[r]^{f^{*}}\ar[d]^{p^{*}} & Z\nCalg\ar[d]^{q^{*}} \\ Y\nCalg\ar[r]^{g^{*}} &X\nCalg } 
\end{equation}
is vertically right-adjoinable
\end{lem}
\begin{proof} We must show that 
 the Beck-Chevalley canonical
transformation
\begin{equation}\label{kjoegwegergwfg}  f^{*} p_{!}   \to  q_{!}q^{*}f^{*}p_{!}\cong q_{!}  g^{*}p^{*}p_{!}  \to  q_{!}g^{*}  :Y\nCalg\to Z\nCalg
\end{equation} 
 is an isomorphism.
For $A$ in $Z\nCalg$ this map is the restriction of
\begin{eqnarray*}
\prod_{z\in Z} \prod_{y\in p^{-1}(\{f(z)\})} A(y)&\to&   \prod_{z\in Z}\prod_{x\in q^{-1}(\{z\})}  \prod_{y\in p^{-1}(\{f(z)\})} A(y)\\&\cong&  \prod_{z\in Z}\prod_{x\in q^{-1}(\{z\})}  \prod_{x'\in q^{-1}( \{z\})} A(g(x'))\\&\to&
 \prod_{z\in Z}\prod_{x\in q^{-1}(\{z\})}   A(g(x))
\end{eqnarray*}
where the first map is the obvious diagonal map and the second map is the evaluation at $x=x'$.
Since \eqref{rwefwerfwer} is cartesian, for a fixed $z$ in $Z$ there is a bijection
$p^{-1}(\{f(z)\})\ni y\mapsto x_{y}\in q^{-1}(z)$ uniquely determined by the condition $g(x_{y})=y$.
Using this bijection, we can identify,  for each $z$ in $Z$, the inner products in the domain and target. 
Under this identification, the composition becomes the identity.
We observe that it preserves constant sections and is compatible with the multiplication by $C_{0}(Z)$.
Consequently, the isomorphism restricts to the isomorphism \eqref{kjoegwegergwfg}.
%
\end{proof}

\begin{rem}
The proof of \cref{joihetrgertgertgerg}
shows that \eqref{kjoegwegergwfg} is the inverse of  \eqref{fiqjiwoefqwefqewfqfw}. \hB
\end{rem}

   \subsection{Symmetric monoidal structure}\label{opkhetrhrtgergertg}
   
  In this section, we equip the functor \begin{equation}\label{gweirhgierowferfwefwerf}(-)\nCalg:\LCH^{\op}\to \Cat\ , \quad X\mapsto X\nCalg\ , \quad f\mapsto f^{*}\end{equation}  from \cref{kogpwtegrwref}
  with a lax symmetric monoidal structure.  Part of the lax symmetric monoidal structure is, for every two locally compact Hausdorff spaces $X$ and $Y$, the datum of a bifunctor
$$- \boxtimes-  :X\nCalg\times Y\nCalg\to (X\times Y)\nCalg\ .$$
This bifunctor has already been considered   in \cite{zbMATH00854200}, \cite{zbMATH01213731}, \cite{Popescu_2004}.
Since  $\Cat$ is a $2$-category,  defining a lax symmetric monoidal functor to $\Cat$ requires  providing  a considerable amount of additional data and  checking  many coherences. The main point of this section is to 
 provide this data completely and to check the necessary coherences, thereby keeping the 
  necessary work at a reasonable amount. We will furthermore check the projection formulas, which go into the proof
  of \cref{okpherthrtgetgetg} stating that $(-)\nCalg$ is a three-functor formalism.

      \begin{prop} \label{ijiopgwergrgwferfw}The 
  functor  in 
 \eqref{gweirhgierowferfwefwerf} 
canonically extends to a lax symmetric monoidal functor.
\end{prop}
\begin{proof}

     We start with the composition of lax symmetric monoidal functors
 $$\Phi:\LCH^{\op}\to \Set^{\op}\to \Cat\ , \qquad X\to X^{\delta}\to \prod_{X^{\delta}} \nCalg\ ,$$
 where
 the first functor sends $X$ to its underlying set $X^{\delta}$, and
 the second functor is the restriction of the lax symmetric monoidal functor $\Fun(-,\nCalg):\Cat^{\op}\to \Cat$ along $\Set\to \Cat$. The lax symmetric monoidal structure on the latter
  is given by the symmetric monoidal structure with the maximal tensor product
 on $\nCalg$ viewed as an object in $\CAlg(\Cat)$.  We denote the corresponding bifunctor by $\hat \boxtimes$. It is given for $X,Y$ in $\LCH$ and $(A_{x})_{x\in X}  $ in $\Phi(X)$ and $ (B_{y})_{y\in Y}$ in $\Phi(Y)$ by
$$(A_{x})_{x\in X}\hat \boxtimes (B_{y})_{y\in Y}=(A_{x}\otimes B_{y})_{(x,y)\in X\times Y}$$ in $\Phi(X\times Y)$.
 We  next construct  the  natural transformation
 $$N:(-)\nCalg\to \Phi\ , \quad X\mapsto \left(N_{X}:X\nCalg\to\Phi(X) , \quad A\mapsto  (A(x))_{x\in X}\right)\ .$$
 For a morphism  $f:X\to Y$ in $\LCH$ we have a canonical commutative square
 \begin{equation}\label{}
 \xymatrix{Y\nCalg\ar[r]^{f^{*}}\ar[d]^{N_{Y}} & X\nCalg\ar[d]^{N_{X}} \\ \Phi(Y) \ar[r]^{\Phi(f)} &\Phi (X)} 
\end{equation}
given on $B$ in $Y\nCalg$ by the  obvious identification  induced by \eqref{fqweihiquwoefqwef} at the lower right corner of
 \begin{equation}\label{}
 \xymatrix{ B\ar@{|->}[r]\ar@{|->}[d]  & f^{*}B \ar@{|->}[d] \\  (B(y))_{y\in Y}\ar@{|->}[r] &(B(f(x))_{x\in X}\cong ((f^{*}B)(x))_{x\in X}}\ .
\end{equation}
   If $g:Y\to Z$ is a second map, then the two fillers of the squares corresponding to $f$ and $g$ compose to the filler of the square corresponding to $g\circ f$. This finishes the construction of the transformation $N$.

 For every $X$ in $\LCH$  the functor $N_{X}$ is faithful since for every $A$ in $X\nCalg$ the homomorphism   $$\ev=(\ev_{x})_{x}:A\to \prod_{x\in X} A(x)$$
 is injective by \cref{lpherhrtgert}. So the map
$$\Hom_{X\nCalg}(A,B)\ni f\mapsto (f_{x}:=\ev_{x}(f):A(x)\to B(x))_{x\in X}\in \Hom_{\Phi(X)}(N_{X}(A),N_{X}(B))$$ is injective.
 
 The transformation $N$ thus  presents 
 $(-)\nCalg$ as a subfunctor of $\Phi$.
 We want to construct the lax symmetric monoidal structure on $(-)\nCalg$ by restriction along $N$. To this end, we must provide, for every $X,Y$ in $\LCH$, the bifunctor \begin{equation}\label{boijiotwgergwer}\boxtimes:X\nCalg\times Y\nCalg\to (X\times Y)\nCalg\end{equation}
 such that  \begin{equation}\label{boijiowgergwer}N_{X\times Y}(A\boxtimes B)\cong N_{X}(A)\hat \boxtimes N_{Y}(B)\ .
\end{equation} We then show that the structure maps can be inherited from those of $\Phi$.
 
   We start with the bifunctor  in \eqref{boijiotwgergwer} whose construction was  already proposed in \cite[Sec. 2]{zbMATH01213731}.
 For $A$ in $X\nCalg$ and $B$ in $Y\nCalg$ we provisionally  define the underlying $C^{*}$-algebra of
 $A\boxtimes B$ as the maximal tensor product $A\otimes B$.
 The structure of $A\boxtimes B$ as an $(X\times Y)$-algebra is  given by
 $$C_{0}(X\times Y)\otimes (A\otimes B)\cong (C_{0}(X)\otimes A)\otimes (C_{0}(Y)\otimes B)\stackrel{m\otimes n}{\to} A\otimes B\ , $$ where   $m$ and $n$ are the multiplication maps for  $A$ and $B$, and 
 we use the isomorphism
 $C_{0}(X\times Y)\cong C_{0}(X)\otimes C_{0}(Y)$, and the fact that the maximal tensor product of two surjections is again surjective. We have a family of maps $$\ev_{(x,y)}:=\ev_{x}\otimes \ev_{y} :A\otimes B\to A(x)\otimes B(y)\ , \quad (x,y)\in X\times Y$$
 which provide a map
 \begin{equation}\label{her90tiertgtegertgerg}\ev:A\otimes B\to \prod_{(x,y)\in X\times Y}A(x)\otimes B(y)\ .
\end{equation}
\begin{lem}[{\cite[Sec. 2]{zbMATH01213731}}]
The map $\ev$  in \eqref{her90tiertgtegertgerg} is injective, its image satisfies the equivalent conditions in \cref{jgwrieogjwefwefwrf}, and
we have an isomorphism  $$(A\boxtimes B)(x,y)=A(x)\otimes B(y)\ .$$
\end{lem}
\begin{proof}

 By the exactness of the maximal tensor product, for every  $(x,y)$ we have an exact sequence (see \cref{kophetrhrtgtrge})
 $$0\to A(X\setminus \{x\})\otimes B+A\otimes B(Y\setminus \{y\})\to A\otimes B\stackrel{\ev_{(x,y)}}{\to} A(x)\otimes B(y)\to 0\ .$$
 The kernel of $\ev_{(x,y)}$ is generated by
elements of the form
$ \pr_{X}^{*}f a\otimes  \pr_{Y}^{*}gb+ \pr_{X}f'a\otimes  \pr_{Y}^{*}g' b$, where $f,f'$  are in $C_{0}(X)$ with $f(x)=0$ and $g,g'$ are  in $C_{0}(Y)$ with $g'(y)=0$. 
We can write this element as
$$ (m\otimes n)((\pr_{X}^{*}f  \otimes  \pr_{Y}^{*}g + \pr_{X}f'\otimes  \pr_{Y}^{*}g' )\otimes  a\otimes b) \ .$$
The functions of the form $(\pr_{X}^{*}f  \otimes  \pr_{Y}^{*}g + \pr_{X}f'\otimes  \pr_{Y}^{*}g' )$ generate
$C_{0}(X\times Y\setminus \{(x,y)\})$.
It follows that
$\ker(\ev_{(x,y)})=(A\boxtimes B)(X\times Y\setminus \{(x,y)\})$.
We conclude that 
\begin{equation}\label{griegowerjkfoiwerfwerf}
(A\boxtimes B)(x,y)=A(x)\otimes B(y)\ .
\end{equation} and that
$\ev$ in \eqref{her90tiertgtegertgerg} is injective  by  \cref{lpherhrtgert}.
  \end{proof}
In view of the above considerations, we now change the definition of $\boxtimes$ by replacing 
$A\boxtimes B$  with its isomorphic  image under $\ev$ in $\prod_{(x,y)}A(x)\otimes B(y)$. Then $A\boxtimes B$
  is generated by the products of 
  the constant sections
 $ (\ev_{x}(a)\otimes \ev_{y}(b))_{(x,y)}$  with elements of $C_{0}(X\times Y)$.
  The relation \eqref{boijiowgergwer} follows from \eqref{griegowerjkfoiwerfwerf}.
 
By \cref{lkopherthertgetg}, we have an isomorphism \begin{equation}\label{gwepoiokgwerfewrfwergwegwerfwf}
 A\boxtimes B\cong C_{0}(X\times Y,A\otimes B)/I\ ,
\end{equation}
where $I$ consists of functions $\phi$ such that
$(\ev_{x}\otimes \ev_{y})(\phi) =0$ for all $(x,y)$ in $X\times Y$.

%

   Let $f:X\to Y$ be a morphism in $\LCH$ and $Z$ in $\LCH$.
 Then we have a commutative square
 \begin{equation}\label{}\xymatrix{\Phi(Y)\times \Phi(Z)\ar[r]^{\Phi(f)\times \id}\ar[d]^{\hat \boxtimes} &\Phi(X)\times \Phi(Z) \ar[d]^{\hat \boxtimes} \\ \Phi(Y\times Z) \ar[r]^{\Phi(f\times \id)} &\Phi(X\times Z) } \ .
\end{equation}

On $(A,B)$ in $ \Phi(Y)\times \Phi(Z)$ 
it is given by the obvious identity
 \begin{equation}\label{}\xymatrix{((A_{y})_{y\in Y} , (B_{z})_{z\in Z})\ar@{|->}[r]\ar@{|->}[d]  &((A_{f(x)})_{x\in X} ,(B_{z})_{z\in Z} ) \ar@{|->}[d]  \\  (A_{y}\otimes B_{z})_{(y,z)\in Y\times Z} \ar@{|->}[r] & (A_{f(x)}\otimes B_{z})_{(x,z) \in X\times Z} } 
\end{equation}
at the lower right corner. If $A$ is in $Y\nCalg$ and $B$ is in $Z\nCalg$ and $(A_{y})_{y\in Y}=(A(y))_{y\in X}$ and $(B_{z})_{z\in Z}=(B(z))_{z\in Z}$, then 
this identity is  the image
of an isomorphism $(f\times \id)^{*} (A\boxtimes B)\cong (f^{*}A)\boxtimes B$
 under $N_{X\times Z}$ of a morphism in $(X\times Z)\nCalg$ which defines the filler of 
 \begin{equation}\label{fwewfef}\xymatrix{ Y\nCalg\times Z\nCalg\ar[r]^{f^{*}\times \id}\ar[d]^{\boxtimes} &X\nCalg\times Z\nCalg \ar[d]^{\boxtimes} \\ (Y\times Z)\nCalg \ar[r]^{(f\times \id)^{*}} &(X\times Z) \nCalg} \ .
\end{equation}
Indeed, the identification preserves constant sections generated by elements $a\otimes b$ for $a$ in $A$ and $b$ in $B$ and commutes with the multiplication by functions in $C_{0}(X\times Z)$.
 This finishes the description of the data of the bifunctor $\boxtimes$. 
 The necessary relations are inherited from those of $\hat \boxtimes$.

We now consider the structure maps of the lax symmetric monoidal structure.
We consider the  associators. For $X,Y,Z$ in $\LCH$ we have the isomorphism
 $$(\Phi(X)\hat \boxtimes \Phi(Y))\hat \boxtimes \Phi(Z)\stackrel{\cong}{\to} \Phi(a)(
 \Phi(X)\hat \boxtimes (\Phi(Y)\hat \boxtimes \Phi(Z)))$$
 in $\Phi((X\times Y)\times Z)$
  induced by the associator on $\nCalg$,
 where $a:(X\times Y)\times Z\to X\times (Y\times Z)$ is the associator of the structure on $\LCH$.
  On objects $A,B,C$ in the respective categories, this associator map is given by
 the family of associators \begin{equation}\label{gwuegijeorgwreg} (A_{x}\otimes B_{y})\otimes C_{z}\stackrel{\cong}{\to} 
 A_{x}\otimes (B_{y}\otimes C_{z})
\end{equation}  for all $((x,y),z)$ in $(X\times Y)\times Z$.
 We must show that this map  is the image under $N_{(X\times Y)\times Z}$ of an isomorphism  
 \begin{equation}\label{gwerojgokpwergerwf}(A\boxtimes B)\boxtimes C\to a^{*}(A\boxtimes (B\boxtimes C))\ .
\end{equation} 
 The family of maps \eqref{gwuegijeorgwreg} induces a map 
 $$\prod_{((x,y),z)\in (X\times Y)\times Z} (A(x)\otimes B(y))\otimes C(z) \to \prod_{((x,y),z)\in (X\times Y)\times Z} A(x)\otimes (B(y)\otimes C(z))\ .$$
 This map is compatible with the multiplication by $C_{0}((X\times Y)\times Z)$ and preserves constant sections. 
 Consequently, it  restricts to  a map as in  \eqref{gwerojgokpwergerwf}, which under 
  $N_{(X\times Y)\times Z}$ is mapped to  the family \eqref{gwuegijeorgwreg}.
  
 We obtain the unit and symmetry constraints for $\boxtimes$ in a similar manner. The necessary relations are inherited from those of $\hat \boxtimes$.
 \end{proof}

 \begin{kor}\label{herthrtegtrgtrgtrhertherth}
 The functor $(-)\nCalg:\LCH^{\op}\to \Cat$
  refines to a functor with values in $\CAlg(\Cat )$.
  \end{kor}
\begin{proof} This is the standard fact  using that the symmetric monoidal structure on $\LCH^{\op}$ is induced by the cartesian symmetric monoidal structure on $\LCH$. More concretely, the reason is   that every object of $\LCH$ is a coalgebra with the diagonal as structure map.
\end{proof}

We let $\otimes_{X}$ denote the  induced symmetric monoidal structure on $X\nCalg$. In the literature, it is called the maximal  
  balanced tensor product over $X$ \cite{zbMATH00854200}, \cite{zbMATH01213731}, \cite{Park_2000}, \cite{Popescu_2004}.

\begin{rem}For $A,B$ in $X\nCalg$, we have
$$A\otimes_{X} B\cong d^{*}(A\boxtimes B)\ ,$$
where $d:X\to X\times X$ is the diagonal. 
For $f:X\to Y$, the  commutative diagram
\begin{equation}\label{}\xymatrix{X\ar[r]^{f}\ar[d]^{d_{X}} &Y \ar[d]^{d_{Y}} \\ X\times X\ar[r]^{f\times f}&Y\times Y }
 \end{equation}
again implies that $$f^{*}: Y\nCalg\to X\nCalg$$
canonically refines to a symmetric monoidal functor.
\hB
 \end{rem}


 For the proof of \cref{okpherthrtgetgetg}, we must verify various
  projection formulas.
 Let $f:X\to Y$ be a morphism in $\LCH$.
  \begin{lem}\label{okhoperhtrgertgetg}
Let $f:X\to Y$ be   a morphism  in $\LCH$, $A$ be in $X\nCalg$ and $B$ be in $Y\nCalg$. Then
 we have an isomorphism 
\begin{equation}\label{egwregwerfw}f_{!}(A)\otimes_{Y} B \stackrel{\cong}{\to} f_{!}(A\otimes_{X} f^{*}B)  \ .
\end{equation}   
 \end{lem}
\begin{proof}
By definition
$$f_{!}(A)\otimes_{Y} B\subseteq \prod_{y\in Y} \left(\prod_{x\in f^{-1}(\{y\})}    A(x)\right)\otimes B(y)$$
is generated by sections of the form  \begin{equation}\label{t234ju39ot2rrrr43}((\kappa(x) \ev_{x}(a))_{x\in f^{-1}(\{y\})}\otimes \ev_{y}( b) )_{y\in Y}\end{equation} for $a$ in $A$, $b$ in $B$ and $\kappa$ in $ C_{0}(X)$.
Similarly
 $$ f_{!}(A\otimes_{X} f^{*}B)\subseteq \prod_{y\in Y} \prod_{x\in f^{-1}(\{y\})} \left(A(x)\otimes B(f(x))\right)$$
 is generated by sections of the form \begin{equation}\label{t234ju39ot243}((\ev_{x}(a)\otimes \kappa(x)\ev_{f(x)}( b))_{x\in f^{-1}(\{y\})})_{y\in Y}
\end{equation}
 for  $a$ in $A$,  $b$ in $B$, and $\kappa$ in $C_{0}(X)$. The map
 \eqref{egwregwerfw} will be obtained by restriction of 
 the canonical map
\begin{equation}\label{hprtoekpoherthrtgert}\prod_{y\in Y} \left(\prod_{x\in f^{-1}(\{y\})}    A(x)\right)\otimes B(y) \to  \prod_{y\in Y} \prod_{x\in f^{-1}(\{y\})} \left(A(x)\otimes B(f(x))\right)\ .
\end{equation} 
This map sends a generator of the form \eqref{t234ju39ot2rrrr43}
 to the generator
$$((\ev_{x}(a)\otimes  \kappa(x)\ev_{f(x)}(b))_{x\in f^{-1}(\{y\})})_{y\in Y}\ ,
$$
which has the form \eqref{t234ju39ot243}.
Consequently, the map \eqref{hprtoekpoherthrtgert} restricts to a map \eqref{egwregwerfw}. Vice versa,
the generator \eqref{t234ju39ot243}  has a preimage of the form \eqref{t234ju39ot2rrrr43}.
 This already shows that \eqref{egwregwerfw} is surjective. We now show that this map is injective.
We consider an element $c$ in the kernel.    It suffices to check injectivity on the level of fibres at the points $y$ in $Y$.
We use the notation   $i:\{y\}\to Y$, $j:f^{-1}(\{y\})\to X$ and $k:f^{-1}(\{y\})\to \{y\}$.
The map is then given  by
\begin{align*}(f_{!}A\otimes_{Y} B)(y)\cong i^{*} (f_{!}A\otimes_{Y} B)  \cong  (f_{!}A)(y)\otimes B(y)\cong (j^{*}A)(f^{-1}(\{y\}))\otimes B(y)&\\ \to  (j^{*}A\otimes_{f^{-1}(\{y\})}  k^{*} B(y))(f^{-1}(\{y\}))\cong  j^{*}(A\otimes_{X} f^{*} B) (f^{-1}(\{y\}))\cong f_{!}(A\otimes_{X}f^{*}B)(y)\end{align*}
It suffices to show that 
$$(j^{*}A)(f^{-1}(\{y\}))\otimes B(y)\to  (j^{*}A\otimes_{f^{-1}(\{y\})}  k^{*} B(y))(f^{-1}(\{y\}))$$ is injective.
But this map has an inverse which sends
a generator $(\ev_{x}(a) \otimes \kappa(x) \ev_{y}(b))_{x\in f^{-1}(\{y\})}$
to $(\kappa(x)\ev_{x}(a) \otimes  \ev_{y}(b))_{x\in f^{-1}(\{y\})}$.
\end{proof}

\begin{rem}
 Note that the projection formula in \cref{okhoperhtrgertgetg} does not require any condition on the map $f$. But for the proof  of  \cref{okpherthrtgetgetg}, we must know that 
 for $f$ a proper map or an open inclusion, this isomorphism is given by the canonical map coming from the lax symmetric monoidal structure and the units or counits of the respective adjunctions. This is the content of \cref{kohperthertgertgertg}
 and \cref{okhoperhtrgertgetg1} below.
 \hB \end{rem}

  \begin{lem}\label{kohperthertgertgertg}
 If $f:X\to Y$ is an open inclusion, $A$ is in $X\nCalg$ and $B$ is in $Y\nCalg$,
then the canonical morphism
 \begin{equation}\label{gerfwefwfwererf45}f_{!}(A\otimes f^{*}B)\to f_{!}(A)\otimes B
\end{equation}  is an isomorphism.
 \end{lem}
\begin{proof}
Note that the canonical morphism \eqref{gerfwefwfwererf45} is 
$$f_{!}(A\otimes f^{*}B)\to f_{!}(f^{*}f_{!}A \otimes f^{*}B)\cong f_{!}f^{*}( f_{!}A\otimes  B)\to f_{!}A\otimes B\ ,$$
which uses the fact that $f^{*}$ is symmetric monoidal and the unit  and counit of the adjunction $f_{!} \dashv f^{*}  $.
  Using the explicit formulas for the operations and for the unit and counit from \cref{kopherhrtgertge},
 the canonical morphism   \eqref{gerfwefwfwererf45} is the restriction of the map 
 \begin{eqnarray*}\prod_{y\in Y}  \left(\prod_{x\in f^{-1}(\{y\})}A(x)\otimes B(f(x)) \right)&\to& \prod_{y\in Y} \left(\prod_{x\in f^{-1}(\{y\})}  \left(\prod_{x'\in f^{-1}(\{f(x)\})}A(x')\right)\otimes B(f(x)) \right)
 \\&\to  & \prod_{y\in Y}  \left(\prod_{x\in f^{-1}(\{y\})}A(x)\otimes B(f(x))\right)
 \end{eqnarray*}
 where the first map is the diagonal and the second map is the evaluation at $x=x'$.
 Using the fact  that $f$ is injective, these maps are both isomorphisms.
 \end{proof}
 
 \begin{rem}
 By an inspection one observes that  the  canonical map  \eqref{gerfwefwfwererf45} is equal to the inverse of the isomorphism  \eqref{egwregwerfw}. 
 \end{rem}

\begin{lem}\label{okhoperhtrgertgetg1}
If  $f:X\to Y$ is   a proper morphism  in $\LCH$, $A$ is in $X\nCalg$ and $B$ is in $Y\nCalg$, then
  the canonical morphism
 \begin{equation}\label{vijiowefvwefv}f_{!}(A)\otimes B \to f_{!}(A\otimes f^{*}B)  
\end{equation}is an isomorphism.
 \end{lem}
\begin{proof}
Note that the canonical morphism \eqref{vijiowefvwefv} is given by  \begin{equation}\label{gweoijowgei}f_{!}(A)\otimes B\to f_{!}f^{*}(f_{!}A\otimes B)\cong  f_{!}(f^{*}f_{!}A\otimes f^{*}B)\to f_{!}(A\otimes f^{*}B)
\end{equation}
 which uses the fact that $f^{*}$ is symmetric monoidal and the unit  and counit of the adjunction $f^{*} \dashv f_{!} $.
 We use the explicit formulas for the operations and units and counits in order to write \eqref{vijiowefvwefv} as the restriction of the composition \begin{eqnarray*}
\prod_{y\in Y} \left(\prod_{x\in f^{-1}(\{y\})}    A(x)\right)\otimes B(y)&\to &
 \prod_{y\in Y} \prod_{x'\in f^{-1}(\{y\})}  \left( \left( \prod_{x\in f^{-1}(\{f(x')\})}A(x)\right)\otimes B(y) \right)\\&\cong&
  \prod_{y\in Y} \prod_{x'\in f^{-1}(\{y\})}  \left( \left( \prod_{x\in f^{-1}(\{f(x')\})}A(x)\right)\otimes B(f(x')) \right)
  \\&\to& \prod_{y\in Y}\left( \prod_{x\in f^{-1}(\{y\})} A(x)\otimes B(y)\right)
\end{eqnarray*}
 where the first map is the diagonal and the second is the evaluation at $x=x'$.
This is precisely the map \eqref{hprtoekpoherthrtgert}.
Hence, the assertion follows from \cref{okhoperhtrgertgetg}.
%
%
%
%
%
%
\end{proof}

\subsection{Proof of \cref{okpherthrtgetgetg}}\label{okpherthrtgetgetg1}

By \cref{ijiopgwergrgwferfw} and \cref{herthrtegtrgtrgtrhertherth} we have a   functor
\begin{equation}\label{hrtophkpretgrtgtreg}(-)\nCalg:\LCH^{\op}\to \CAlg(\Cat)\ .
\end{equation}
We must verify the conditions from \cref{kojpjrtzjzhrtzh}. 

We argue that the functor in \eqref{hrtophkpretgrtgtreg}
  is compatible with $I$. The existence of left-adjoints for open embeddings
follows from  \cref{kopherhrtgertge}. The base change
and projection formulas have been verified in \cref{oigjworiegrefwerfewfref}
and \cref{kohperthertgertgertg}. 

The functor in \eqref{hrtophkpretgrtgtreg} is furthermore compatible with $P$.  The existence
of right-adjoints for proper maps follows from \cref{okhprethetgertgtge}. The base change
and projection formulas have been verified in \cref{joihetrgertgertgerg}
and \cref{okhoperhtrgertgetg1}.

Finally the mixed Beck-Chevalley condition  follows from the fact that
for proper maps $f$, the right-adjoint of $f^{*}$ is given by the  the covariant functoriality $f_{!}$.

    \subsection{Good objects}\label{koperthrtgetrgetgt}
    
    In this section, we analyze to what extend  the functor $c$ in \eqref{gjiojowiejroijfwer} preserves exact sequences.  Let $X$ be a locally compact Hausdorff space, and let $A$ be in $\Open(X)\nCalg$.
 Then we can consider $C_{0}(X)\otimes A\cong C_{0}(X,A)$.
 In this algebra, we have two ideals:
 \begin{enumerate}
 \item \label{ijgowerfwerf} The ideal $J$ is generated by tensors $\phi\otimes a,$ where
 $\phi\in C_{0}(U)$ and $a\in A(V)$  for some  open sets $U,V $ in $X$ with  $U\cap V=\emptyset$.
 \item The ideal $I$ (used in the proof of \cref{okprherhrtgertge})  which consists of functions $\phi$ in $C_{0}(X,A)$ with the property
 that $\phi(x)\in A(X\setminus \{x\})$ for all $x$ in $X$.
 \end{enumerate}
  
 \begin{lem}
 We have an inclusion $J\subseteq I$.
 \end{lem}
 \begin{proof} Assume that $\phi\otimes a$ is a generator of $J$ as in \ref{ijgowerfwerf}.
Then $\phi\otimes a$ corresponds to the function $x\mapsto \phi(x)a$.
 If $x\not\in U$, then  $\phi(x)a=0\in A(X\setminus \{x\})$ since $\phi(x)=0$. If $x\in U$, then $x\not\in V$ and hence $V\subseteq X\setminus \{x\}$.
 Consequently, $\phi(x)a\in A(X\setminus\{x\})$ since $a\in A(X\setminus\{x\})$.
 This implies that all generators of $J$ belong to $I$ and therefore proves  the assertion.
 \end{proof}

   \begin{ddd}
An object of $A$ in $\Open(X)\nCalg$ is called good if $I=J$.
We let $\Open(X)^{\gd}\nCalg$ be the full subcategory of $\Open(X)\nCalg$ on 
objects that are good. 
 \end{ddd}

In contrast to the conventions used elsewhere in the present paper, a sequence in  $\Open(X)^{\gd}\nCalg$  will be called exact if it is  exact in $ \Open(X) \nCalg$.
Our main reason for introducing good objects is the following fact:

\begin{lem} \label{gjweroijgorefwerf}The restriction $$c_{|\Open(X)^{\gd}\nCalg}:\Open(X)^{\gd}\nCalg\to X\nCalg$$ of the functor $c$ in \eqref{gioujgiotrgwergwerfwerf} 
 preserves exact sequences.
\end{lem}
\begin{proof}    Assume that $$0\to A\to B\to C\to 0$$ is exact in $\Open(X)^{\gd}\nCalg$. Then we must show that 
$$0\to c(A)\to c(B)\to c(C) \to 0$$  is exact in $X\nCalg$. We use that by \cref{kophrtgretgggrethrthe}, exactness in $X\nCalg$ can be checked on the underlying algebras.
Let $I,I',I''$ and $J,J',J''$ be the ideals associated to $A,B$, and $C$ respectively.
We consider the map of vertical exact sequences 
 $$
\xymatrix{&0\ar[d]&\ar[d]0&\ar[d]0&\\
0\ar[r]&I\ar[d]\ar[r]&I'\ar[r]\ar[d]&I''\ar[r]\ar[d]&0\\0\ar[r]&C_{0}(X)\otimes A\ar[d]\ar[r]&C_{0}(X)\otimes B\ar[r]\ar[d]&C_{0}(X)\otimes C\ar[r]\ar[d]&0\\
0\ar[r]& c(A)\ar[r]\ar[d]&c(B)\ar[r]\ar[d]&c(C)\ar[r]\ar[d]&0\\&0&0&0&}\ .
$$
 
By homological algebra and the exactness of the middle sequence, we get isomorphisms
$$\coker(c(B)\to c(C))\cong 0\ ,$$
$$\ker(c(A)\to c(B))\cong \ker(I'\to I'')/\im(I\to I')\ ,$$ and $$\ker(c(B)\to c(C))/\im(c(A)\to c(B))\cong \coker(I'\to I'')\ .$$
Assume now that $\phi$ is in  $\ker(I'\to I'')$.
Then $\phi\in C_{0}(X,A)$ and $\phi(x)\in B(X\setminus \{x\})$ for all $x$ in $X$.
Hence $\phi(x)\in A(X\setminus \{x\})$ for all $x$ and hence $\phi\in I$.
It follows that $\ker(c(A)\to c(B))\cong 0$.
 
 In order to show that  $\ker(c(B)\to c(C))/\im(c(A)\to c(B))\cong 0$ we use
goodness. It implies that $ I''$ is generated by 
elements of the form $\phi\otimes c$ with $\phi\in C_{0}(U)$ and $c\in C(X\setminus \overline{U})$ for $U$ in $\Open(X)$.
Since $B(X\setminus   \overline{U} ) \to C(X\setminus  \overline{U})$ is surjective, we can  lift such elements to generators of $I'$.
Hence $I'\to I''$ is surjective, which implies the desired result.
 \end{proof}

 \begin{lem}\label{kophtrtegertgertgte} We have $\Open(X)^{\reg}\nCalg\subseteq \Open(X)^{\gd}\nCalg$.
 \end{lem}
 \begin{proof}  
 We consider $A$  in $\Open(X)^{\reg}\nCalg$. Let
 $\phi$  be in $I$. We must show that  $\phi\in J$. It suffices to show that
 $\phi$ can be approximated by elements of $J$.
 
 We first show that elements $\phi$  in $I$ can be approximated by elements with compact support. Since the condition $\phi\in I$ is pointwise, for all  $\kappa$  in $C_{0}(X)$ we also have  $\kappa   \phi\in I$.
Fix $\epsilon$ in $(0,\infty)$. 
 The subset $\{x\in X\mid \|\phi(x)\|\ge \epsilon\}$ is compact.
  We can choose $\kappa $   in $C_{c}(X)$ such that $\|\kappa\|\le 1$ and $\kappa\equiv 1$ on this subset.  Then $\|\phi-\kappa \phi\|\le \epsilon  $. Note that $\kappa\phi$ has compact support.

From now one, we assume that $\phi$ has compact support and show that it can be approximated by elements of $J$.

Fix again $\epsilon$ in $(0,\infty)$. 
 For every $x$ in $X$
we have $\phi(x)\in A(X\setminus \{x\})$. By regularity of $A$
we have $$A(X\setminus \{x\})=\overline{\bigcup_{U\Subset  X\setminus \{x\}} A(U)}\ .$$
 Using this and continuity of $\phi$ we can therefore 
  fix an open neighborhood $U_{x}$ of $x$  and $a_{x}$ in $A(X\setminus \bar U_{x})$
  such that   $$\|\phi(x)-a_{x}\|\le \epsilon/2 \ ,\quad 
 \| \phi(x)- \phi(y)\|\le \epsilon/2 $$  for all $y$ in $U_{x}$.
In total, we have
$$\| a_{x}-\phi(y)\|\le \epsilon$$ for all $y$ in $U_{x}$.

Since $\supp( \phi)$ is compact, we can choose a finite
set $(x_{k})_{k\in K}$ of points in $X$ such that $\supp( \phi)\subseteq \bigcup_{k\in K}U_{x_{k}}$.
We further choose a partition of unity $ (\chi_{k})_{k\in K}$ subordinated to this covering.
Then for every $y$ in $X$ we have  $$ \phi(y)=\sum_{k\in K}\chi_{k}(y)  \phi(y)  \ .$$
We define
$$\tilde \phi(y):=\sum_{k\in K}  \chi_{k}(y) \otimes a_{x_{k}} \ .$$
Then,
$$\|\phi(y)-\tilde \phi(y)\|\le \sum_{k\in K} \chi_{k}(y)  \|\phi(y)- a_{x_{k}}) \|\le 
\sum_{k\in K}\chi_{k}(y) \epsilon=\epsilon\ .$$
 By construction, we have  $\tilde \phi\in J$ and $\| \phi -\tilde \phi\| \le \epsilon$.  

We conclude that every element of $I$ can be approximated by elements of $J$.  
\end{proof}

 Let $X$ be a locally compact Hausdorff space, let  $Y$ be a locale,  and  let $f:X\to Y$ be a map of locales. 
 We let $f^{-1}:\cP(Y)\to \Open(X)$ be the corresponding map of frames. The following result provides   an extension of  the adjunction obtained in \cref{okhprethetgertgtge} to a  bigger class of maps.
 In the following statement,  $f_{!}$ is the morphism of $E$-theory contexts  as in \cref{lkohperghtrgertgertg9}.
 \begin{lem} \label{rgrokpowergwergwerferwf}  If $f^{-1}$ is perfect (\cref{herthertgergtrgtrge}), then we have an adjunction  \begin{equation}\label{boiui0wru90gbgdfb}f^{*}:Y\nCalg\leftrightarrows X\nCalg:f_{!}
\end{equation} 
and $f^{*}$ is a morphism of $E$-theory contexts. 
 \end{lem}
\begin{proof}
 
   By  \cref{hiuqhiufhweifqwefqwefqw}, we have the adjunction  
 \begin{equation}\label{gurt9gertgertg}f^{-1}:\cP(Y)\leftrightarrows \Open(X):f_{\sharp}\end{equation}   in $\Poset^{\uc}$, 
 where $\cP(Y)$ denotes the poset corresponding to $Y$. Since $(-)^{\uc}\nCalg$ is a functor to $E$-theory contexts defined on $\Poset^{\uc}$ we get   an adjunction
  \begin{equation}\label{oijviowfvwefd}f^{\sharp}:\cP(Y)\nCalg\leftrightarrows \Open(X)^{\uc}\nCalg:f_{!}\ ,
\end{equation}
 where $f_{!}=(f^{-1})^{*}$ and $f^{\sharp}=(f_{\sharp})^{*}$ are morphisms of $E$-theory contexts.
 If $f^{-1}$  is perfect, then 
 $f_{\sharp}$ preserves filtered joins,
 and the adjunction \eqref{gurt9gertgertg} lifts to one in $\Poset^{\prfr}_{(*)}$ (see \cref{tklphertrtgergrgret}). Since $(-)^{\reg}\nCalg$ is a functor to $E$-theory contexts defined on $\Poset_{(*)}^{\prfr}$ we then get 
 an adjunction
 \begin{equation}\label{vodsisdf0vuis0d9fuvsdfvsdfvs}f^{\sharp}:\cP(Y)^{\reg}\nCalg\leftrightarrows \Open(X)^{\reg}\nCalg:f_{!}\ ,
\end{equation}
where again $f_{!}$ and $f^{*}$ are morphisms of $E$-theory contexts. 
 In particular, $f^{\sharp}$ takes values in $\Open(X)^{\gd}\nCalg$ by \cref{kophtrtegertgertgte}.
 Composing the adjunctions  \eqref{vodsisdf0vuis0d9fuvsdfvsdfvs} and  \eqref{gioujgiotrgwergwerfwerf},
 we  finally get the adjunction
 \begin{equation}\label{boiui0wru90gbgdfb}f^{*}:Y\nCalg\leftrightarrows X\nCalg:f_{!}
\end{equation} 
 with $f^{*}=c\circ f^{\sharp}$ and $c$ as in \eqref{gioujgiotrgwergwerfwerf}. 
 Since $c$ preserves filtered colimits, finite products, and also exact sequences when restricted to 
 $\Open(X)^{\gd}\nCalg$ we see that $f^{*}$ is also a morphism of $E$-theory contexts.
   \end{proof}

\subsection{Proof of \cref{kophejhthgerthetrh} }\label{kophejhthgt5t5erthetrh1}
By \cref{ijiopgwergrgwferfw}, we have   the lax symmetric monoidal functor
$$(-)\nCalg:\LCH^{\op}\to  \Cat
\ , \quad X\mapsto  X\nCalg\ , \quad  f\mapsto f^{*}\ .$$
 We must show that it satisfies the conditions listed in \cref{tohpertgertgrteg}. 
 
 We clearly have $\{*\}\nCalg\simeq \nCalg$ as a symmetric monoidal category.

 Recall from \cref{gwergerfwerfw} the definition of an $E$-admissible map.
 \begin{lem} \label{hlepertgertger}For every morphism $f:X\to Y$ in $\LCH$
the functor 
$f^{*}:X\nCalg\to Y\nCalg$ is $E$-admissible. \end{lem}
\begin{proof}
We must show that $f^{*}$ preserves finite products, filtered colimits, and exact sequences.

We first assume that  $f$ is proper. Then $f^{*}$ is a left-adjoint and preserves colimits.
In order to see that it preserves exact sequences and finite products,
we use the bold  commutative diagram and the adjoints as indicated.
 $$ \xymatrix{ \Open(X)^{\reg}\nCalg\ar[r]^{\perp}_{f_{!}}  \ar@/^-0.5cm/@{..>}[d]^{ \dashv}_{c} & \ar@/^-0.5cm/@{..>}[d]^{ \dashv}_{c}\ar@/^-0.5cm/@{..>}[l]_{f^{\sharp}}\Open(Y)^{\reg}\nCalg \\ \ar[u]_{\incl'}X\nCalg\ar[r]_{f_{!}}^{\perp} &  \ar@/^-0.5cm/@{..>}[l]_{f^{*}}  Y\nCalg\ar[u]_{\incl}} \ .
  $$
Since $\incl$ is fully faithful, we get $f^{*}\simeq c\circ  f^{\sharp}\circ \incl$, where $f^{\sharp}$ is the pull-back along the    right-adjoint in the poset adjunction $$f^{-1}:\Open(Y)\leftrightarrows  \Open(X):f_{\sharp}$$ which is a preframe morphism since $f$ is proper, and therefore $f^{-1}$ is a  perfect frame morphism (see \cref{tklphertrtgergrgret} and \cref{kopjrtzjzthzhtzhrtzhtz}).
Now $c$, $f^{\sharp}$ and $\incl$ are all $E$-admissible.   In detail, 
in order to see that  $c$ preserves  exact sequences, we use that regular objects are good and that $c$ restricted to good objects preserves
exact sequences, see  \cref{gjweroijgorefwerf}. Further $c$ preserves finite products by an inspection of the constructions.
Since $c$ is a left-adjoint, it preserves  colimits.

By a combination of
\cref{kohperthertgertgetrg}  and \cref{kohperthertgertgetrg1}, the functor $\incl$ 
  is $E$-admissible.
The morphism $f^{\sharp}$  is the pull-back morphism along a preframe morphism,
and is  therefore   $E$-admissible by \cref{okpherhtrgertgetrg}.
  We conclude that $f^{*}$ is $E$-admissible    provided $f$ is proper.

We now assume that $f:X\to Y$ is an open embedding. By  \cref{kopherhrtgertge},
the functor $f^{*}$ is the right-adjoint of the functor $f_{!}$. By  \cref{okhpzrhtzjrtzj},
the right-adjoint of $f_{!}$ is $E$-admissible. Therefore 
 $f^{*}$ is $E$-admissible.
 
Since every morphism $f$ in $\LCH$ is a composition of a proper map and an open inclusion, we conclude that $f^{*}$  is $E$-admissible  for every morphism $f$ in $\LCH$. 
   \end{proof}

\begin{lem}\label{khrptohertgertgrtge} For $X$ and $Y$ in $\LCH$ and
 every $A$ in $X\nCalg$ the functor $$A\boxtimes -:Y\nCalg\to (X\times Y)\nCalg$$   is $E$-admissible. 
\end{lem}
\begin{proof} We start by showing that it preserves exact sequences.
Let $$0\to B\to C\to D\to 0$$ be an exact sequence in $Y\nCalg$.
We must show that then
$$0\to A\boxtimes B\to A\boxtimes C\to A\boxtimes D\to 0$$ is exact in $(X\times Y)\nCalg$.
 By \cref{kophrtgretgggrethrthe}.\ref{jiggtre2} and \cref{kophrtgretgggrethrthe}.\ref{jiggtre3},  the exactness of this sequence can be checked  on the level of underlying algebras. In view of the construction of $\boxtimes$ given in the proof of \cref{ijiopgwergrgwferfw},
 the assertion follows from the fact that
$$0\to A\otimes B\to A\otimes C\to A\otimes D\to 0$$
is exact in $\nCalg$ since the maximal tensor product preserves exact sequences.

Using \eqref{gwepoiokgwerfewrfwergwegwerfwf} it is easy to see that $A\boxtimes -$ preserves  finite products.

In order to show that    that $A\boxtimes -$  preserves  filtered colimits,
we again use \cref{kophrtgretgggrethrthe}.\ref{jiggtre} and that the maximal tensor product $A\otimes - $
on $\nCalg$ preserves filtered colimits.  
  \end{proof}

\cref{hlepertgertger} and \cref{khrptohertgertgrtge} verify the remaining conditions of  \cref{tohpertgertgrteg}. 

The following consequence of  \cref{kophejhthgerthetrh}  is surely well-known, but
a proof is difficult to locate in the literature.

\begin{kor} For every locally compact Hausdorff space $X$ and  $A$ in $X\nCalg$
the  (maximal balanced) tensor  product $$A\otimes_{X}-:X\nCalg\to X\nCalg$$
preserves exact sequences and filtered colimits. 
 \end{kor}
 \begin{proof}
 Using the diagonal $d:X\to X\times X$ we have an isomorphism
 of functors
  $$A\otimes_{X}-\cong  d^{*}\circ (A\boxtimes-)\ .$$ Since the right-hand side 
is a composition of two $E$-admissible functors, it is itself $E$-admissible. So $A\otimes_{X}-$  preserves in particular
filtered colimits and exact sequences.
 \end{proof}

\subsection{Approximation by finite spaces}

For a locally compact Hausdorff space
the collection of evaluations $$(X\nCalg\ni A\mapsto A(U)\in \nCalg)_{U\in \Open(X)}$$
is clearly conservative. The fact that this remains to be true on the level of $E$-theory
is a crucial input for the proof of our main \cref{trkopherthrrrrgertg}.
  In view of our construction of $E$-theory by forcing the universal properties,
  we do not have a direct access to the mapping spaces. The only way to get some control
  is to use adjunctions like in \cref{okpthpertherterg} below in order to relate mapping spaces in the $E$-theory
of complicated spaces with those   in the $E$-theory of simpler spaces.  
The proof of the following 
  theorem,  which states in a condensed way that the collection of evaluations
 $$(E(X)\ni A\mapsto A(U)\in \EE)_{U\in \Open(X)}$$ is jointly conservative,
 applies this idea.
  Recall the morphism $s_{X}:E(X)\to \CoShv(X,\EE)$ from  \cref{herthertgrtger}.
   
 \begin{theorem}\label{herthetrgertgetrg}
 For every $X$ in $\LCH$ the functor $s_{X}:E(X)\to \CoShv(X,\EE)$ is conservative.
\end{theorem}

\begin{proof} 
  Let $X$ be a locally compact Hausdorff space, and let  $f:X\to Y$ be a  map to a finite locale. Since $Y$ is finite, the map   $f^{-1}:\cP(Y)\to \Open(X)$ is perfect. Consequently, by 
 \cref{rgrokpowergwergwerferwf} 
 the adjunction  \begin{equation}\label{hrtehokpergtre}f^{*}:Y \nCalg \leftrightarrows X\nCalg:f_{!}\ .
\end{equation}
 is   an adjunction of $E$-theory contexts.

%
%
%
%

\begin{kor}\label{okpthpertherterg} If $X$ is a locally compact Hausdorff space, and $f:X\to Y$ is a morphism to a finite locale, then
  we have an adjunction 
\begin{equation}\label{hrtehovfvfvfvkpergtre}f^{*}:E(Y) \leftrightarrows E(X):f_{!}\ .
\end{equation}
\end{kor}

\begin{lem}\label{kohprthergetrgetrg}
If $X$  is a topological space, then
there exists a filtered system $(Y_{I})_{I\in \bI}$ of finite sober topological spaces such that
$$ \colim_{I\in \bI}  \Open(Y_{I})\cong   \Open(X)  $$ in $\Poset$.
\end{lem}
\begin{proof}
We let 
$\bI$ be the filtered system of finite subframes $(\Open(Y_{I}))_{I\in \bI}$ of  $\Open(X)$.
 The set of  points of $\Open(Y_{I})$ will be denoted by    $Y_{I}$. These points   are the frame morphisms $y:\Open(Y_{I})\to \{ 0\le 1\}$. For every $x$ in $X$ we get a point $$f_{I}(x): U\mapsto \left\{ \begin{array}{cc}0 & x\not\in U\\ 1& x\in U\end{array}\right.\ .$$  
  We claim that $\Open(Y_{I})$ is  spatial. In fact,  if $U,V$ are in $\Open(Y_{I})$ and $U\not\subseteq V$, then there exists $x$ in $V\setminus U$. Then $f_{I}(x)(V)=1$ and $f_{I}(x)(U)=0$.
  It follows that
 $\Open(Y_{I})$ is precisely the topology on the topological space $Y_{I}$ and
 $f_{I}:X\to Y_{i}$ is a continuous map.
 
 An inclusion $I\subseteq I'$ in $\bI$ is an inclusion of frames, and  hence induces a surjective map of topological spaces $f_{I',I}:Y_{I'}\to Y_{I}$.   They are the structure maps of the functor
 $$\bI^{\op}\to \Locale\ , \quad I\mapsto  Y_{I}\ .$$
 For every $I$ in $\bI$ we have  a surjective map of locales
 $f_{I}:X\to Y_{I}$.
 We claim that these maps present $\Open(X)$ as the colimit
 \begin{equation}\label{gwerfwrefwffwerf}\colim_{I\in \bI} \Open(Y_{I})\stackrel{\cong}{\to}\Open(X) \end{equation} in $\Poset$.
 Indeed, since every open subset of $U$ of $X$ belongs to  $\Open(Y_{I})$ for some $I$ in $\bI$  this is an isomorphism of underlying sets which turns out to be a poset-isomorphism.
%
%
%
%
%
%
%
%
\end{proof}

\begin{rem}
One must be careful, as \eqref{gwerfwrefwffwerf} is not true when one interprets the colimit in frames.
The colimit in frames is bigger.
On other words, we have a map of  locales
$$X\to \lim_{I\in \bI^{\op}} Y_{I}$$ which is not an isomorphism in general. \hB
\end{rem}

 We apply \cref{kohprthergetrgetrg} to a locally compact Hausdorff space $X$.
 \begin{lem}
If $X$ is a locally compact Hausdorff space, then the canonical functor
$$E(X)\to \lim_{I\in \bI^{\op}} E(Y_{I})$$
is fully faithful.
\end{lem}
\begin{proof}
In view of the presence of the adjunctions $f_{I}^{*}\dashv f_{I,!}$ obtained in \cref{okpthpertherterg},
 we   must show that for every $A$ in $E(X)$ the natural transformation
\begin{equation}\label{gkerogkwperfwerfrefwefwerfwef} \colim_{I\in \bI}f_{I}^{*}f_{I,!}A \to  A   \end{equation}   
is an equivalence.  These adjunctions come from adjunctions \eqref{hrtehokpergtre} on the level of $E$-theory contexts.
We  first show that for every $A'$ in $X\nCalg$ the transformation
$$ \colim_{I\in \bI}f_{I}^{*}f_{I,!}A' \to  A'$$ is an isomorphism. Since $c$ commutes with colimits
and $c\circ \incl\simeq \id$ for $\incl:X\nCalg\to \Open(X)^{\reg}\nCalg$
this is equivalent to the assertion that
$$  \colim_{I\in \bI}f_{I}^{\sharp}f_{I,!}  A' \to  A'$$ 
is an isomorphism in $ \Open(X)^{\reg}\nCalg$, where $f^{\sharp}_{I}$ is as in \eqref{vodsisdf0vuis0d9fuvsdfvsdfvs}  and we omitted the functor $\incl$ for better readability.
 Using the adjunctions \eqref{vodsisdf0vuis0d9fuvsdfvsdfvs},
we must show  for every $B'$ in $   \Open(X)^{\reg}\nCalg$ that
$$\Hom_{ \Open(X)^{\reg}\nCalg}(A',B')\cong  \lim_{I\in \bI}  \Hom_{\Open(Y_{I})^{\reg}\nCalg}(f_{I,!}A',f_{I,!}B')\ .$$
 The right-hand side is the intersection over $I$ in $\bI$  of the subsets of  $\Hom_{\nCalg}(A,B)$  of homomorphisms
 which are compatible with the $\Open(Y_{I})^{\reg}\nCalg$-structure. This intersection is precisely
$ \Hom_{ \Open(X)^{\reg}\nCalg}(A',B')$ since every open subset $U$ of $X$ belongs to $\Open(Y_{I})$ for some $I$ in $\bI$.

Since $\ee: X\nCalg\to E(X)$ preserves filtered colimits, we see that
the map \eqref{gkerogkwperfwerfrefwefwerfwef}, is an equivalence for   $A=\ee(A')$  for every $A'$ in $X\nCalg$. 
By the universal property of $\ee$ it is an eqivalence for all $A$ in $E(X)$.
 \end{proof}

\begin{kor}\label{gjkiowgwrefwerfrewf}
The collection of functors 
$$(f_{I,!}:E(X)\to E(Y_{I}))_{I\in \bI} $$ is jointly conservative.
\end{kor}

We now invoke  \cref{gweoirgjoergeferggr}  which states  the equivalences
$$E(Y_{I})\simeq \CoSh(Y_{I},\EE)$$ for all $I$ in $\bI$.
 Since equivalences on cosheaves are detected objectwise,  we conclude that
the collection 
$$\left(( F\mapsto F(U)) : E(Y_{I}) \to \EE\right)_{U\in \Open(Y_{I})}$$
is jointly conservative.
By  unfolding definitions, we see that for $U$ in $\Open(Y_{I})$ we have the equivalence
$$s_{X}(A)(U)\simeq s_{Y_{I}}(f_{I,!}A)(U) \ .$$
   Combining this with \cref{gjkiowgwrefwerfrewf}  we conclude that
 the functor $s_{X}:E(X)\to \CoShv(X,\EE)$ is conservative. This  finishes the proof of \cref{herthetrgertgetrg}
 \end{proof}

%
%
%
%


\begin{rem}
The classical  analog of \cref{herthetrgertgetrg} is \cite[Thm. 3.10]{Dadarlat_2012}. \hB
\end{rem}

\section{Functor formlisms}
 
\subsection{Three  functor formalisms }

 Following \cite{arXiv:2410.13038}, \cite{arXiv:2510.26269}, \cite{arXiv:2507.13537}, we consider the Nagata set-up $(\LCH,I,P)$ on  the category $\LCH$ of locally compact Hausdorff spaces, with the  subclasses of morphisms  
\begin{itemize}
\item $I$: the  open inclusions  
\item $P$: the proper maps.
\end{itemize}
We consider a functor
  \begin{equation}\label{vfdvsfvsdfvwr}D:\LCH^{\op}\to \CAlg(\Cat_{\infty}) \end{equation}
 \begin{rem} Since the symmetric monoidal structure on $\LCH^{\op}$ is induced by the cartesian structure on $\LCH$ this datum is equivalent to the datum of a lax symmetric monoidal functor
  $$D:\LCH^{\op}\to \Cat_{\infty}\ .$$  \hB\end{rem}
  
  \begin{rem}
  We will usually denote the contravariant functoriality of $D$ on morphisms $f$ in $\LCH$ by $f^{*}$,
but sometimes we also use the notation $D(f)$, in particular when there are different functors floating around. In the case of sheaves, we  will decorate the operations with $\hat{}$, i.e., we write $\hat f^{*}$. \hB \end{rem}

For $X$ in $\LCH$ we  will denote the tensor unit and tensor product in $D(X)$ by $1_{X}$ and $\otimes_{X}$.

We will consider the following groups of conditions on $D$. They are  standard conditions
considered in the theory of six-functor formalisms, as stated, e.g., \cite{arXiv:2410.13038},  \cite{arXiv:2510.26269}, \cite{arXiv:2507.13537}.
The wording is copied from \cite{buvo}.
 \begin{ddd}
We say that $D$ as in \eqref{vfdvsfvsdfvwr} is compatible with $I$ if:\begin{enumerate}
\item (I-la: left adjoints) For every $i:U\to X$ in $I$ the functor $i^{*}$ admits a left adjoint $i_{!}:D(U)\to D(X)$.
\item  (I-bc: base change)  For every cartesian square $$ \xymatrix{U\ar[r]^{i}\ar[d]^{g} &X \ar[d]^{f} \\ V\ar[r]^{j} &Y } $$ in $\LCH$ with $i,j$ in $I$
the square $$\xymatrix{ D(Y)\ar[r]^{j^{*}}\ar[d]^{f^{*}} & D(V)\ar[d]^{g^{*}} \\D(X) \ar[r]^{i^{*}} & D(U)} $$
is horizontally left-adjoinable. 
 \item (I-pf: projection formula) For every $i:U\to X$ in $I$ the canonical map $$i_{!}(-\otimes_{U} i^{*}(-))\to i_{!}(-)\otimes_{X} i^{*}(-)$$ is an equivalence.
 \end{enumerate}
\end{ddd}

\begin{ddd}\label{hkopertergtrgerg}
We say that $D$ as in \eqref{vfdvsfvsdfvwr}  is compatible with $P$ if:\begin{enumerate}
\item\label{kopherthetrgtr} (P-ra: right adjoint) For every $p:X\to Y$ in $P$ the functor $p^{*}$ has a right-adjoint $p_{*}$. \item (P-bc: base change)  For every cartesian square $$ \xymatrix{W\ar[r]^{f}\ar[d]^{q} &X \ar[d]^{p} \\ Z\ar[r]^{g} &Y } $$ in $\LCH$ with $p,q$ in $P$
the square $$\xymatrix{ D(Y)\ar[r]^{g^{*}}\ar[d]^{p^{*}} & D(Z)\ar[d]^{q^{*}} \\D(X) \ar[r]^{f^{*}} & D(W)} $$
is vertically right-adjoinable.
 \item (P-pf: projection formula) For every $p:X\to Y$ in $P$ the canonical map $$p_{*}(-)\otimes_{Y}(-) \to p_{*}(-\otimes_{X} p^{*}(-))$$ is an equivalence.
 \end{enumerate}
\end{ddd}

If $D$ is compatible with $I$ and $P$, then we make the following definition:
 \begin{ddd}
We say that $D$ satisfies the mixed Beck-Chevalley (mBC) condition if 
for every   cartesian square $$ \xymatrix{W\ar[r]^{i}\ar[d]^{q} &X \ar[d]^{p} \\ Z\ar[r]^{j} &Y } $$ in $\LCH$ with $p,q$ in $P$
and $i,j$ in $I$ 
the square  $$\xymatrix{ D(Z)\ar[r]^{j_{!}}\ar[d]^{p^{*}} & D(Y)\ar[d]^{q^{*}} \\D(W) \ar[r]^{i_{!}} & D(X)} $$  given by I-bc
is vertically right-adjoinable.\end{ddd}

Let \begin{equation}\label{bwerfwerwregw}\cB:D\to D':\LCH^{\op}\to \CAlg(\Cat_{\infty}) \end{equation}be a natural transformation between functors.  
 In the following it is useful to write $D(i)$ instead of $i^{*}$.
\begin{ddd}\label{htohopthrtgtrger} \mbox{}\begin{enumerate}\item\label{igh4iuhrtgiow4oigoir} (I($\cB$))
We say that $\cB$ is compatible with $I$ if for every morphism $j:U\to X$ in $I$
the square
$$\xymatrix{D(X)\ar[r]^{\cB_{X}}\ar[d]^{D(j)} &D'(X) \ar[d]^{D'(j)} \\D(U) \ar[r]^{\cB_{U}} & D'(U) } $$
is vertically left-adjoinable.\item  (P($\cB$)) We say that $\cB$ is compatible with $P$ if for every morphism $p:X\to Y$ in $P$ 
the square
$$\xymatrix{D(Y)\ar[r]^{\cB_{Y}}\ar[d]^{D(p)} &D'(Y) \ar[d]^{D'(p)} \\D(X) \ar[r]^{\cB_{X}} & D'(X) } $$
is vertically right-adjoinable. \end{enumerate}
\end{ddd}

\begin{ddd} \mbox{}\label{kojpjrtzjzhrtzh}\begin{enumerate}\item 
A    functor
$D:\LCH\to \CAlg(\Cat_{\infty})$ is called 
a three-functor formalism  on the Nagata context $(\LCH,I,P)$ if it is compatible with $I$, $P$, and satisfies the mBC condition.
\item A morphism between  three-functor formalisms on the Nagata context $(\LCH,I,P)$ is  a natural transformation
  as   in \eqref{bwerfwerwregw} 
which is compatible with $I$ and $P$.
\end{enumerate}
\end{ddd}

\begin{rem}
By \cite[Prop. 3.3.3]{arXiv:2410.13038}, the datum of  a  three-functor formalism  on the Nagata context $(\LCH,I,P)$ 
in the sense of \cref{kojpjrtzjzhrtzh} is equivalent to the datum
of a three-functor formalism  in the sense of  \cite[Def. 3.1.1]{arXiv:2410.13038}
on the pair $(\LCH,\All)$ defined as a lax symmetric monoidal functor $$D:\Corr(\LCH,\All)\to \Cat_{\infty}\ .$$   The three functors are $f^{*}$, $f_{!}$, and $\otimes_{X}$.\hB \end{rem}
 

\subsection{Six-functor formalisms}

In view of \cite[Lem. 3.2.5]{arXiv:2410.13038}, we adopt the following definition.
 \begin{ddd}\label{hertkprtgertgtregertg} \mbox{}
\begin{enumerate}
\item  A three-functor formalism  $D$ on the Nagata context $(\LCH,I,P)$ is    a presentable six-functor formalism  on $(\LCH,I,P)$ if  it
factorizes over  a functor $$D:\LCH^{\op}\to \CAlg(\Pr^{L})$$ and
the right-adjoints $p_{*}$ in P-ra (\cref{hkopertergtrgerg}.\ref{kopherthetrgtr}) are cocontinuous.
 
\item A   morphism between  three-functor formalisms on the Nagata context $(\LCH,I,P)$ is  morphism  between  presentable six-functor formalisms if 
it is implemented by left-adjoint functors.\end{enumerate}
 \end{ddd}
 
 Note that the factorization condition means that $D$ takes values in presentable $\infty$-categories and left adoint functors, and that for every $X$ in $\LCH$ the functor $\otimes_{X}$ preserves colimits in both arguments.

 \begin{rem}
 We will not need the more general definition of a $\Cat_{\infty}$-valued  six-functor formalism from \cite[Def. 3.1.2]{arXiv:2410.13038} which involves more conditions to be checked. All examples of six functor formalisms considered in the present paper are presentable.
 \hB \end{rem}

 Recall  \cref{kophejhthgerthetrh} stating that we have a lax symmetric monoidal functor to $E$-theory contexts
 $$(-)\nCalg:\LCH^{\op}\to \Cat\ .$$ We can now apply \cref{kohpehtregetrg} in in order to make the following definition.
\begin{ddd} \label{okrgpwerwerfewrferwf}We define   the  lax symmetric monoidal functor 
 \begin{equation}\label{vsdfvsfsdr} E(-):=\EE((-)\nCalg):\LCH^{\op}\to \Pr^{L}_{\st}\to \Cat_{\infty}\ . \end{equation} 
\end{ddd}

\begin{theorem}\label{tkphrethergertgertg}
The functor in \eqref{vsdfvsfsdr} is a presentable six-functor formalism.
\end{theorem}
\begin{proof}
By \cref{okpherthrtgetgetg}, the functor in $(-)\nCalg$ in \eqref{rvewiojvofvdfvsdfvsdfv} is a three-functor formalism.
We now observe that it takes values in $E$-theory contexts, and that the left-adjoints $i_{!}$ for $i$ in $I$ and the right-adjoints $p_{*}$ for $p$ in $P$ are morphisms of $E$-theory contexts.  We conclude that
$E(-)$ is a three-functor formalism. Since $E(-)$ takes values in $\Pr^{L}_{\st}$, and the right-adjoints  $\EE(p_{*})$ for $p$ in $P$ and the functors
$A\otimes_{X}-:E(X)\to E(X)$ for all $X$ in $\LCH$ and $A$ in $E(X)$   are cocontinuous,  we conclude that $E$ in \eqref{vsdfvsfsdr} is a presentable six-functor formalism. 
\end{proof}

\subsection{Comparison with sheaves}

In this section, we compare the presentable six-functor formalism $E$ from \cref{tkphrethergertgertg}
with the six-functor formalism $\Shv(-,\EE)$. We first show that $E$ is also a coefficient system. 
Since $\Shv(-,\EE)$ plays the role of  an initial coefficient system with value $\EE$ at the point
(this has been shown in  a precise sense in \cite{arXiv:2507.13537} under slightly stronger assumptions on coefficient systems), we get the natural comparison transformation  \cref{hrtgetgretget}.
The main result of  \cite{buvo}, stated here as \cref{thkoperthertegrtger}, provides a sufficient condition implying that the comparison transformation is an equivalence. We then show 
our main \cref{thkoperthertegrtger} by verifying these conditions for $E$.

\begin{ddd}
A  functor as in \eqref{vfdvsfvsdfvwr}  is presentable and stable if it takes values in $\CAlg(\Pr^{L}_{\st})$.
\end{ddd}

We consider a presentable and stable functor $D$ as in   \eqref{vfdvsfvsdfvwr}, which is compatible with $I$.
The following definition is taken from  \cite[Def. 4.11]{arXiv:2507.13537}.
\begin{ddd}\label{kojghptrertgt}
We say that $D$ has  canonical  descent if it satisfies:
\begin{enumerate} 
\item (zero) \label{hkopzetgerth}$D(\emptyset)\simeq 0$.
\item \label{kohperhgertgert} (localization)  \label{hkopzetgert1h} For every $X$ in $\LCH$, open subset $j:U\to X$ and closed complement $i:Z:=X\setminus U\to X$  we have a recollement
$$ \xymatrix{ D(U)\ar@/^0.8cm/[r]_{\perp}^{j_{!}}\ar@/_0.8cm/[r]^{\perp}_{j_{*}}&\ar[l]_{j^{*}}D(X)\ar@/^0.8cm/[r]_{\perp}^{i^{*}}\ar@/_0.8cm/[r]^{\perp}_{i^{!}}&\ar[l]^{i_{*}}D(Z)}\ .$$
 \item  (open exhaustions)
For every filtered family $(U_{i})_{i\in I}$ in $\Open(X)$ with $U=\bigcup_{i\in I} U_{i}$ we have an equivalence
\begin{equation}\label{giojewroigwerfewrfwer}D(U)\stackrel{\simeq}{\to} \lim_{i\in I^{\op}} D(U_{i})
\end{equation} in $\Pr^{L}_{\st}$ induced by contravariant functoriality for the inclusions.
 \hB
\end{enumerate}
\end{ddd}

\begin{rem} Compatibility of $D$ with $I$  first of all ensures by I-la that $j_{!}$ exists.
The Condition \ref{hkopzetgerth} together with $U\cap Z=\emptyset$  and  I-bc ensures that $i^{*}\circ j_{!}\simeq 0$. The Condition  \ref{hkopzetgert1h} then requires that
 the sequence $$D(U)\stackrel{j_{!}}{\to} D(X)\stackrel{i^{*}}{\to} D(Z)$$ is a fibre and cofibre sequence in $\Pr^{L}_{\st}$. 
 As a consequence, we get a cofibre sequence of endofunctors
 \begin{equation}\label{bwiojoievewrvfds}j_{!}j^{*}\to \id\to i_{*}i^{*}
\end{equation} 
of $D(X)$.
Furthermore, the pair of functors $(j^{*},i^{*})$ is jointly conservative. 
  \hB
\end{rem}

 The following notion is taken  from  \cite[Def. 3.1]{arXiv:2507.13537} (it is a version of the notion with the same name  introduced in \cite{zbMATH07666924} in the context of algebraic geometry).
\begin{ddd}\label{kohperthertgertg}\mbox{}\begin{enumerate}
\item 
A functor $D:\LCH\to \CAlg(\Cat_{\infty})$  is called a coefficient system 
 if it is presentable and stable, compatible with $I$, and satisfies canonical descent. 
 \item A morphism  between coefficient systems
 is a natural transformation of functors 
 which is compatible with $I$ (in the sense of \cref{htohopthrtgtrger}.\ref{igh4iuhrtgiow4oigoir}) and whose components  are morphisms in $\CAlg(\Pr^{L}_{\st})$. 
  \end{enumerate}
\end{ddd}
\begin{theorem}\label{thkoptrhergtrgege}
 The functor  $$E:\LCH^{\op}\to \CAlg(\Cat_{\infty})$$  introduced in  \cref{okrgpwerwerfewrferwf} is a coefficient system.
\end{theorem}
\begin{proof}By \cref{tkphrethergertgertg},
we have a six-functor formalism
 $$E:\LCH^{\op}\to \Pr_{\st}^{L}\ .$$ 
 It is presentable and stable and satisfies the axioms of canonical descent by \cref{kohpertghertgtrege} and \cref{okphterthertg}.
  \end{proof}

In \cite{buvo}, we have shown that for every 
  coefficient system $D$ there is a canonical natural $D(\pt)$-linear transformation of coefficient systems
\begin{equation}\label{bfgkbdfgbdfgbert}\cB:\Shv(-,D(\pt))\to D(-)\ .
\end{equation}

\begin{rem} Let $D:\LCH^{\op}\to \CAlg(\Cat_{\infty})$ be a six-functor formalism on the Nagata context $(\LCH,I,P)$.   For 
  $X$  in $ \LCH$ and    $U$ in $\Open(X)$ we consider  the inclusion $j_{U\to X}:U\to X$     and   the projection    $p_{U}:U\to *$.
If $D$ is,  in addition, a coefficient system, then the association
$$\Open(X)\ni U\mapsto j_{U\to X,!} p_{U}^{*}1_{D(\pt)}\in D(X)$$ defines a cosheaf
in $\CoShv(X,D(X))$. The colimit-preserving $D(\pt)$-linear functor $$\cB_{X}:\Shv(X,D(\pt))\to D(X) $$ corresponds to this cosheaf under the equivalence
$$\Fun_{D(\pt)}^{\colim}(\Shv(X,D(\pt)),D(X))\simeq \CoShv(X,D(X))\ .$$ \hB \end{rem}

Note that $\EE\simeq E(\pt)$.
\begin{kor}\label{hrtgetgretget}
We have a natural transformation \begin{equation}\label{gerwfergwerg}
\cB:\Shv(-,\EE)\to E:\LCH^{\op}\to \CAlg(\Pr^{L}_{\st})\ .
\end{equation}
\end{kor}

\begin{ddd}  \label{okhtperhertgrtgertgerg}
For $X$ in $\LCH$ and $U$ in $\Open(X)$ we  define the evaluation
\begin{equation}\label{hztzhtzhtzhtzrtzhj}\ev_{U}:=p_{U,*}j_{U\to X}^{*}:D(X)\to D(\pt)\ . \end{equation}
\end{ddd}
 \begin{ddd}\label{lkpohejztjrtz} \mbox{}
 \begin{enumerate} \item
 We say that $D$ is section-determined on $X$ in $\LCH$ if the collection $(\ev_{U})_{U\in \Open(X)}$ of evaluation   maps  is jointly conservative. \item  We say that $D$ is section-determined if it is section-determined on every $X$ in $\LCH$.
 \end{enumerate}
\end{ddd}

For a  six-functor formalism $D$  which is also a coefficient system $D$,   we define, for every $X$ in $\LCH$, a functor  \begin{equation}\label{tgretgrtgeertt}
s_{X}:D(X)\to \CoShv(X,D(\pt))\ , \quad s_{X}(A):U\mapsto p_{U,!} j_{U\to X}^{*}A\ .
\end{equation}
As in the proof of \cref{gweoirgjoergeferggr}, the cosheaf conditions are   implied by the assumption that $D$ is a coefficient system.
  Note that, in contrast to $\ev_{U}(A)=p_{U,*}j_{U\to X}^{*}A $, the definition of
  $s_{X}(A)(U) $ involves $p_{U,!}$. 
 \begin{prop}\label{hkopeththregte}
 Assume that $D$ is a  six-functor  formalism  on $(\LCH,I,P)$   which is also a coefficient system $D$.
 Then $D$ is section-determined if and only if, for every $X$ in $\LCH$ the functor $s_{X}$ in \eqref{tgretgrtgeertt}  is conservative.,
 \end{prop}
 \begin{proof}
   \begin{lem}\label{hrrtgrtgrtege} Let $D$ be a six-functor formalism on $(\LCH,I,P)$ which is also a coefficient system.
 For every $X$ in $\LCH$, $U$ in $\Open(X)$, and $A$ in $D(X)$ we have an equivalence
 \begin{equation}\label{gwrgerferffw}
\ev_{U}(A)\simeq \lim_{K\subset U} \Cofib(s_{X}(A)(U\setminus K)\to s_{X}(A)(U))\ ,
\end{equation}
 where $K$ runs over the compact subsets of $U$.
 \end{lem}
 \begin{proof}
 For every $B$ in $D(\pt)$ we have, by definition of $\ev_{U}$,
 $$\map_{D(\pt)}(B,\ev_{U}(A))\simeq \map_{D(\pt)}(B,p_{U,*}j_{U\to X}^{*}A)\ .$$
 We now use the defining adjunctions of the operations  to get   $$\map_{D(\pt)}(B,p_{U,*}j_{U\to X}^{*}A)\simeq 
 \map_{D(U)}(p_{U}^{*}B, j_{U\to X}^{*}A)\simeq  \map_{D(X)}(j_{U\to X,!}p_{U}^{*}B,  A)\ .$$
 We now use that $U=\bigcup_{V\Subset U} V$ and 
 the fact that for a coefficient system, we have $$\colim_{V\Subset U} j_{V\to X,!} j_{V\to X}^{*}\simeq j_{U\to X,!}j_{U\to X}^{*}$$ in order to
get 
 $$\map_{D(X)}(j_{U\to X,!}p_{U}^{*}B,  A)\simeq 
 \map_{D(X)}( \colim_{V\Subset U} j_{V\to X,!}p_{V}^{*}B,  A)\simeq 
 \lim_{V\Subset U} \map_{D(V)}(  j_{V\to X}^{*}p_{X}^{*}B,   j_{V\to X}^{*} A)\ .$$
  Using that $V\Subset U$ implies  $V\subseteq \bar V\subseteq U$ and that $\bar V$ is compact,
   we can rewrite the limit as 
$$ \lim_{V\Subset U} \map_{D(V)}(  j_{V\to X}^{*}p_{X}^{*}B,   j_{V\to X}^{*} A)
 \simeq \lim_{K\subseteq U} \map_{D(K)}(  j_{K\to X}^{*}p_{X}^{*}B,   j_{K\to X}^{*} A)\ .$$
Using that $K\to X$ is a closed inclusion   and $p_{K}:K\to *$
is proper, we get
 $$\lim_{K\subseteq U} \map_{D(K)}(  j_{K\to X}^{*}p_{X}^{*}B,   j_{K\to X}^{*} A)\simeq
  \lim_{K\subseteq U} \map_{D(\pt)}( B,   p_{K,!} j_{K\to X}^{*} A)\ .$$
  Using the sequence \eqref{twuiotjioergergw} for $U\setminus K\to U\leftarrow K$
  we get
  $$p_{K,!} j_{K\to X}^{*} A\simeq \Cofib(s(A)(U\setminus K)\to s(A)(U))\ .$$
Inserting this  and combining all the equivalences we get 
$$\map_{D(\pt)}(B,\ev_{U}(A))\simeq \map_{D(\pt)}(B, \lim_{K\subseteq U} \Cofib(s_{X}(A)(U\setminus K)\to s_{X}(A)(U)))\ .$$
This implies the assertion.
   \end{proof}
   We have a Verdier duality equivalence
   $$\cV:\Shv(X,D(\pt)) \leftrightarrows \CoShv(X,D(\pt)):\cV^{-1}\ .$$
   The right-hand side of the formula  \eqref{gwrgerferffw}
   is precisely the formula
   for $\cV^{-1}(s_{X}(A))(U)$.
   
   Assume that  $D$ is section-determined.
   Let $f:A\to B$ be a morphism in $D(X)$. If $s_{X}(f)$ is an equivalence,
   then by \cref{hrrtgrtgrtege}, we conclude that $\ev_{U}(f)$ is an equivalence for every $U$ in $\Open(X)$.
   By the assumption on $D$ we conclude that $f$ is an equivalence. 
   This shows that $s_{X}$ is conservative.
    
   Assume now that $s_{X}$ is conservative. If $\ev_{U}(f)$ is an 
   equivalence for every $U$ in $\Open(X)$, then $\cV^{-1}( s_{X}(f))$ is an equivalence. Hence, also $s_{X}(f)$ is an equivalence.
   We conclude that $f$ is an equivalence by the assumption on $D$.
   This shows that $D$ is section-determined.
   \end{proof}

 \begin{theorem}\label{trkopherthrgertg}
The functor $E$  in \eqref{vsdfvsfsdr} is section-determined.
\end{theorem}
\begin{proof} By \cref{thkoptrhergtrgege}
  and  \cref{tkphrethergertgertg}, the $E$-theory functor is a six-functor formalism on $(\LCH,I,P)$ and a coefficient system.  In \cref{herthetrgertgetrg} we have shown that $s_{X}:E(X)\to \CoShv(X,\EE)$ is conservative for every $X$ in $\LCH$. We now apply \cref{hkopeththregte} in order to conclude that $E$ is section-determined.
\end{proof}

We consider a functor
$D:\LCH^{\op}\to \CAlg(\Pr^{L})$. Then, for every morphism $f$ in $\LCH$ we have an adjunction $f^{*}\dashv f_{*}$.
  \begin{ddd}\label{okphertgertgerg}
We   define the  cohomology functor associated with  $D$ by
$$\Gamma^{D}:\LCH^{\op}\to D(\pt)\ , \quad X\mapsto p_{X,*} p_{X}^{*}1\ .$$\end{ddd}
Here 
 $1$ is  the tensor unit  in $D(\pt)$.

Let  $H:\LCH^{\op} \to \cC $ be  some functor.    
\begin{ddd}\label{iugwoierwerfwefwef}\mbox{}
 We say that $H$ is finitary\footnote{In \cite{NKP} this property is called "profinite descent".}     if, for every cofiltered system $(X_{i})_{i\in I}$   in $\CH$ 
 with $X:=\lim_{i\in I}X_{i}$, we have
 an equivalence
 $$  \colim_{i\in I^{\op}} H(X_{i})\stackrel{ \simeq}{\to} H(X)  \ .$$ 
 \end{ddd}

 \begin{prop}\label{tkohperthretgertg}
 The  cohomology functor associated with  $E$ by is finitary.
 \end{prop}
\begin{proof}
We have a functor
$$\Gamma^{(-)\nCalg}:\CH^{\op}\to \nCalg\ , \quad  X\mapsto p_{X,*}p_{X}^{*}1\ .$$
Explicitly, this functor is given by $X\mapsto C(X)$. We conclude that the restriction
$$\Gamma^{E}_{|\CH^{\op}}:\CH^{\op }\to \EE$$  of the cohomology functor to compact Hausdorff spaces is given by 
$$X\mapsto \ee(C(X)) , \quad \left(f:X\to Y)\mapsto \ee(f^{*}):\ee(C(Y))\to \ee(C(X))\right)\ .$$
For every cofiltered system $(X_{i})_{i\in I}$   in $\CH$ 
 with $X:=\lim_{i\in I}X_{i}$, we have, by Gelfand duality, 
 $$\colim_{i\in I^{\op}}C(X_{i}) \cong C(X)$$
 in $\nCalg$.
 Since  the functor $\ee:\nCalg\to \EE$ preserves filtered colimits, we conclude that
 $$\colim_{i\in I^{\op}}\ee(C(X_{i}) )\simeq  \ee(C(X)) \ .$$
  This shows that $\Gamma^{E}$ is finitary.
  \end{proof}

The following theorem is shown in \cite{buvo}.
  \begin{theorem} \label{thkoperthertegrtger} Assume that 
 $D$ is a presentable and stable six-functor formalism on the Nagata context $(\LCH,I,P)$ which is a coefficient system, section-determined, and is such that $D(\pt)$ is     dualizable and its 
 associated cohomology $\Gamma^{D}$  is finitary. 
 Then the transformation in \eqref{bfgkbdfgbdfgbert}
 is an equivalence of six-functor formalisms. 
  \end{theorem}

  \begin{kor}\label{trkohjperthretgertgertg}
  The transformation  \eqref{gerwfergwerg} is an equivalence of six-functor formalisms
  $$\cB:\Shv(-,\EE)\stackrel{\simeq}{\to} E(-):\LCH^{\op}\to \Pr^{L}_{\st}\ .$$
  \end{kor}
  \begin{proof}
  The assumptions of \eqref{thkoperthertegrtger} are satisfied:
  \begin{enumerate}
  \item $E$ is stable and presentable by construction.
  \item $E$ is a coefficient system by \cref{thkoptrhergtrgege}.
  \item $E$ is section-determined by \cref{trkopherthrgertg}.
  \item $E(\pt)\simeq \EE$ is dualizable by \cite{budu}.
  \item $\Gamma^{E}$ is finitary by \cref{tkohperthretgertg}.
  \end{enumerate}

  \end{proof}

 \section{Technical lemmas}
 
 The following algebraic and categorical properties of $C^*$-algebras and their maximal tensor products are well-known standard results. We sketch the arguments  for the reader's convenience.
 We state the results precisely in the form in which they are used  in \cref{gergwergwerioug9erwgwregw}. 
 A good reference for the exactness of the maximal tensor product is \cite[Prop. 3.7.1]{zbMATH05256855}.
 
 \begin{lem}\label{hpzhjetophherth}
 Filtered colimits preserve exact sequences.
 \end{lem}
 \begin{proof}
 Let $$(0\to I_{i}\to A_{i}\to Q_{i}\to 0)_{i\in I}$$ be a filtered system of exact sequences in $\nCalg$.
 Then the assertion is that
 $$0\to \colim_{i\in I}I_{i}\to \colim_{i\in I}A_{i}\to \colim_{i\in I} Q_{i}\to 0$$
 is again exact. This is a well-known fact from $C^{*}$-algebra theory. See e.g. \cite[Lem. 7.20]{KKG}
 for an argument that even works for $C^{*}$-categories.
 \end{proof}

 If $C'\to C$ is a homomorphism of $C^{*}$-algebra,  we let $\bar C'$ denote the image of
the homomorphism and $\langle \bar C'\rangle$ be the ideal generated by $\bar C'$.
Recall that $\otimes$ denotes the maximal tensor product in $\nCalg$.
 \begin{lem}
  \label{ihrfiwuoghewergwreg}
  If $ A'\to A$ and $B' \to B$  are homomorphisms of $C^{*}$-algebras,
  then  we have
 a canonical isomorphism of ideals  $$ \langle \overline{A'\otimes B'} \rangle\cong \langle\bar A' \rangle \otimes\langle \bar B'\rangle$$ in $A\otimes B$.
  \end{lem}
 \begin{proof}
 It follows from the exactness of $\otimes$ that it preserves ideal inclusions in both arguments.
 Consequently  $\langle\bar A' \rangle \otimes\langle \bar B'\rangle$ is naturally an ideal in $A\otimes B$.
 The ideal $ \langle \overline{A'\otimes  B'} \rangle$ is the smallest ideal
 of $A\otimes B$ that contains the images $\overline{ a\otimes b}$ 
 of $a\otimes b$ in $A'\otimes^{\alg} B'$ for all
 $a$ in $A'$ and $b$ in $B'$. Such an  element is also the image  in $A\otimes  B$ of 
 $\bar a\otimes \bar b$ in $A\otimes^{\alg}B$. 
 This implies
  $$ \langle \overline{A'\otimes B'} \rangle \subseteq  \langle\bar A' \rangle \otimes \langle \bar B'\rangle\ .$$
 Every element of $ \langle\bar A' \rangle \otimes \langle \bar B'\rangle$ be be approximated by sums of elements of the form
 $x \bar ax'\otimes y \bar b y'$ for $a$ in $A'$, $B$ in $B'$, $x,x'$ in $A$ and $y,y'$ in $B$. 
 We consider the elements
 $(u\otimes y)(x\otimes v)(\overline{a\otimes b}) (x'\otimes v)(u\otimes y')$ in
 $\langle \overline{A'\otimes B'} \rangle $.
Letting $u$ and $v$ run over approximate units in $A$ and $B$, respectively,
we see that these elements approximate  $x \bar ax'\otimes y \bar b y'$.
We conclude that $$\langle\bar A' \rangle \otimes \langle \bar B'\rangle\subseteq \langle \overline{A'\otimes  B'} \rangle\ .$$
  \end{proof}
 
 \begin{lem} \label{kophertgertgertg}If
 $0\to A\to B\to C\to 0$ is an exact sequence in $\nCalg$, $J$ is an ideal in $B$
 and $I:=A\cap J$, then the canonical map $J/I\to C$ is the inclusion of an ideal.
\end{lem}
\begin{proof}
We have a map of exact sequences 
$$\xymatrix{0\ar[r]  &I \ar[r]\ar[d] &J\ar[r]\ar[d] &J/I\ar[r]\ar[d] &0 \\ 0\ar[r] &A\ar[r]&B\ar[r] &C\ar[r]   &0 } \ .$$
By a diagram chase, we see that the right vertical map is injective.
If $[j]$ is in $J/I$ and $c$ is in $C$, then $c[j]=[cj]\in J/I$. Hence $J/I$ is an ideal in $C$.
\end{proof}
 
For the following statements, we  let  $A,B$ be in $\nCalg$ and $I,J$ be  ideals in $A$, and let $K,L$ be ideals in $B$.

 \begin{lem}\label{kophertgertgrtege}
 The canonical map $I\otimes K\to A\otimes B$ is an inclusion.
 \end{lem}
\begin{proof}
The functor $-\otimes B$ is exact. It sends the exact sequence
$$0\to I\to A\to A/I\to 0$$ to $$0\to I\otimes B\to A\otimes B\to A/I\otimes B\to 0\ .$$

We can apply this also to the functor $I\otimes -$.
Hence
$$I\otimes K\to I\otimes B\to A\otimes B$$
is a composition of inclusions.

\end{proof}

\begin{lem}\label{kophertgertge}
The natural inclusion $(I\cap J)\otimes B\to (I\otimes B)\cap (J\otimes B)$ is an equality.
\end{lem}
\begin{proof}
We have a cartesian square 
  $$\xymatrix{I\cap J\ar[r]\ar[d] &I \ar[d] \\ J\ar[r] & I+J}$$
   of ideals in $A$, meaning in particular that $I+J$ is again a closed ideal.
  The canonical map of quotients $I/(I\cap J)\to (I+J)/J$ is an isomorphism.
   
  We get a square of ideals 
  $$\xymatrix{(I\cap J)\otimes B\ar[r]\ar[d] &I\otimes B \ar[d] \\ J\otimes B\ar[r] & (I+J)\otimes B}\ .$$
  All maps remain injective, and the canonical map of quotients
  $$I\otimes B/((I\cap J)\otimes B) \to ((I+J)\otimes B)/(J\otimes B)$$ is still an isomorphism.
  This implies that the square is still cartesian which implies the assertion.
  \end{proof}

 \begin{lem}\label{kohpertherggetrg}
 The natural inclusion $$(I\cap J)\otimes (K\cap L)\to (I\otimes K)\cap (J\otimes L)$$ of ideals in $A\otimes B$ is an equality. 
 \end{lem}
 \begin{proof}
 We consider the diagram
 $$\xymatrix{(I\cap J)\otimes (K\cap L)\ar@{^{(}->}[r] \ar@{^{(}->}[d] & A\otimes (K\cap L)\ar[r]\ar@{^{(}->}[d]&A/(I\cap J)\otimes (K\cap L)\ar@{^{(}->}[d] \\ (I\cap J)\otimes B \ar@{^{(}->}[r]&A\otimes B\ar[r]&A/(I\cap J)\otimes B\\\ar@/^2cm/@{^{(}..>}[uu] A\otimes (K\cap L)\cap (I\cap J)\otimes B \ar[uur]\ar[u]& &}\ .$$
 The two horizontal sequences are exact.  By a diagram chase, we get the dotted arrow.
 We now show that
 $$(I\otimes K)\cap (J\otimes L)\subseteq A\otimes (K\cap L)\cap (I\cap J)\otimes B\ .$$
 We have inclusions
 $$(I\otimes K)\cap (J\otimes L)\subseteq (I\otimes B)\cap (J\otimes B)=(I\cap J)\otimes B\ ,$$
 where we use \cref{kophertgertge} for the second equality.
 Analogously $$
 (I\otimes K)\cap (J\otimes L)\subseteq A\otimes (K\cap L)\ .$$
 This implies
 $$(I\otimes K)\cap (J\otimes L)\subseteq (I\cap J)\otimes (K\cap L)$$ and hence the assertion.
 \end{proof}

 \begin{lem}\label{jgwoiergwergwreg}
 Let  $(I_{i})_{i}$ be a family of ideals  in a $C^{*}$-algebra $A$, and $B$ be a further  $C^{*}$-algebra. Then  
 we have
 $$\overline{\sum_{i\in I} I_{i}}\otimes B= \overline{\sum_{i\in I} I_{i}\otimes B}$$ inside
 $A\otimes B$. 
 \end{lem}
 \begin{proof}We use that $-\otimes B$ preserves ideal inclusions, finite sums of ideals,  and filtered colimits. 
\end{proof}

\begin{lem}\label{khoprtehegrtege9}
Let $A$ be a $C^{*}$-algebra and let $I,J,K$ be ideals in $A$.
Then  $(I\cap J)/(K\cap I\cap J)= I/(K\cap I)\cap J/(K\cap J) $ in $A/K$.
\end{lem}
\begin{proof}
We have a canonical inclusion map
$(I\cap J)/(K\cap I\cap J)\to  I/(K\cap I)\cap J/(K\cap J) $. We must show that it is surjective.
Let $q$ be in $ I/(K\cap I)\cap J/(K\cap J) $. Then there exists 
$a$ in $I$ and $b$ in $J$ such that $[a]=q=[b]$.
Hence $ a+d=b$ for some $d$ in $K\cap J$. It follows $a=b-d$ which implies $d\in K\cap I$.
Hence $d\in K\cap I\cap J$ and $[b]=[a+d]=q$. Furthermore, $b\in I\cap J$.
The class of $b$ is the desired preimage of $q$.

\end{proof}

\begin{lem} \label{ojgowepregwfw}
Let $I,J,K$ be ideals in $A$.
Then 
$$(I\cap J)+K =(I+K)\cap (J+ K)\ .$$
\end{lem}
\begin{proof} The lemma says that 
the ideals in a $C^{*}$-algebra  constitute a distributive lattice. 
This follows from the identification of the poset $I(A)$ with $\Open(\Prim(A))$ (see \cite[Sec. 2.9]{zbMATH04073740}).
First observe that 
 $$(I\cap J)+K \subseteq (I+K)\cap (J+ K)\ .$$
Assume now that
$$i+k=j+k'\in (I+K)\cap (J+ K)\ .$$
Let $(u)$ be an approximate unit of $I+K$.
We have $\lim_{u} u(i+k)=i+k$. On the other hand,
$$u(i+k)=u(j+k')\in (I+K)(J+K)=IJ+K\subseteq (I\cap J)+K\ .$$
Hence $i+k\in  (I\cap J)+K$.

\end{proof}

\begin{lem}\label{kophetrhrtgtrge}
If $0\to I\to A\to Q\to 0$ and $0\to J\to B\to P$ are exact sequences, then
$$0\to I\otimes B+A\otimes J\to A\otimes B\to Q\otimes P\to 0$$ is exact.
\end{lem}
\begin{proof}
This follows from the exactness of the maximal tensor product.
\end{proof}

    \bibliographystyle{alpha}
\bibliography{forschung2021}

\end{document}